\documentclass[11pt]{scrartcl}

\newcounter{savfert}
\usepackage[utf8]{inputenc}
\RequirePackage[T1]{fontenc}
\usepackage{lmodern}
\usepackage{alphabeta}

\usepackage{config}
\usepackage{titling}
\usepackage{ifthen}

\usepackage[letterpaper, top=1in, bottom=1in, left=1in, right=1.0in, includefoot]{geometry}

\usetikzlibrary{arrows.meta}

\makeatletter

\pgfdeclarearrow{
  name = ipe _linear,
  defaults = {
    length = +1bp,
    width  = +.666bp,
    line width = +0pt 1,
  },
  setup code = {
    \pgfarrowssetbackend{0pt}
    \pgfarrowssetvisualbackend{
      \pgfarrowlength\advance\pgf@x by-.5\pgfarrowlinewidth}
    \pgfarrowssetlineend{\pgfarrowlength}
    \ifpgfarrowreversed
      \pgfarrowssetlineend{\pgfarrowlength\advance\pgf@x by-.5\pgfarrowlinewidth}
    \fi
    \pgfarrowssettipend{\pgfarrowlength}
    \pgfarrowshullpoint{\pgfarrowlength}{0pt}
    \pgfarrowsupperhullpoint{0pt}{.5\pgfarrowwidth}
    \pgfarrowssavethe\pgfarrowlinewidth
    \pgfarrowssavethe\pgfarrowlength
    \pgfarrowssavethe\pgfarrowwidth
  },
  drawing code = {
    \pgfsetdash{}{+0pt}
    \ifdim\pgfarrowlinewidth=\pgflinewidth\else\pgfsetlinewidth{+\pgfarrowlinewidth}\fi
    \pgfpathmoveto{\pgfqpoint{0pt}{.5\pgfarrowwidth}}
    \pgfpathlineto{\pgfqpoint{\pgfarrowlength}{0pt}}
    \pgfpathlineto{\pgfqpoint{0pt}{-.5\pgfarrowwidth}}
    \pgfusepathqstroke
  },
  parameters = {
    \the\pgfarrowlinewidth,%
    \the\pgfarrowlength,%
    \the\pgfarrowwidth,%
  },
}

\pgfdeclarearrow{
  name = ipe _pointed,
  defaults = {
    length = +1bp,
    width  = +.666bp,
    inset  = +.2bp,
    line width = +0pt 1,
  },
  setup code = {
    \pgfarrowssetbackend{0pt}
    \pgfarrowssetvisualbackend{\pgfarrowinset}
    \pgfarrowssetlineend{\pgfarrowinset}
    \ifpgfarrowreversed
      \pgfarrowssetlineend{\pgfarrowlength}
    \fi
    \pgfarrowssettipend{\pgfarrowlength}
    \pgfarrowshullpoint{\pgfarrowlength}{0pt}
    \pgfarrowsupperhullpoint{0pt}{.5\pgfarrowwidth}
    \pgfarrowshullpoint{\pgfarrowinset}{0pt}
    \pgfarrowssavethe\pgfarrowinset
    \pgfarrowssavethe\pgfarrowlinewidth
    \pgfarrowssavethe\pgfarrowlength
    \pgfarrowssavethe\pgfarrowwidth
  },
  drawing code = {
    \pgfsetdash{}{+0pt}
    \ifdim\pgfarrowlinewidth=\pgflinewidth\else\pgfsetlinewidth{+\pgfarrowlinewidth}\fi
    \pgfpathmoveto{\pgfqpoint{\pgfarrowlength}{0pt}}
    \pgfpathlineto{\pgfqpoint{0pt}{.5\pgfarrowwidth}}
    \pgfpathlineto{\pgfqpoint{\pgfarrowinset}{0pt}}
    \pgfpathlineto{\pgfqpoint{0pt}{-.5\pgfarrowwidth}}
    \pgfpathclose
    \ifpgfarrowopen
      \pgfusepathqstroke
    \else
      \ifdim\pgfarrowlinewidth>0pt\pgfusepathqfillstroke\else\pgfusepathqfill\fi
    \fi
  },
  parameters = {
    \the\pgfarrowlinewidth,%
    \the\pgfarrowlength,%
    \the\pgfarrowwidth,%
    \the\pgfarrowinset,%
    \ifpgfarrowopen o\fi%
  },
}

\newdimen\ipeminipagewidth

\tikzstyle{ipe import} = [
  x=1bp, y=1bp,
  ipe node stretch/.store in=\ipenodestretch,
  ipe stretch normal/.style={ipe node stretch=1},
  ipe stretch normal,
  ipe node/.style={
    anchor=base west, inner sep=0, outer sep=0, scale=\ipenodestretch
  },
  ipe mark scale/.store in=\ipemarkscale,
  ipe mark tiny/.style={ipe mark scale=1.1},
  ipe mark small/.style={ipe mark scale=2},
  ipe mark normal/.style={ipe mark scale=3},
  ipe mark large/.style={ipe mark scale=5},
  ipe mark normal, % Set default
  ipe circle/.pic={
    \draw[line width=0.2*\ipemarkscale]
      (0,0) circle[radius=0.5*\ipemarkscale];
    \coordinate () at (0,0);
  },
  ipe disk/.pic={
    \fill (0,0) circle[radius=0.6*\ipemarkscale];
    \coordinate () at (0,0);
  },
  ipe fdisk/.pic={
    \filldraw[line width=0.2*\ipemarkscale]
      (0,0) circle[radius=0.5*\ipemarkscale];
    \coordinate () at (0,0);
  },
  ipe box/.pic={
    \draw[line width=0.2*\ipemarkscale, line join=miter]
      (-.5*\ipemarkscale,-.5*\ipemarkscale) rectangle
      ( .5*\ipemarkscale, .5*\ipemarkscale);
    \coordinate () at (0,0);
  },
  ipe square/.pic={
    \fill
      (-.6*\ipemarkscale,-.6*\ipemarkscale) rectangle
      ( .6*\ipemarkscale, .6*\ipemarkscale);
    \coordinate () at (0,0);
  },
  ipe fsquare/.pic={
    \filldraw[line width=0.2*\ipemarkscale, line join=miter]
      (-.5*\ipemarkscale,-.5*\ipemarkscale) rectangle
      ( .5*\ipemarkscale, .5*\ipemarkscale);
    \coordinate () at (0,0);
  },
  ipe cross/.pic={
    \draw[line width=0.2*\ipemarkscale, line cap=butt]
      (-.5*\ipemarkscale,-.5*\ipemarkscale) --
      ( .5*\ipemarkscale, .5*\ipemarkscale)
      (-.5*\ipemarkscale, .5*\ipemarkscale) --
      ( .5*\ipemarkscale,-.5*\ipemarkscale);
    \coordinate () at (0,0);
  },
  /pgf/arrow keys/.cd,
  ipe arrow normal/.style={scale=1},
  ipe arrow tiny/.style={scale=.4},
  ipe arrow small/.style={scale=.7},
  ipe arrow large/.style={scale=1.4},
  ipe arrow normal,
  /tikz/.cd,
  ipe arrows/.style={
    ipe normal/.tip={
      ipe _pointed[length=1bp, width=.666bp, inset=0bp,
                   quick, ipe arrow normal]},
    ipe pointed/.tip={
      ipe _pointed[length=1bp, width=.666bp, inset=0.2bp,
                   quick, ipe arrow normal]},
    ipe linear/.tip={
      ipe _linear[length = 1bp, width=.666bp,
                  ipe arrow normal, quick]},
    ipe fnormal/.tip={ipe normal[fill=white]},
    ipe fpointed/.tip={ipe pointed[fill=white]},
    ipe double/.tip={ipe normal[] ipe normal},
    ipe fdouble/.tip={ipe fnormal[] ipe fnormal},
    ipe arc/.tip={ipe normal},
    ipe farc/.tip={ipe fnormal},
    ipe ptarc/.tip={ipe pointed},
    ipe fptarc/.tip={ipe fpointed},
  },
  ipe arrows, % Set default sizes
]

\tikzset{
  rgb color/.code args={#1=#2}{%
    \definecolor{tempcolor-#1}{rgb}{#2}%
    \tikzset{#1=tempcolor-#1}%
  },
}

\makeatother

\newcommand{\cupall}{\pmb{\pmb{\bigcup}}}

\newcommand{\yes}{{\textsf{yes}}}

\usepackage{tikz}

\newcommand{\remove}[1]{}

\newcommand{\makefast}[2]{\ifthenelse{\value{savfert}=1}{#2}{#1}}

\newcommand{\hh}{\end{document}}

\usepackage[inline]{enumitem}

\newcommand{\classP}{\#\textsf{P}}
\newcommand{\pmprob}{\#\textsc{TotalWeight}}
\newcommand{\svmg}{\textrm{SVMG}}

\setlist[itemize]{topsep=0pt,partopsep=0pt,itemsep=0pt,parsep=0pt}
\setlist[itemize,1]{label={\small\textbullet}}
\setlist[itemize,2]{label={\tiny\textbullet}}
\setlist[itemize,3]{label=$\cdot$}
\setlist[enumerate]{topsep=0pt,partopsep=0pt,itemsep=0pt,parsep=0pt}
\setlist[enumerate,1]{label=\roman*)}
\setlist[enumerate,2]{label=\alph*)}
\setlist[enumerate,3]{label=\arabic*)}

\hypersetup{
colorlinks=true,
linkcolor=AO!65!black,
citecolor=AO!65!black,
urlcolor=AppleGreen!65!black,
bookmarksopen=true,
bookmarksnumbered,
bookmarksopenlevel=2,
bookmarksdepth=3
}

\title{Excluding Single-Crossing Matching Minors in Bipartite Graphs}

\author{\bigskip\large 
Archontia C.\@ Giannopoulou\thanks{Department of Informatics and Telecommunications, National and Kapodistrian University of Athens.}
\and
Dimitrios M. Thilikos\thanks{LIRMM, Univ de Montpellier, CNRS, Montpellier, France.}~$^{,}$\thanks{Supported by the ANR projects DEMOGRAPH (ANR-16-CE40-0028), ESIGMA (ANR-17-CE23-0010), and the French-German Collaboration ANR/DFG Project UTMA (ANR-20-CE92-0027).}\and\large 
\and\large 
Sebastian Wiederrecht$^{2,}$\thanks{Discrete Mathematics Group, Institute for Basic Science, Daejeon, South Korea.}~$^{,}$\thanks{The research of Sebastian Wiederrecht was supported by the ANR project ESIGMA (ANR-17-CE23-0010) and the Institute for Basic Science (IBS-R029-C1).}}

\date{}

\begin{document}

\maketitle

\begin{abstract}
\noindent By a seminal result of Valiant, computing the permanent of $(0,1)$-matrices is, in general,  $\#\mathsf{P}$-hard.
In 1913 P\'olya asked for which $(0,1)$-matrices $A$ it is possible to change some signs such that the permanent of $A$ equals the determinant of the resulting matrix.
In 1975, Little showed these matrices to be exactly the biadjacency matrices of bipartite graphs excluding $K_{3,3}$ as a \textsl{matching minor}.
This was turned into a polynomial time algorithm by McCuaig, Robertson, Seymour, and Thomas in 1999.
However, the relation between the exclusion of some matching minor in a bipartite graph and the tractability of the permanent extends beyond $K_{3,3}.$
Recently it was shown that the exclusion of any planar bipartite graph as a matching minor yields a class of bipartite graphs on which the \textsl{permanent} of the corresponding $(0,1)$-matrices can be computed efficiently.
In this paper we unify the two results above into a single, more general result in the style of the celebrated structure theorem for single-crossing-minor-free graphs.
We identify a class of bipartite graphs strictly generalising planar bipartite graphs and $K_{3,3}$ which includes infinitely many non-Pfaffian graphs.
The exclusion of any member of this class as a matching minor yields a structure that allows for the efficient evaluation of the permanent.
Moreover, we show that the evaluation of the permanent remains $\#\mathsf{P}$-hard on bipartite graphs which exclude $K_{5,5}$ as a matching minor.
This establishes a first computational lower bound for the problem of counting perfect matchings on matching minor closed classes.
\end{abstract}

\newpage
\tableofcontents
\newpage

%\pagenumbering{arabic}
%

\section{Introduction}\label{sec_introduction}

The \emph{permanent} of an $(n\times n)$-matrix $A=(a_{i,j})$ is defined as
\begin{eqnarray}
	\mathsf{perm}(A)=\sum_{\sigma\in S_{n}}\prod_{i=1}^{n}a_{i,\sigma(i)},\label{def_perm}
\end{eqnarray}
where $S_{n}$ is the set of all possible permutations of the set $[n]=\{1,\ldots,n\}.$
Permanents of $(0,1)$-matrices arise in many combinatorial questions related to counting, especially  when it comes to the counting of permutations with certain restrictions \cite{percus2012combinatorial}.
A particularly well-studied application of the permanent of $(0,1)$-matrices is the counting of \emph{perfect matchings} in graphs.
This problem has further applications, under the name ``\emph{dimer problem}'', in theoretical physics \cite{kasteleyn1967graph,Kasteleyn61thest,TemperleyF61dimer,Kasteleyn1963dimer}.
The deep connection of the permanent to permutations and the number of perfect matchings in graphs has lead Valiant to the definition of the complexity class $\#\mathsf{P}$ for counting problems \cite{Valiant79theco}. 
He proved that the computation of the permanent of a $(0,1)$-matrix is $\#\mathsf{P}$-hard.

\paragraph{Excluding a matching minor.}
Among the problems surrounding the permanent one may identify \emph{P\'olya's Permanent Problem} as one of the most prominent.
Due to the close similarity between the definition of the permanent and the definition of the determinant of a matrix, P\'olya \cite{polya1913aufgabe} asked for which $(0,1)$-matrices $A$ it would be possible to reduce the computation of the permanent to the computation of the determinant of a similar matrix obtained from $A$ by changing some of its signs.
This property was later found to have a graph theoretical  formulation in terms of \emph{Pfaffian orientations} of bipartite graphs \cite{kasteleyn1967graph,little1975characterization,mccuaig2004polya}.
Indeed, Little \cite{little1975characterization} observed that the applicability of P\'olya's approach is preserved under certain reduction rules for $(0,1)$-matrices which have natural graph theoretic counterparts that eventually gave rise to the notion of \emph{matching minors}.
Roughly speaking, the matching minor relation is a ``matching respecting'' restriction of the minor relation.
That is, the deletion of edges and vertices and the contraction of edges are restricted to preserve certain matching theoretic properties of the graph.

In graph theoretic terms, Little proved that a bipartite graph has a Pfaffian orientation if and only if it excludes $K_{3,3}$ as a matching minor.
Little's result gave a structural solution to P\'olya's question.
It took another $24$ years until McCuaig and, independently, Robertson, Seymour, and Thomas found a precise description of bipartite graphs without a $K_{3,3}$ matching minor that implied a polynomial time algorithm for the computation of the permanent for the corresponding $(0,1)$-matrices \cite{robertson1999permanents,mccuaig2004polya}.
The methods used to obtain this result strongly resemble those used in the more classical theory of (ordinary) graph minors.
Recently it was shown by Giannopoulou, Kreutzer, and Wiederrecht that also the exclusion of a planar bipartite graph as a matching minor in bipartite graphs gives rise to a class of $(0,1)$-matrices on which the permanent can be computed efficiently \cite{giannopoulou2021excluding}. 
We give a more in-depth discussion on the relation of permanents with matching minors in Subsection~\ref{subsec_history}.

\paragraph{Counting perfect matchings in $H$-minor-free graphs.}
Notice that all the results above deal with the permanent of general $(n\times n)$-matrices with $(0,1)$-entries.
Such matrices can be seen as the \emph{biadjacency matrices} of bipartite graphs, hence the tight relation with bipartite graphs and their perfect matchings.

The complexity of counting perfect matchings on special graph classes  has been studied extensively.
We particularly emphasise here the study of the problem of counting perfect matchings in $H$-minor-free graphs.
In his work on the dimer problem, Kasteleyn originally proved that every planar graph has a Pfaffian orientation and thus the number of perfect matchings can be computed efficiently.
In \cite{Kasteleyn61thest,kasteleyn1967graph} it was claimed that this method could be extended to graphs of bounded genus by combining several orientations.
This was proven by Galluccio and Loebl \cite{GalluccioL99onthe} for orientable surfaces and by Tesler \cite{Tesler00match} for non-orientable surfaces.

We now enter the realm of Graph Minors  Series  developed by Robertson and Seymour (R\&S).
Vazirani \cite{Vazirani89ncalg} lifted Kasteleyn's result to graphs that exclude $K_{3,3}$ as a minor by using Wagner's structural description \cite{wagner1937eigenschaft} of $K_{3,3}$-minor-free graphs\footnote{As we describe in detail in Subsection~\ref{subsec_history}, the structure of $K_{3,3}$-minor-free graphs differs vastly from the structure of bipartite $K_{3,3}$-matching-minor-free graphs. Indeed, any graph is a minor of a bipartite graph which excludes $K_{3,3}$ as a matching minor.}.
Similarly, it was shown by Straub, Thierauf, and Wanger \cite{StraubTW14count} that $K_5$-minor-free graphs allow for a polynomial time algorithm that computes the number of perfect matchings.
By combining these ideas with known dynamic programming techniques on bounded treewidth graphs, Curticapean and, independently, Eppstein and Vazirani \cite{Curticapean14count,EppsteinV19ncalg} proposed an algorithm to compute the number of perfect matchings of any graph excluding a fixed \textsl{single-crossing minor} efficiently.
The key to this algorithm was the classic result of R\&S \cite{robertson1993excluding} that any single-crossing-minor-free graph has a tree decomposition where the torso of each large enough bag\footnote{Here ``large enough'' only depends on the excluded single-crossing minor.} is a planar graph.
Recently these positive results were matched by Curticapean and Xia \cite{CurticapeanXia22} who showed that counting the perfect matchings of $K_{8}$-minor-free graphs remains $\#\mathsf{P}$-hard.
In fact, they showed an even stronger result by tying this hardness result to the occurrence of \emph{vortices} in the celebrated Graph Minors Structure Theorem (GMST) by R\&S \cite{robertson2003graph}.
This computational lower bound was matched in \cite{ThilikosW22killi1} where a precise description of those $H$  whose minor exclusion allows to avoid the occurrence of vortices in the GMST by R\&S was given.
With this result, a full dichotomy for counting perfect matchings in $H$-minor-free graphs was established.

\paragraph{Our contribution.}

In this paper we leave the restricted setting of symmetric $(0,1)$-matrices and focus on biadjacency matrices of bipartite graphs with perfect matchings.
That is, our results apply to all square $(0,1)$-matrices.

We extend the theory of matching minors in bipartite graphs and derive a structure theorem in the spirit of the structure theorem for single-crossing-minor-free graphs by R\&S \cite{robertson1993excluding}.
This is the first complete structural result on matching minors since the solution of P\'olya's Permanent Problem by McCuaig et al.\@ and the first ever extension of the tractability of the permanent on matching minor-closed classes of graphs beyond the exclusion of $K_{3,3}$ and planar graphs (see \cref{fig_classes} for an illustration of the containment relation between the relevant classes). 
Our theorem unifies the setting of $H$-matching-minor-free bipartite graphs where $H$ is  planar and the realm of $K_{3,3}$-matching-minor-free bipartite graphs (also known as bipartite Pfaffian graphs) into a single more powerful framework.

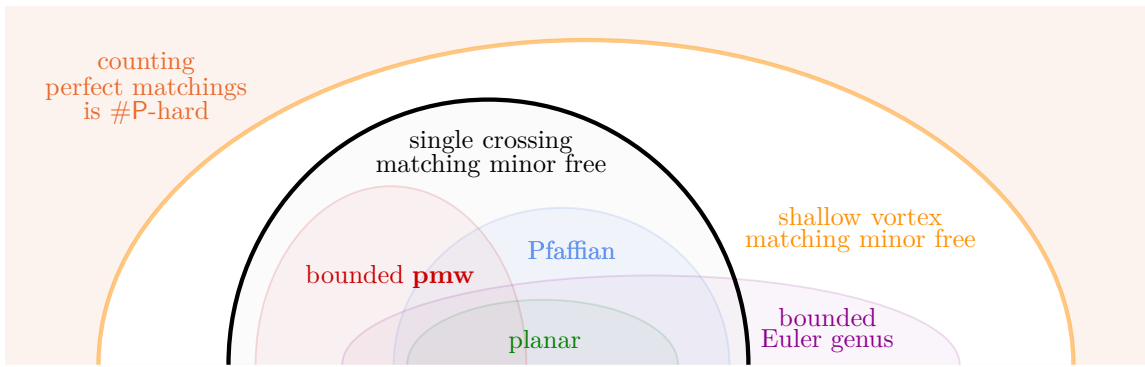
\begin{figure}[h]
	\centering
	\makefast{\scalebox{.9}{% !TEX root = ../single_comb.tex
% !TeX spellcheck = en_GB	
	\begin{tikzpicture}[scale=0.8]
			\pgfdeclarelayer{background}
			\pgfdeclarelayer{foreground}
			\pgfsetlayers{background,main,foreground}
			\node[v:ghost] (C) {};
			
			\node[v:ghost,position=0:0mm from C] (PlanarMidBottom) {};
			\node[v:ghost,position=170:25mm from PlanarMidBottom] (PlanarLabel) {\textcolor{AO}{planar}};
			
			\node[v:ghost,position=180:28mm from C] (PMWMidBottom) {};
			\node[v:ghost,position=147:30mm from PMWMidBottom] (PMWLabel) {\textcolor{BostonUniversityRed}{bounded $\mathbf{pmw}$}};
			
			\node[v:ghost,position=0:9.5mm from C] (PfaffianMidBottom) {};
			\node[v:ghost,position=144:36mm from PfaffianMidBottom] (PfaffianLabel) {\textcolor{CornflowerBlue}{Pfaffian}};
			
			\node[v:ghost,position=0:9.5mm from C] (PfaffianMidBottom) {};
			\node[v:ghost,position=144:36mm from PfaffianMidBottom] (PfaffianLabel) {\textcolor{CornflowerBlue}{Pfaffian}};
			
			\node[v:ghost,position=0:13mm from C] (CrossingMidBottom) {};
			\node[v:ghost,position=139:63mm from CrossingMidBottom] (CrossingLabel1) {single crossing};
			\node[v:ghost,position=270:4.5mm from CrossingLabel1] (CrossingLabel2) {matching minor free};
			
			\node[v:ghost,position=0:52mm from PlanarMidBottom] (GenusMidBottom) {}; 
			\node[v:ghost,position=160:26mm from GenusMidBottom] (GenusLabel1) {\textcolor{DarkMagenta}{bounded}};
			\node[v:ghost,position=270:4.5mm from GenusLabel1] (GenusLabel2) {\textcolor{DarkMagenta}{Euler genus}};
			
			\node[v:ghost,position=0:73mm from PlanarMidBottom] (VortexMidBottom) {}; 
			\node[v:ghost,position=145:48mm from VortexMidBottom] (VortexLabel1) {\textcolor{PrincetonOrange}{shallow vortex}};
			\node[v:ghost,position=270:4.5mm from VortexLabel1] (VortexLabel2) {\textcolor{PrincetonOrange}{matching minor free}};
			
			\node[v:ghost,position=162:180mm from VortexMidBottom] (OutsideLabel1) {\textcolor{DeepCarrotOrange}{counting}};
			\node[v:ghost,position=270:4.5mm from OutsideLabel1] (OutsideLabel2) {\textcolor{DeepCarrotOrange}{perfect matchings}};
			\node[v:ghost,position=270:4.5mm from OutsideLabel2] (OutsideLabel3) {\textcolor{DeepCarrotOrange}{is $\classP$-hard}};
			
			\begin{pgfonlayer}{background}
			
			\node[rectangle, draw = white,fill = DeepCarrotOrange,opacity=0.08,minimum width = 17cm, minimum height = 5.3cm] (r) at (-1.8,3.31) {};
				
			\draw[thick,white,fill=white,opacity=1] (VortexMidBottom) arc (0:180:90mm and 60mm);
			\draw[line width=1.8pt,PrincetonOrange,opacity=0.5] (VortexMidBottom) arc (0:180:90mm and 60mm);	
				
			\draw[thick,DarkMagenta,fill=DarkMagenta,opacity=0.04] (GenusMidBottom) arc (0:180:57mm and 16.5mm);
			\draw[thick,DarkMagenta,opacity=0.1] (GenusMidBottom) arc (0:180:57mm and 16.5mm);
					
			\draw[thick,AO,fill=AO,opacity=0.04] (PlanarMidBottom) arc (0:180:25mm and 12mm);
			\draw[thick,AO,opacity=0.1] (PlanarMidBottom) arc (0:180:25mm and 12mm);
			
			\draw[thick,BostonUniversityRed,fill=BostonUniversityRed,opacity=0.04] (PMWMidBottom) arc (0:180:25mm and 33mm);
			\draw[thick,BostonUniversityRed,opacity=0.1] (PMWMidBottom) arc (0:180:25mm and 33mm);
			
			\draw[thick,CornflowerBlue,fill=CornflowerBlue,opacity=0.07] (PfaffianMidBottom) arc (0:180:31mm and 29mm);
			\draw[thick,CornflowerBlue,opacity=0.1] (PfaffianMidBottom) arc (0:180:31mm and 29mm);
			
			\draw[thick,fill,opacity=0.018] (CrossingMidBottom) arc (0:180:48mm and 49mm);
			\draw[line width=1.8pt,opacity=1] (CrossingMidBottom) arc (0:180:48mm and 49mm);
			
			\end{pgfonlayer}
		\end{tikzpicture}}}{\scalebox{.476}{\includegraphics{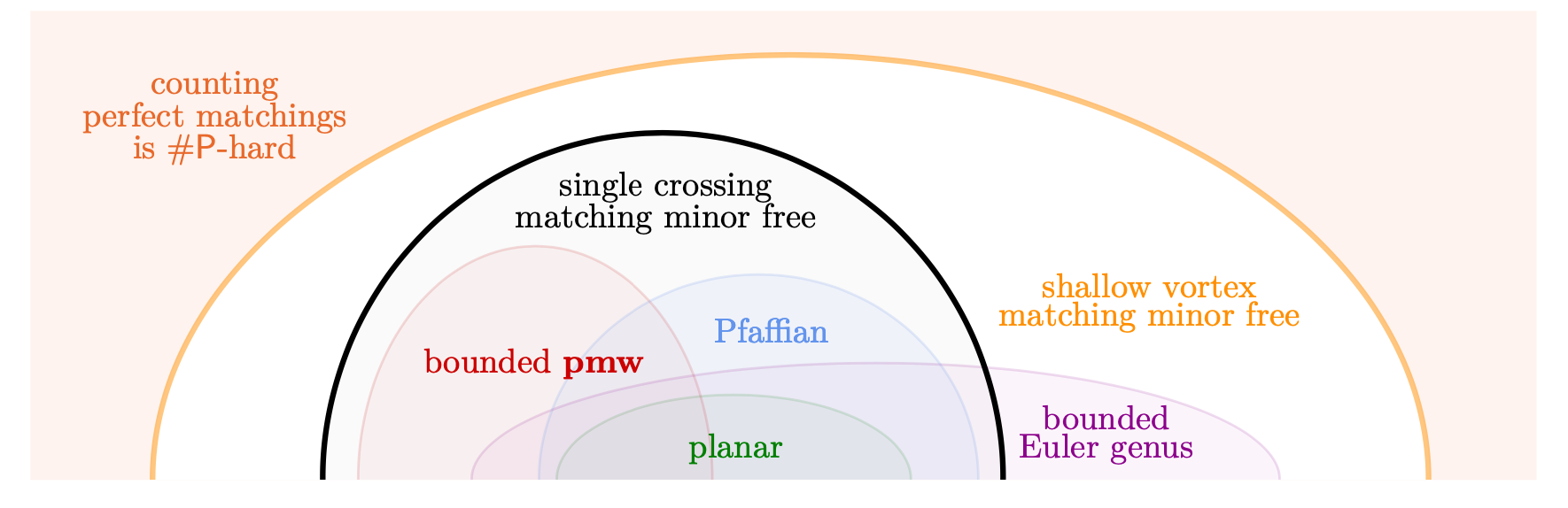}}}
	
	\caption{Known matching minor-closed classes of bipartite graphs.
	% for whose biadjacency matrices the computation of the permanent is tractable.
		The case of planar graphs was resolved by Kasteleyn \cite{kasteleyn1967graph} while the general result for bounded Euler genus is due to Gallucio, Loebl, and Tesler \cite{GalluccioL99onthe,Tesler00match}.
		The algorithm for Pfaffian graphs is a consequence of the solution of P\'olya's Permanent Problem \cite{robertson1999permanents,mccuaig2004polya}, and the case of bounded perfect matching width ($\mathsf{pmw}$) was treated in \cite{giannopoulou2021excluding}.
		We unify the latter two classes into the families of single-crossing matching-minor-free graphs.
		The area above the shallow vortex matching-minor-free class is where our complexity lower bounds apply. }
	\label{fig_classes}
\end{figure}

A graph is said to be \emph{matching covered} if it is connected and each of its edges is contained in some perfect matching.
A \emph{tight cut} is an edge cut in a matching covered graph $G$ such that any perfect matching contains exactly one edge in this cut.
Notice that the graph obtained by identifying one of the two shores of a tight cut into a single vertex results again in a matching covered graph.
Lovász noticed \cite{lovasz1987matching} that any matching covered graph can be decomposed by repeated such \emph{tight cut contractions} into a unique set of matching covered graphs where all tight cuts have a trivial shore.
These ``prime'' elements are called \emph{bricks} if they are non-bipartite and \emph{braces} if they are bipartite.
Bricks and braces are of particular importance in Matching Theory since many properties, including the computation of the permanent, can be reduced to these prime elements of the tight cut decomposition (see for example \cref{thm_splicegeneratingfunctions}).

We call the tight cut contraction of the closed neighbourhood of a vertex of degree two a \emph{bicontraction}.
A \emph{matching minor} of a matching covered graph $G$ is a graph $H$ which can be obtained from some subgraph $H'$ of $G$ such that $G-H'$ has a perfect matching by repeated applications of bicontractions.

In the style of R\&S's celebrated result on \textsl{single-crossing minors} \cite{robertson1993excluding} we define the class of \textsf{single-crossing matching minors} to be the class of all bipartite matching covered graphs $H$ for which some $t$ exists such that $H$ is a matching minor of the \emph{single-crossing matching grid of order $t$}.
For an illustration of this grid-like graph see \cref{fig_singlecrossingmatchinggrid}.
We denote the class of all single-crossing matching minors by $\mathfrak{S}.$
Our main structural result is the following.

\begin{figure}
	\centering
	\makefast{ % !TEX root = ../single_comb.tex
% !TeX spellcheck = en_GB	
	\begin{tikzpicture}[scale=0.8]
		\pgfdeclarelayer{background}
		\pgfdeclarelayer{foreground}
		\pgfsetlayers{background,main,foreground}
		\node[v:ghost] (C) {};
		
		\foreach \x in {1,3,5,7}
		\foreach \y in {2,4,6,8}
		{	
			%			\pgfmathsetmacro\X{}
			\node[v:main] (v_\x_\y) at (\x*0.87,\y*0.87){};
		}
		
		\foreach \x in {2,4,6,8}
		\foreach \y in {1,3,5,7}
		{	
			%			\pgfmathsetmacro\X{}
			\node[v:main] (v_\x_\y) at (\x*0.87,\y*0.87){};
		}
		
		\foreach \x in {1,3,5,7}
		\foreach \y in {1,3,5,7}
		{	
			%			\pgfmathsetmacro\X{}
			\node[v:mainempty] (v_\x_\y) at (\x*0.87,\y*0.87){};
		}
		
		\foreach \x in {2,4,6,8}
		\foreach \y in {2,4,6,8}
		{	
			%			\pgfmathsetmacro\X{}
			\node[v:mainempty] (v_\x_\y) at (\x*0.87,\y*0.87){};
		}
		
		\node[v:main,position=245:5mm from v_5_5] (a) {};
		\node[v:mainempty,position=295:5mm from v_4_5] (b) {};
		
		\begin{pgfonlayer}{background}
			
			\foreach \x in {1,...,8}
			{
				\draw[e:main] (v_\x_1) to (v_\x_8);
				\draw[e:main] (v_1_\x) to (v_8_\x);
			}
			
			\foreach \z in {1,...,8}
			{
				\draw[e:coloredborder] (v_1_\z) to (v_2_\z);
				\draw[e:coloredborder] (v_3_\z) to (v_4_\z);
				\draw[e:coloredborder] (v_5_\z) to (v_6_\z);
				\draw[e:coloredborder] (v_7_\z) to (v_8_\z);
			}
			
			\foreach \z in {1,...,8}
			{
				\draw[e:colored] (v_1_\z) to (v_2_\z);
				\draw[e:colored] (v_3_\z) to (v_4_\z);
				\draw[e:colored] (v_5_\z) to (v_6_\z);
				\draw[e:colored] (v_7_\z) to (v_8_\z);
			}
			
			\draw[e:coloredborder] (a) to (b);
			\draw[e:main] (b) to (v_4_5);
			\draw[e:main] (b) to (v_5_4);
			\draw[e:main] (a) to (v_5_5);
			\draw[e:main] (a) to (v_4_4);
			
			\draw[e:colored] (a) to (b);
			
		\end{pgfonlayer}
	\end{tikzpicture}}{\scalebox{.15}{\includegraphics{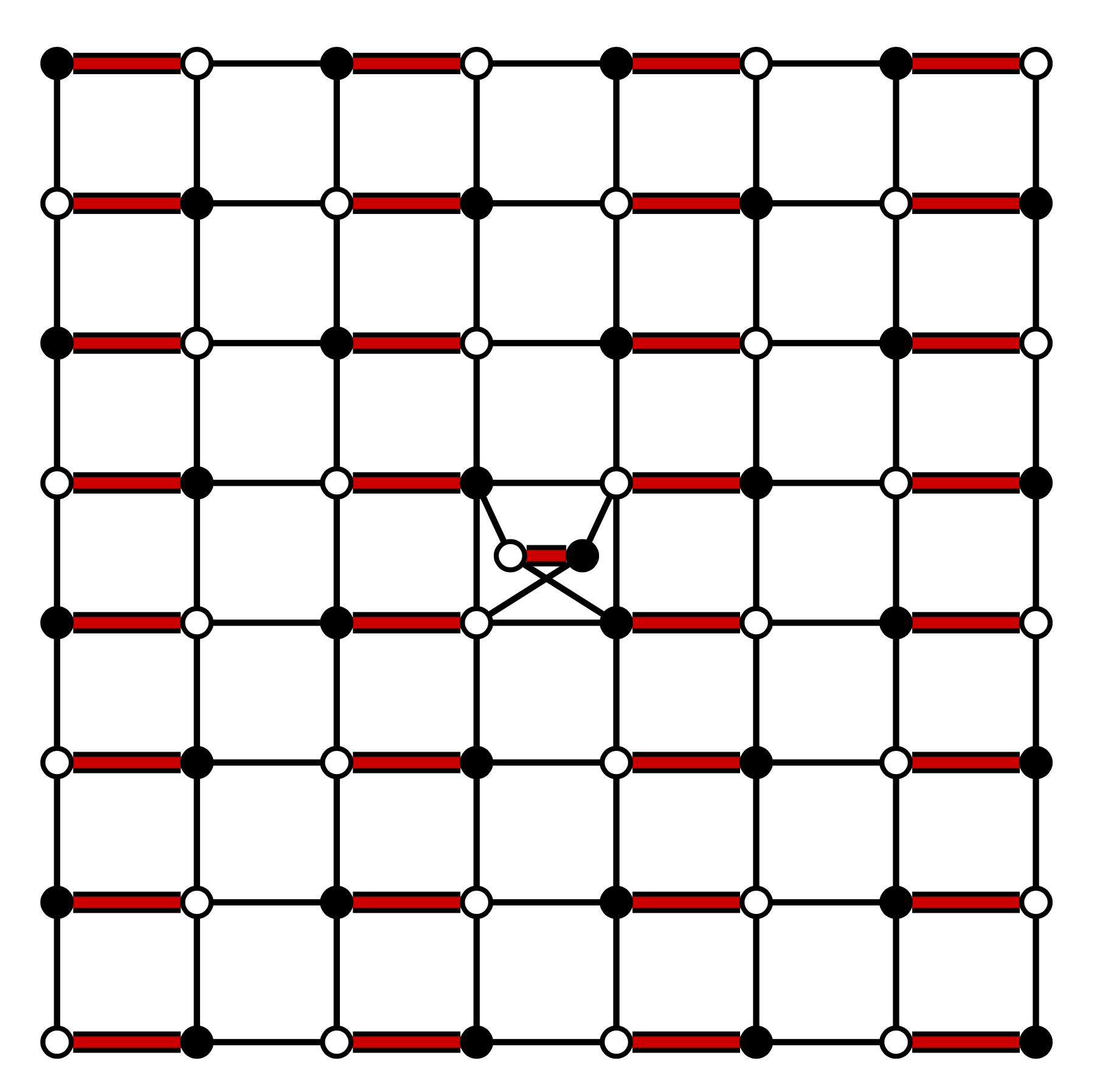}}}
	\caption{The single-crossing matching grid of order $4.$}
	\label{fig_singlecrossingmatchinggrid}
\end{figure}
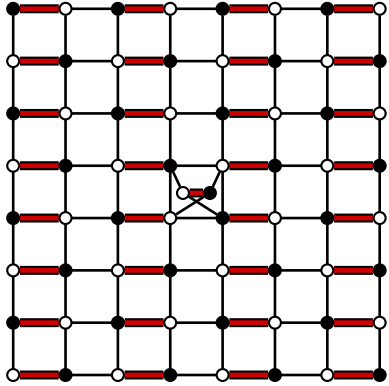

\begin{theorem}\label{thm_matchingmainthm}
	There exists a function $h_1\colon \mathbb{N}\rightarrow\mathbb{N}$ such that for every $H\in\mathfrak{S},$ every brace $B$ that excludes $H$ as a matching minor is either Pfaffian or satisfies $\pmw{B}\leq h_1(|H|).$
\end{theorem}
Here the term $\mathsf{pmw}$ refers to \emph{perfect matching width}, a treewidth-like parameter introduced by Norin \cite{norine2005matching} as a tool for the study of matching minors.
We postpone the definition of perfect matching width to \cref{sec_perfectmatchingwidth}.
This result is parametrically tight since every single-crossing matching grid has a non-Pfaffian brace with large perfect matching width.

Our proofs are constructive and imply an algorithm that can find a decomposition of every bipartite matching covered graph $B$ that excludes a single-crossing matching minor $H$ into, first its braces, then, second, each brace of small perfect matching width into a small width perfect matching decomposition, in polynomial time.
By combining dynamic programming on the parts of small perfect matching width with the concept of \textsl{Pfaffian orientations} on the Pfaffian braces, we then obtain our algorithmic main result:
\begin{theorem}\label{thm_algomainthm}
There exist a function $h_2\colon \mathbb{N}\rightarrow{\mathbb{N}}$ and an algorithm that, for every $H\in\mathfrak{S}$ and every bipartite graph $B$ with a perfect matching that does not contain $H$ as a matching minor, outputs the number of perfect matchings in $B,$ and thus the permanent of its biadjacency matrix, in time $O(\Abs{B}^{h_2(|H|)}).$
\end{theorem}
In light of \cref{thm_algomainthm} one might feel compelled to believe that counting perfect matchings is tractable on \textsl{all} proper matching minor-closed classes of bipartite graphs.
To complement the above algorithmic result we show that this is not the case.
\begin{theorem}\label{thm_K5hardness}
	Counting perfect matchings is $\#\mathsf{P}$-hard on the class of bipartite $K_{5,5}$-matching-minor-free graphs with perfect matchings.
\end{theorem}
In fact, we show in \cref{@interconnected} that the tractability of counting perfect matchings stops at a point which is similar to the setting of ordinary minors \cite{CurticapeanXia22}.
That is, we introduce a notion of \textsl{vortices} appropriate for the setting of matching minors in bipartite graphs and identify these as obstacles for the tractability of the permanent.
We stress here that, despite the similarity to the setting of $H$-minor-free graphs, our result is incomparable to those of Curticapean and Xia    \cite{CurticapeanXia22} and the recent algorithmic results of Thilikos and Wiederrecht  in \cite{ThilikosW22killi1}.

In the remainder of this introduction we give an in-depth discussion on the connection between matching minors in bipartite graphs and the tractability of the permanent itself as well as a clear distinction between matching minors and ordinary minors.
That is, for every $t\in \N,$ there exists a brace with perfect matching width two and a bipartite Pfaffian  graph such that both   contain $K_t$ as a minor.
In \cref{subsec_techniques} we provide a high-level explanation of the different techniques used to obtain our results.

\subsection{Permanent and matching minors}\label{subsec_history}

Let us start by recalling the formal definition of the \emph{determinant} of an  $(n\times n)$-matrix $A=(a_{i,j}),$
\begin{eqnarray*}
	\mathsf{det}(A)=\sum_{\sigma\in S_{n}}\Big(\mathsf{sgn}(\sigma)\cdot \prod_{i=1}^{n}a_{i,\sigma(i)}\Big),
\end{eqnarray*}
where $\mathsf{sgn}(\sigma)$ is the sign of the permutation $\sigma\in S_{n}.$
It is apparent that the difference between the determinant and the permanent (define in \cref{def_perm}) is the additional factor of the sign of the permutation $\sigma.$
This observation is exactly the root of P\'olya's idea of changing the signs of some entries of a matrix.
The strategy is to change the signs in such a way that they cancel out the signs of the permutations in the definition of the determinant and thereby produce the permanent of the original matrix.

Let us focus on the matrix operations used by Little in the characterization of the $(0,1)$-matrices on which P\'olya's strategy can be applied \cite{little1975characterization}.

Consider the following two operations on an $A=(a_{i,j})\in\Set{0,1}^{n\times n}$ with $\mathsf{perm}(A)\neq 0.$
A \emph{cross deletion} is the simultaneous removal of some row \textsl{and} some column of $A,$ whose shared entry is equal to one, while an \emph{entry deletion} is the operation that replaces some non-zero entry of $A$ with a zero.
A $(0,1)$-matrix $A'$ is said to be a \emph{conformal submatrix} of $A$ if there exists a sequence $A_1,\ldots,A_h$ of $(0,1)$ matrices with $A_0=A,$ $A_h=A',$ $\mathsf{perm}(A_i)\neq 0$ for all $i\in[n],$ and for every $i\in[n-1],$ the matrix $A_{i+1}$ is obtained from $A_i$ by a cross deletion or an entry deletion.
A class $\mathcal{A}$ of $(0,1)$-matrices is said to be \emph{hereditary} if for every $A\in\mathcal{A},$ $\mathcal{A}$ contains all conformal submatrices of $A.$

A {\em cross deletion} of the row $i$ and the column $j$ can be seen as the restriction of the possible choices for $\sigma\in S_n$ where $\sigma(i)=j,$ while the entry deletion of the entry $a_{i,j}$ is the exclusion of all $\sigma\in S_n$ with $\sigma(i)=j$ in the computation of the permanent.
By the requirement that performing either operation does not drop the permanent to zero, both of these operations restrict the formula for $\mathsf{perm}(A)$ to a proper, but non-trivial, subsum of the same expression. 

Now suppose our matrix $A$ has a row, say row $i,$ (or column) with exactly two non-zero entries, say $a_{i,i_1}$ and $a_{i,i_2},$ $i_1<i_2.$
Since the permanent is invariant under the permutation of rows and columns, we may assume $i=i_2=n$ and $i_1=n-1.$
Then we can partition the set $Y\subseteq S_n$ of permutations $\sigma$ with $A_{\sigma}\neq 0$ into the sets $Y_1,$ $Y_2,$ and $Y_3$ such that:
For $j\in[2],$ $Y_j$ contains exactly those $\sigma$ with $\sigma(n)=i_j$ and the entry at $i_j$ of the row $\sigma^{-1}(i_{3-j})$ is $0$ and $Y_3$ contains all $σ$  for which the entry at $i_{j}=\sigma(n)$ of the row $\sigma^{-1}(i_{3-j})$ is equal to $1.$
Notice that $|Y_{3}|$ is always even since for every $\sigma\in Y_3$ with $i_j=\sigma(n)$ there exists a $\sigma'\in Y_3$ where $\sigma'(n)=i_{3-j},$ $\sigma'(\sigma^{-1}(i_{3-j}))=i_j,$ and $\sigma(h)=\sigma(h')$ for all $h\in[n-1]\setminus\{\sigma^{-1}(i_{3-j})\}.$
Moreover, let $Y'\subseteq S_{n-1}$ be the collection of all $\sigma'\in S_{n-1}$ such that there exist $\sigma\in Y$ and $j\in[n-1]$ with
$$ \sigma'(k)=\sigma(k) \text{ for all }k\in[n-1]\setminus\Set{j},~\sigma(j)\in\Set{n-1,n}\text{ and }\sigma'(j)=n-1.$$
It follows that $\Abs{Y'}=\Abs{Y_1}+\Abs{Y_2}+\frac{1}{2}\Abs{Y_3}.$
Let us define a new column by setting $a'_{j,i_1}\coloneqq\max\Set{a_{j,i_1},a_{j,i_2}}$ for every $j\in[n].$
Now let $A'$ be the matrix obtained from $A$ by replacing the column $i_1$ with $(a'_{j,i_1})_{j\in[n]}$ and then cross deleting row $i$ and column $i_2.$
We call the above operation a \emph{bicontraction} of the matrix $A$ at the row $i.$
In case one of the two columns $i_1$ and $i_2$ also has exactly two non-zero entries, we call the corresponding bicontraction \emph{elementary}.
Notice that in this case we have $Y_3=\emptyset$ and thus an elementary bicontraction does not change the permanent.
A class $\mathcal{A}$ of $(0,1)$-matrices is said to be \emph{closed} if it is hereditary, and for every $A\in\mathcal{A},$ $\mathcal{A}$ contains all matrices obtainable from $A$ by bicontractions.

\paragraph{From matrices to matching minors.}

Building on the previous discussion, we now introduce the concept of matchings minors as the graph theoretic analogue of the above matrix operations.

A set $X\subseteq\V{G}$ is \emph{conformal} if $G-X$ has a perfect matching and a subgraph $H$ of $G$ is \emph{conformal} if $V(H)$ is conformal.

Any $(0,1)$-matrix $A=(a_{i,j})$ can be seen as the biadjacency matrix of a bipartite graph $B(A).$
Let $A=(a_{i,j})\in\Set{0,1}^{n\times n}$ be a matrix for some positive $n$ and let $\sigma\in S_n.$
Consider the term $A_{\sigma}\coloneqq \prod_{i=1}^na_{i,\sigma(i)}$ and observe that $A_{\sigma}=1$ if and only if for all $i\in[n],$ we have $a_{i,\sigma(i)}=1.$
Since each $a_{i,j}$ corresponds to an edge of $B(A),$ we may associate with $\sigma$ the set $F_{\sigma}\coloneqq \CondSet{\Set{u_i,v_{\sigma(i)}}}{i\in[n]\text{ and }a_{i,\sigma(i)}=1}.$
It follows that $A_{\sigma}=1$ if and only if $F_{\sigma}$ is a perfect matching of $B(A).$
Hence, $\mathsf{perm}(A)$ equals the total number of perfect matchings in $B(A).$
Since determining whether a bipartite graph has a perfect matching is polynomial time solvable, deciding whether $\mathsf{perm}(A)=0$ is as well.
Moreover, it follows that a bipartite graph $H$ with perfect matching and biadjacency matrix $A_H$ is a conformal subgraph of a bipartite graph $B$ with a perfect matching and biadjacency matrix $A_B$ if and only if $A_H$ is a conformal submatrix of $A_B.$

Notice, that a row (or column) of $A_H$ with exactly two non-zero entries corresponds to a vertex $v$ of $H$ with exactly two neighbours.
Hence, $v$ is incident with exactly two edges $e_1$ and $e_2$ in $H.$
Bicontracting this row (column) in $A_H$ corresponds to the operation of contracting the edges $e_1$ and $e_2$ simultaneously and then removing all resulting loops and parallel edges.
We call this operation a \emph{bicontraction}.
\begin{definition}[Matching minor]
Let $G$ be a graph with a perfect matching.
A graph $H$ is a \emph{matching minor} of $G$ if it can be obtained from a conformal subgraph of $G$ by bicontractions.	
\end{definition}
Given bipartite graphs $B$ and $H$ with perfect matchings and biadjacency matrices $A_B,$ $A_H$ respectively, then $H$ is a matching minor of $B$ \textsl{if and only if} $A_H$ can be obtained from a conformal submatrix of $A_B$ by bicontractions.
Hence a class $\mathcal{B}$ of bipartite graphs with perfect matchings is closed under matching minors if and only if $\{ A_B \mid B\in\mathcal{B} \}$ is a complete class of $(0,1)$-matrices.
Hence, the theory of matching minors for bipartite graphs gives rise to a structural theory of square $(0,1)$-matrices and their permanents.

\paragraph{Matching minors and (ordinary) minors.}
A natural question which might arise is, \textsl{``To what extent the theory of matching minors is actually different from the theory of Graph Minors  created by R\&S?''}.
To be more precise, is perfect matching width different from treewidth and does the exclusion of $K_{3,3}$ as a matching minor imply also the exclusion of some minor?
In \cite{giannopoulou2019braces}, the authors give an infinite family of braces of perfect matching width two, which contains arbitrary large clique minors.
The second question has a similar outcome.
It is a well known fact in graph minor theory that a grid with many dispersed faces which contain crossings also contains a large clique minor (see for example \cite{kawarabayashi2018new}).
The structure theorem for Pfaffian braces \cite{robertson1999permanents,mccuaig2004polya} says that the identification of planar braces at a four cycle is Pfaffian.
By doing so, it is possible to construct a Pfaffian brace that contains a grid with many dispersed crosses as a minor.
Hence, for every $t$ there exists a Pfaffian brace that contains $K_t$ as a minor.
This means that the results for matching minor-closed classes of bipartite graphs with perfect matchings as presented here are \textsl{incomparable}  with those for $H$-minor-free graphs.

\subsection{Our techniques}\label{subsec_techniques}

The cornerstone of our proofs is a ``societal'' version of the matching theoretic \textsl{Two Paths Theorem} as introduced in \cite{giannopoulou2021two}.
The \textsl{society} version of the Two Paths Theorem can be seen as the ``secret weapon'' behind the new proofs for the Flat Wall Theorem \cite{kawarabayashi2018new} and the new proof of the R\&S GMST for $H$-minor-free graphs \cite{kawarabayashi2020quickly} by Kawarabayashi, Thomas, and Wollan.
Roughly speaking, this theorem gives a characterisation of the  circumstances under which a given set of vertices of some graph $G$ can be forced to be embedded, in a given cyclic order,  on the same face of a planar graph that may be obtained from $G$ by resolving clique sums of order at most three.
The cases where this is not possible are certified by the \textsl{Two Path Theorem} as those in which one can link four vertices of the aforementioned ``society'' by two vertex disjoint paths which cross with respect to the given order \cite{jung1970verallgemeinerung,seymour1980disjoint,shiloach1980polynomial,thomassen19802,robertson1990graph}.
This tight topological link between a crossing and an (almost) planar embedding can be seen as the key enabling tool for the topological part of the Graph Minors series.

As such, for the development of a (topological) theory of matching minors, a matching theoretic counterpart of this result is more than desirable.

We introduce matching theoretic counterparts of the tools introduced in \cite{kawarabayashi2020quickly} for the study of ``almost embeddings'' and ``societies''.
In particular, we provide a ``society'' based characterisation that produces either a cross consisting of two alternating paths, which interacts nicely with the underlying matching structure or defines an area for which a fixed order of some prescribed vertices can be realised while the area itself appears ``flat''.
The notion of flatness is a well-studied concept from the Graph Minors series.
As part of our contribution in this paper, we give a precise and versatile formalisation of a matching theoretic analogue of flatness.

The original matching theoretic \textsl{Two Paths Theorem} from \cite{giannopoulou2021two} has two main features which make it difficult to be applied in a less restricted setting:
The theorem is only applicable to braces, and  it requires a cycle of length four.
We overcome both problems by identifying a key feature of non-trivial tight cuts.
That is, they cannot pierce through a collection of pairwise-disjoint conformal cycles.
This observation gives rise to the idea to equip a matching theoretic society with a small grid-like structure that provides the necessary infrastructure to fully control every appearance of $K_{3,3}$ as a matching minor (we stress that $K_{3,3}$ turns out to be the actual obstruction to the absence of crossing pairs of alternating paths).

The original proof of the structure theorem for single-crossing-minor-free graphs by R\&S \cite{robertson1993excluding} is relatively short and not exactly self-contained.
R\&S  employ several deep and technical results from the Graph Minors series which have since been revised in more modern works and replaced by much more digestible techniques.
However, a proof of the classic result for single-crossing-minor-free graphs using these modern techniques does not exist in the literature so far.
To illustrate our proof for matching minors in a friendlier setting, and as part of the development of our proof strategy, we provide a complete and almost self-contained\footnote{With the exceptions being the Two Path Theorem and the Grid Theorem.} proof of the original result in terms of societies.
This approach allows us to highlight the following key features of the proof.
\begin{enumerate}
	\item A reduction to cases with more robust and sufficient connectivity to trivialise the structure arising from small clique sums (or non-trivial tight cuts in the matching theoretic setting).
	\item An identification of several, pairwise equivalent, universal obstructions that arise from different cases in the analysis of the case of large treewidth (or perfect matching width).
	\item And finally, the application of the grid theorem (a matching theoretic version was proven in \cite{hatzel2019cyclewidth}) to obtain a well-structured basis for the application of the Two Paths Theorem.
\end{enumerate}
Essentially, step iii) finds a large grid and then, utilising step ii), determines that this grid must contain a large, extremely well-behaved area.
This can be seen as a stronger version of the Flat Wall Theorem for single-crossing-minor-free graphs.
The Two Paths Theorem (or its matching theoretic analogue) allows to analyse the structure ``outside'' the resulting planar wall.
If this structure contains a cross, we have found our single-crossing minor, otherwise the ``outside'' must also be planar (or at least be ``well-behaved'' in the matching setting).

Finally, to obtain our algorithmic results, we combine well known dynamic programming techniques with specialised methods to deal with the respective problems on graphs that  may have large width but behave well in a topological sense.

When designing a dynamic programming algorithm for tree decompositions where some bags have topological properties rather than being of bounded size one often runs into the problem that the number of children of corresponding nodes cannot be bounded.
In particular this can mean that the size of the union of all interfaces (or adhesions) towards the children cannot be bounded as well.
This is particularly true for the more sensitive nature of structural decompositions such as the one arising from the structure theorem of single-crossing-minor-free graphs.
In our case we encounter particular difficulties in the case of bipartite matching covered graphs.
To deal with such cases one usually employs specialised gadgets.
The purpose of these gadgets is to represent the information from the subtree below in a way that behaves well with respect to the methods used to handle the case of large bags.
To overcome this problem in the case of counting perfect matchings, instead of \cref{thm_algomainthm} we prove a much stronger and more general result.
That is, we allow the edges of our graph to be labelled with polynomials.
These polynomials will then be used to store partial generating functions for the perfect matchings within subgraphs in order to  encode the necessary information directly on these labels (a similar trick has been used in \cite{ThilikosW22killi1}).

\subsection{Organisation}\label{subsec_oragnisation}

The remainder of this paper is organised as follows.
In \cref{sec_excludesinglecrossing} we introduce the key notions developed by Kawarabayashi, Thomas, and Wollan for their revision of the proof for the GMST.
We highlight the undirected version of the Two Paths Theorem in its society form and present a new and short proof of the structure theorem for single-crossing-minor-free graphs.
\Cref{sec_background} is dedicated to the introduction of the matching theoretic concepts deployed in this paper together with a short discussion of the structural aspects of $K_{3,3}$-matching-minor-free graphs.
We then move on to \cref{sec_rendition} where we introduce the matching theoretic analogues of ``partial'' and ``almost'' embeddings.
Moreover, this section contains a proof of a matching theoretic society version of the Two Paths Theorem.
The concepts and results of this section are held very general to allow for easy access in future papers.
These concepts are then utilised for the proof of \cref{thm_matchingmainthm} in \cref{sec_proof_struct}.
We present the proofs and techniques for our algorithmic main results in Subsection \ref{@desfavorecer} in the case for counting matchings, i.e., the computation of the permanent.
Finally, the \textsf{\#P}-hardness result of \cref{thm_K5hardness} is presented in Subsections \ref{@interconnected} and  \ref{@unanswerable}.
We conclude in \cref{@strengthening} with some conjectures and directions for further research.

%%%%%%%%%%%%%%%%%%%%%%%%%%%%%%%%%%%%%%%%%%%%%%%%%%%%%%%%%%%%%%%%%%%%%%%%%%%%%%%%%%%%%%%%%%%%%%%%%%%%%%%%%%%%%%%%%%%%%%%%%%%%%%%%%%%%%%%%%%%%%%%%%%%%%%%%%%%%%%%%%%%%%%%%%%%%%%%%%%%%%%%%%%%%%%%%%%%%%%%%%%%%%%%%%%%%%%%%%%%%%%%%%%%%%%%%%%%%%%%%%%%%%%%%%%%%%%%%%%%%%%%%%%%%%%%%%%%%%%%%%%%%%%%%

\section{Excluding an ordinary single-crossing minor}\label{sec_excludesinglecrossing}

A graph $G$ is said to be \emph{singly crossing} if it can be drawn in the plane with a single crossing.
A graph $H$ is called a \emph{single-crossing minor} if it is a minor of some singly crossing graph.
Please note that a single-crossing minor is not necessarily singly crossing itself.
A first instance of the exclusion of a single-crossing minor was found by Wagner \cite{wagner1937eigenschaft} in form of his characterisation of $K_5$-minor-free graphs.
A very similar description can be found for $K_{3,3}$-minor-free graphs:
Both classes can be, roughly, described as the graphs obtainable from planar graphs -- and a single non-planar exception in form of the M\"obius-ladder with four rungs in the case of $K_5,$ or the $K_5$ itself in the case of $K_{3,3}$ -- by joining these graphs along small clique sums and, possibly, removing some edges of these cliques.
We say that a graph, or a graph class, is \emph{SCM-free} if it minor-excludes some single-crossing minor.
This phenomenon was unified in the \textsl{parametric} characterisation of SCM-free graphs by R\&S \cite{robertson1993excluding}.
Roughly speaking, a graph excludes some single-crossing minor if and only if it can be pieced together from graphs of small treewidth and planar graphs by using small clique sums.
The proof of this theorem, as presented in \cite{robertson1993excluding}, is relatively short and relies heavily on specialised versions of the GMST \cite{robertson2003graph}.

In this section we introduce the more streamlined tool set for dealing with graph minors introduced by Kawarabayashi et al.\@ \cite{kawarabayashi2018new,kawarabayashi2020quickly} in their new proof of the Graph Minors Structure Theorem.
We then provide a new and short proof of the structure theorem for SCM-free graphs utilising this tool set.
This proof acts as a high-level sketch of the proof of our \hyperref[thm_matchingmainthm]{structural main theorem};
In the following sections we generalise the tool set of \cite{kawarabayashi2018new,kawarabayashi2020quickly} to the setting of bipartite graphs with perfect matchings.
Our proof of \cref{thm_matchingmainthm} then follows along the lines of the \textsl{much simpler} proof presented in this section while replacing the concepts used here with their newly introduced  matching theoretic analogues.

\subsection{Reduction to quasi-${4}$-connected graphs}

The first step is to reduce our graph to its quasi-$4$-connected components, a concept introduced by Grohe \cite{grohe2016quasi} as a generalisation of previous decompositions, the most fundamental of which is the decomposition of a graph into its blocks\footnote{Maximal connected subgraphs without a cut-vertex.}.
An interesting fact about these decompositions is that they can be seen to be uniquely determined by the underlying separator structure of the graph \cite{grohe2016tangles,grohe2016quasi,robertson1991graph}.
For an in-depth discussion on this uniqueness see also \cite{carmesin2017canonical}.

Let $G$ be a graph.
A \emph{separation} in $G$ is a tuple $(A,B)$ such that $A\cup B=V(G)$ and there is no edge in $G$ with one endpoint in $A\setminus B$ and the other in $B\setminus A.$
We call $A\cap B$ the \emph{separator} of $(A,B)$ and the \emph{order} of $(A,B)$ is $\Abs{A\cap B}.$
A separation $(A,B)$ is \emph{trivial} if $A\setminus B=\emptyset$ or $B\setminus A=\emptyset.$

For any $k\geq 1$ we say that $G$ is \emph{$k$-connected} if it has at least $k+1$ vertices and every separation of order at most $k-1$ in $G$ is trivial.

A graph $G$ is said to be \emph{quasi-$4$-connected} if it is $3$-connected and for all separations $(A,B)$ of order three either $\Abs{A\setminus B}\leq 1$ or $\Abs{B\setminus A}\leq 1$ holds.

A \emph{tree decomposition} for a graph $G$ is a tuple $(T,\beta)$ where $T$ is a tree and $\beta\colon V(T)\rightarrow 2^{V(G)}$ maps the vertices of $T$ to subsets of $V(G),$ we call the $\beta(t)$ the \emph{bags} of $(T,\beta),$ such that
\begin{enumerate}
	\item $\bigcup_{t\in V(T)}\beta(t)=V(G),$
	\item for every $e\in E(G)$ there exists $t\in V(T)$ with $e\subseteq \beta(t),$ and
	\item for every $v\in V(G)$ the set $\CondSet{t\in V(T)}{v\in\beta(t)}$ induces a subtree of $T.$
\end{enumerate}
The \emph{width} of a tree decomposition is defined to be $\max_{t\in V(T)}\Abs{\beta(t)}-1$ and the \emph{treewidth} of $G,$ denoted by $\tw{G},$ is defined to be the smallest width among all tree decompositions of $G.$
Let $t\in V(T)$ be any vertex and $e=td\in E(T)$ be an edge of $T$ incident with $t.$
We call the set $\beta(t)\cap\beta(d)$ the \emph{adhesion set} of $t$ corresponding to $d.$
We say that $(T,\beta)$ is of \emph{adhesion at most $k$} if every adhesion set of any vertex of $T$ has size at most $k.$
The \emph{torso} of $G$ at the vertex $t$ is the graph obtained from $\InducedSubgraph{G}{\beta(t)}$ by turning every adhesion set of $t$ into a clique.

\begin{proposition}[\cite{grohe2016quasi}]\label{thm_4components}
	Every graph $G$ has a tree decomposition $(T,\beta)$ of adhesion at most $3$ such that for all $t\in V(T)$ the torso of $G$ at $t$ is a minor of $G$ that is either quasi-$4$-connected or isomorphic to a complete graph on at most four vertices.
	Moreover, this decomposition can be found in time $\mathcal{O}(\Abs{V(G)}^3).$
\end{proposition}

Let $G$ be a graph and let $(T,\beta)$ be the tree decomposition of $G$ provided by \cref{thm_4components}, we call the torsos of $G$ at the vertices of $T$ the \emph{quasi-$4$-components} of $G.$

\subsection{Societies and the Two Paths Theorem}\label{subsec_society}

Let $G$ be a graph and let $s_1,s_2,t_1,t_2\in\V{G}.$
The \textsc{Two Disjoint Paths Problem} (\textsc{TDPP}) with \emph{terminals} $s_1,s_2,t_1,t_2$ is the question for the existence of two paths $P_1$ and $P_2$ such that for both $i\in[1,2]$ $P_i$ joins $s_i$ and $t_i$ and $P_1$ and $P_2$ are vertex disjoint.
The characterisation for the \yes-instances of the \textsc{TDPP} known as the \emph{Two Paths Theorem} plays an integral role in structural graph theory.
The statement of the Two Paths Theorem we present here makes use of the concept of so-called ``societies'' which play a focal role in \cite{kawarabayashi2020quickly}.

\begin{definition}[Society]\label{def_society}
	Let $\Omega$ be a cyclic permutation of the elements of some set which we denote by $V(\Omega).$
	A \emph{society} is a pair $(G,\Omega),$ where $G$ is a graph and $\Omega$ is a cyclic permutation with $V(\Omega)\subseteq V(G).$
	A \emph{cross} in a society $(G,\Omega)$ is a pair $(P_1,P_2)$ of disjoint paths\footnote{When we say two paths are disjoint we mean that their vertex sets are disjoint.} in $G$ such that $P_i$ has endpoints $s_i,t_i\in V(\Omega)$ and is otherwise disjoint from $V(\Omega),$ and the vertices $s_1,s_2,t_1,t_2$ occur in $\Omega$ in the order listed.
\end{definition}
Hence, $(G,s_1,s_2,t_1,t_2)$ is a \yes-instance of the \textsc{TDPP} if and only if the society $(G,\Omega),$ where $V(\Omega)=\{s_1,s_2,t_1,t_2\}$ and the vertices occur in $\Omega$ in the order listed, has a cross.

To fully present the Two Paths Theorem we need to introduce some topological concepts as well.

By a \emph{surface} we mean a closed compact $2$-dimensional manifold with or without boundary.
By the classification theorem of surfaces, every surface is homeomorphic to the sphere with $h$ handles and $c$ cross-caps added, and the interior of $d$ disjoint closed disks $\Delta_1,\dots,\Delta_d$ removed, in which case the \emph{Euler genus} of the surface is defined to be $2h+c.$
We call the union of the boundaries of the disks $\Delta_i$ the \emph{boundary} of the surface and each such boundary is a \emph{boundary component} of the surface.

Note that in this paper we are purely interested in the cases where $\Sigma$ is homeomorphic to the sphere or a closed disk. 
However, we give the definitions in the more general setting of \cite{kawarabayashi2020quickly} since we model or matching theoretic analogues after them and want to ensure compatibility with future works. 

\begin{definition}(Drawing on a surface)\label{def_drawing}
	A \emph{drawing} (with crossings) \emph{in a surface $\Sigma$} is a triple $\Gamma=(U,V,E)$ such that
	\begin{itemize}
		\item $V$ and $E$ are finite,
		\item $V\subseteq U\subseteq\Sigma,$
		\item $V\cup\bigcup_{e\in E}e=U$ and $V\cap (\bigcup_{e\in E}e)=\emptyset,$
		\item for every $e\in E,$ either $e=h(0,1),$ where $h\colon[0,1]_{\R}\rightarrow U$ is a homeomorphism onto its image with $h(0),h(1)\in V,$ or $e=h(\mathbb{S}^2-(1,0)),$ where $h\colon\mathbb{S}^2\rightarrow U$ is a homeomorphism onto its image with $h(0,1)\in V,$ and
		\item if $e,e'\in E$ are distinct, then $\Abs{e\cap e'}$ is finite.
	\end{itemize}
	We call the set $V,$ sometime referred to by $V(\Gamma),$ the \emph{vertices of $\Gamma$} and the set $E,$ referred to by $E(\Gamma),$ the \emph{edges of $\Gamma$}.
	If $G$ is graph and $\Gamma=(U,V,E)$ is a drawing with crossings in a surface $\Sigma$ such that $V$ and $E$ naturally correspond to $V(G)$ and $E(G)$ respectively, we say that $\Gamma$ is a \emph{drawing of $G$ in $\Sigma$ (with crossings)}.
\end{definition}

\begin{definition}[$\Sigma$-Decomposition]\label{def_sigmadecomposition}
	Let $\Sigma$ be a surface.
	A \emph{$\Sigma$-decomposition} of a graph $G$ is a pair $\delta=(\Gamma,\mathcal{D}),$ where $\Gamma$ is a drawing of $G$ is $\Sigma$ with crossings, and $\mathcal{D}$ is a collection of closed disks, each a subset of $\Sigma$ such that
	\begin{enumerate}
		\item the disks in $\mathcal{D}$ have pairwise-disjoint interiors,
		\item the boundary of each disk in $\mathcal{D}$ intersects $\Gamma$ in vertices only,
		\item if $\Delta_1,\Delta_2\in\mathcal{D}$ are distinct, then $\Delta_1\cap\Delta_2\subseteq V(\Gamma),$ and
		\item every edge of $\Gamma$ belongs to the interior of one of the disks in $\mathcal{D}.$
	\end{enumerate}
	Let $N$ be the set of all vertices of $\Gamma$ that do not belong to the interior of the disks in $\mathcal{D}.$
	We refer to the elements of $N$ as the \emph{nodes} of $\delta.$
	If $\Delta\in\mathcal{D},$ then we refer to the set $\Delta-N$ as a \emph{cell} of $\delta.$
	We denote the set of nodes of $\delta$ by $N(\delta)$ and the set of cells by $C(\delta).$
	For a cell $c\in C(\delta)$ the set of nodes that belong to the closure of $c$ is denoted by $\widetilde{c}.$
	Please note that this means that the cells $c$ of $\delta$ with $\widetilde{c}\neq\emptyset$ form the edges of a hypergraph with vertex set $N(\delta)$ where $\widetilde{c}$ is the set of vertices incident with $c.$
	For a cell $c\in C(\delta)$ we define $\sigma_{\delta}(c),$ or simply $\sigma(c)$ if $\delta$ is clear from the context, to be the subgraph of $G$ consisting of all vertices and edges drawn in the closure of $c.$
	We define $\pi_{\delta}\colon N(\delta)\rightarrow V(G)$ to be the mapping that assigns to every node in $N(\delta)$ the corresponding vertex of $G.$
	
	Isomorphisms between two $\Sigma$-decompositions are defined in the natural way.
\end{definition}

Notice that, given some $\Sigma$-decomposition $\delta$ of a graph $G,$ for any cell $c\in C(\delta),$ the nodes of $\widetilde{c}$ ordered by their appearance on the boundary of the closure $c,$ together with the vertices of $\sigma(c)$, form a society.
In case $\Abs{\widetilde{c}}\leq 3,$ whatever is drawn into $c$ may be removed and replaced by a complete graph on $\widetilde{c}.$
By doing so one can incorporate the resulting complete graph into the properly embedded part without introducing any crossings.
However, if $\Abs{\widetilde{c}}\geq 4$ the same operation would result in two crossing edges within the closure of $c.$
Hence, we distinguish between cells with at most $3$ nodes in their boundaries and cells with at least four nodes.

\begin{definition}[Vortex]\label{def_vortex}
	Let $G$ be a graph, $\Sigma$ be a surface and $\delta=(\Gamma,\mathcal{D})$ be a $\Sigma$-decomposition of $G.$
	A cell $c\in C(\delta)$ is called a \emph{vortex} if $\Abs{\widetilde{c}}\geq 4.$
	Moreover, we call $\delta$ \emph{vortex-free} if no cell in $C(\delta)$ is a vortex.
\end{definition}

\begin{definition}[Rendition]\label{def_rendition}
	Let $(G,\Omega)$ be a society, and let $\Sigma$ be a surface with one boundary component $B.$
	A \emph{rendition} of $G$ in $\Sigma$ is a $\Sigma$-decomposition $\rho$ of $G$ such that the image under $\pi_{\rho}$ of $N(\rho)\cap B$ is $V(\Omega),$ mapping one of the two cyclic orders of $B$ to the order of $\Omega.$
	
	If $(G,\Omega)$ has a vortex-free rendition in the disk, we say that $(G,\Omega)$ is \emph{flat}.
\end{definition}

These technical definitions allow us the state the Two Paths Theorem in the general context of the GMST as follows.

\begin{proposition}[(Societal) Two Paths Theorem, \cite{jung1970verallgemeinerung,seymour1980disjoint,shiloach1980polynomial,thomassen19802,robertson1990graph}]\label{thm_twopaths}
	A society $(G,\Omega)$ has no cross if and only if it is flat.
\end{proposition}

Let $(G,\Omega)$ be a society and $G$ be quasi-$4$-connected.
Keep in mind that, in case $(G,\Omega)$ is flat, the graph $G$ must in fact be planar since any cell $c$ of the rendition with $\Abs{\widetilde{c}}\geq 2$ and $\Abs{V(c)}\geq 2$ would induce a separation contradicting the quasi-$4$-connectedness of $G.$

\subsection{The structure of single-crossing-minor-free graphs}

In what follows we present a new proof of the structure theorem for SCM-free graphs.
This proof is roughly broken into the following steps, which we aim to replicate in the setting of bipartite graphs with perfect matchings.
\begin{enumerate}
	\item Identify a parametric graph $\{ U_t\}_{t\in\mathbb{N}}$ that captures the existence of a single-crossing minor in a way which interacts well with the following steps,
	\item decompose a given $U_t$-minor-free graph $G$ into its quasi-$4$-connected components (this step is provided by \cref{thm_4components}),
	\item within a large quasi-$4$-connected component $H$ of $G$ of huge treewidth find a big planar area $H_1$ together with a grid-like infrastructure (this is a special case of the Flat Wall Theorem), and
	\item show that $H_2\coloneqq H-H_1$ must also be planar and ``compatible'' with $H_1$ as otherwise we would find our single-crossing minor.
\end{enumerate}
Combining steps iii) and iv) then yields that any huge quasi-$4$-component of our SCM-free graph must be planar, where ``huge'' means that its treewidth is bounded from below by some function depending \textsl{only} on the excluded single-crossing minor.

\paragraph{Single-crossing grids.}

Let us start with a description of the universal obstruction.

An \emph{$(n\times m)$-grid} is the graph $G_{n,m}$ with vertex set $[1,n]\times[1,m]$ and edge set $\CondSet{\Set{(i,j),(i,j+1)}}{i\in[1,n],j\in[1,m-1]}\cup\CondSet{\Set{(i,j),(i+1,j)}}{i\in[1,n-1],j\in[1,m]}.$
We call the paths of the form $(i,1)(i,2)\dots(i,m)$ the \emph{rows} and the paths of the form $(1,j)(2,j)\dots(n,j)$ the \emph{columns} of the grid.
If $L$ is a row (column) of $G_{n,m}$ of the form $(i,1)(i,2)\dots(i,m)$ ($(1,j)(2,j)\dots(n,j)$), we call it the \emph{$i$th row} (\emph{$j$th column}) of $G_{n,m}$ and the vertex $(i,j)$ is the \emph{$j$th vertex of the $i$th row} and the \emph{$i$th vertex of the $j$th column}.
The edges of $G_{n,m}$ are numbered similarly.
An \emph{elementary $n$-wall}, $n\geq 3,$ is obtained from a $(2n\times n)$-grid by deleting every odd edge in every odd column, every even edge in every even column, and finally delete all occurring vertices of degree 1.
An \emph{$n$-wall} is a subdivision of an elementary $n$-wall.

Given a $(2n\times 2n)$-grid $H,$ we call the cycle $(n,n),(n,n+1),(n+1,n+1),(n+1,n)$ the \emph{central cycle} of $H.$
We generally denote the central cycle of $H$ by $C^1_H.$
Now consider $H-C_H^1,$ then this graph has a unique face\footnote{We implicitly assume $H$ to come with a cross-free embedding in the plane.} which used to contain $C_H^1,$ let $C_H^2$ be the cycle that bounds this face.
Suppose the cycles $C_H^1,\dots,C_H^{i-1}$ have been defined for $i\geq 3.$
Then we define $C_H^i$ to be the cycle in $H-(\bigcup_{j\in[1,i-1]}C_H^j)$ which bounds the face that used to contain the $C_H^j.$
Observe that we find precisely $n$ such cycles.
We call $(C_H^1,\dots,C_H^n)$ the \emph{centred layering} of $H.$

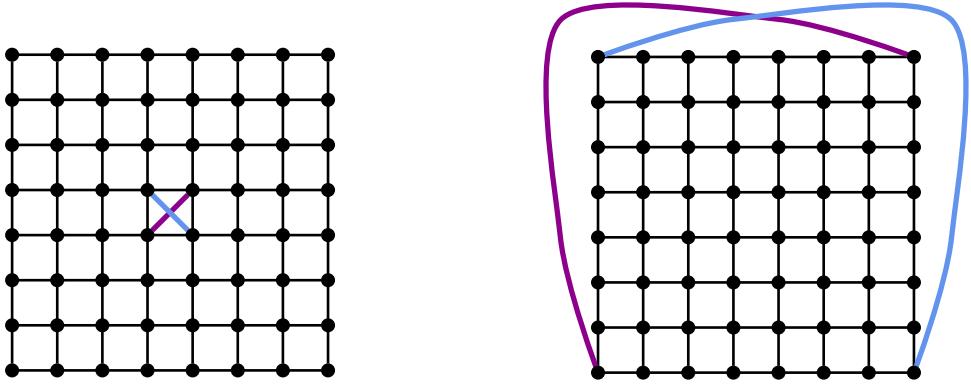
\begin{figure}
	\centering
\makefast{% !TEX root = ../single_comb.tex
% !TeX spellcheck = en_US	
\begin{subfigure}{0.48\textwidth}
		\centering
	\begin{tikzpicture}[scale=1]
		\pgfdeclarelayer{background}
		\pgfdeclarelayer{foreground}
		\pgfsetlayers{background,main,foreground}
		\node[v:ghost] (C) {};
		
		\foreach \x in {1,...,8}
		\foreach \y in {1,...,8}
		{	
%			\pgfmathsetmacro\X{}
			\node[v:main] (v_\x_\y) at (\x*0.6,\y*0.6){};
		}
	
		\node[v:ghost,position=90:8.5mm from v_8_8] (x2) {};
		\node[v:ghost,position=90:8.5mm from v_1_8] (y2) {}; 

		\begin{pgfonlayer}{background}
			
			\foreach \x in {1,...,8}
			{
				\draw[e:main] (v_\x_1) to (v_\x_8);
				\draw[e:main] (v_1_\x) to (v_8_\x);
			}
			
			\draw[e:thick,color=DarkMagenta] (v_4_4) to (v_5_5);
			\draw[e:thick,color=CornflowerBlue] (v_4_5) to (v_5_4);
			
		\end{pgfonlayer}
	\end{tikzpicture}
	\end{subfigure}
	\begin{subfigure}{0.48\textwidth}
		\centering
		\begin{tikzpicture}[scale=1]
			\pgfdeclarelayer{background}
			\pgfdeclarelayer{foreground}
			\pgfsetlayers{background,main,foreground}
			\node[v:ghost] (C) {};
			
			\foreach \x in {1,...,8}
			\foreach \y in {1,...,8}
			{	
				%			\pgfmathsetmacro\X{}
				\node[v:main] (v_\x_\y) at (\x*0.6,\y*0.6){};
			}
			
			\node[v:ghost,position=45:7mm from v_8_8] (x2) {}; 
			\node[v:ghost,position=90:5mm from v_4_8] (x1) {}; 
			\node[v:ghost,position=0:5mm from v_8_4] (x3) {};
			
			\node[v:ghost,position=135:7mm from v_1_8] (y2) {}; 
			\node[v:ghost,position=90:5mm from v_5_8] (y3) {}; 
			\node[v:ghost,position=180:5mm from v_1_4] (y1) {};
			
			\begin{pgfonlayer}{background}
				
				\foreach \x in {1,...,8}
				{
					\draw[e:main] (v_\x_1) to (v_\x_8);
					\draw[e:main] (v_1_\x) to (v_8_\x);
				}
				
				\draw[e:thick,color=DarkMagenta] plot [smooth, tension=2] coordinates {(v_1_1) (y1) (y2) (y3) (v_8_8)};
%				\draw[e:thick,color=CornflowerBlue] (v_1_8) to [curve through={(x1) (x2) (x3)}] (v_8_1);
				\draw[e:thick,color=CornflowerBlue] plot [smooth, tension=2] coordinates {(v_1_8) (x1) (x2) (x3) (v_8_1)};
				
			\end{pgfonlayer}
		\end{tikzpicture}
	\end{subfigure}}{\scalebox{.187}{\includegraphics{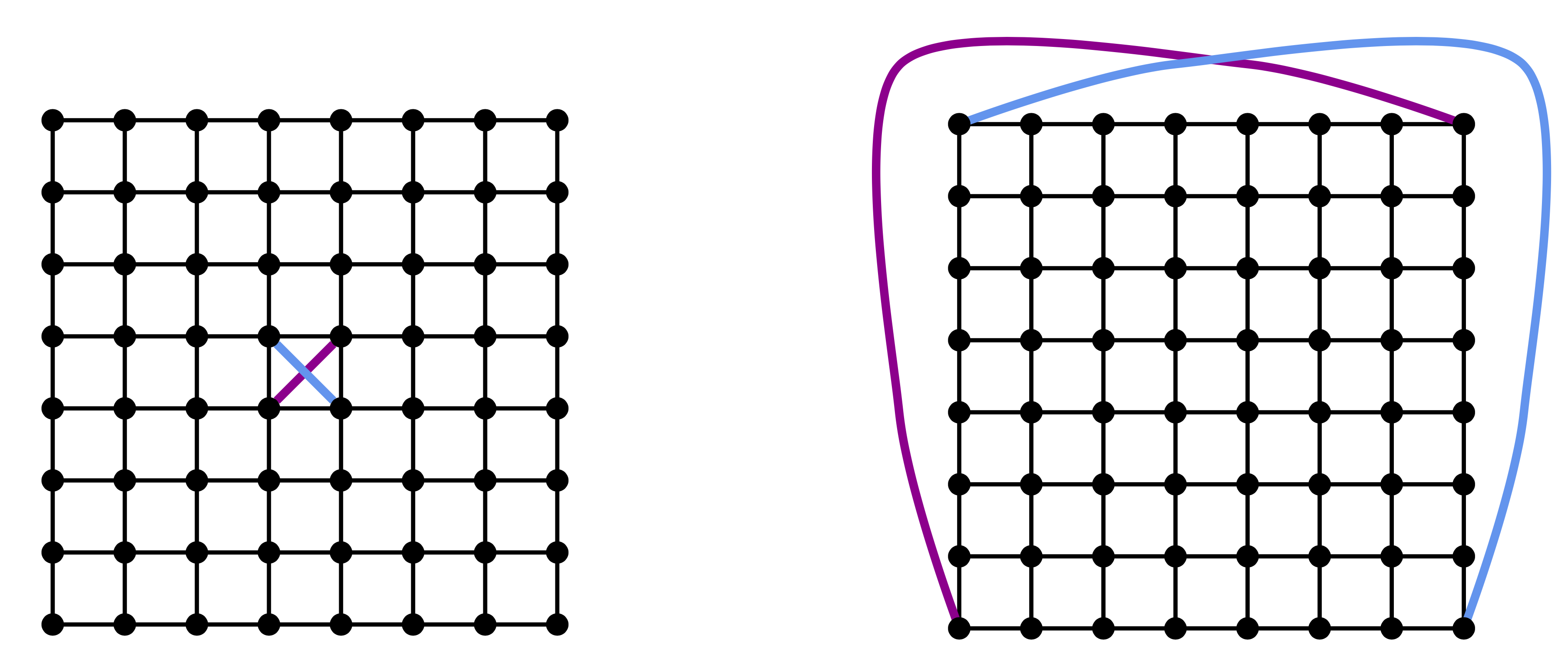}}}
	\caption{The single-crossing grid of order $4$ (left) and the inside-out crossing grid of order $4$ (right).}
	\label{fig_singlecrossinggrid}
\end{figure}

Let $n\in\N$ be some positive integer.
The \emph{single-crossing grid of order $n$} is the graph obtained from the $(2n \times 2n)$-grid by adding the edges $\Set{(n,n),(n+1,n+1)}$ and $\Set{(n,n+1),(n+1,n)}.$
And the \emph{inside-out crossing grid of order $n$} is the graph obtained from the $(2n\times 2n)$-grid by adding the edges $\Set{(1,1),(2n,2n)}$ and $\Set{(1,2n),(2n,1)}.$
For an illustration see \cref{fig_singlecrossinggrid}.

Observe that the single-crossing grid of order $n$ is a minor of the inside-out crossing grid of order $2n$ and vice versa.
Thus, up to a factor of $2,$ these two graphs can be seen as the same.
This is significant for our purpose since the single-crossing grid is usually seen as the universal pattern for single-crossing minors, but in most cases in our proof we will find an inside-out crossing grid as a minor.

A typical example of bipartite graphs that are matching minors of the single-crossing grids the Möbius ladder on $2(2k+1)$ vertices, for $k\geq 1.$   

\paragraph{Components of small treewidth.}

Observe that for every graph $G,$ every clique $H$ in $G$ and every tree-decomposition $(T,\beta)$ of $G,$ there exists some $t\in V(T)$ such that $V(H)\subseteq \beta(t).$
Hence, we immediately obtain the following observation.

\begin{observation}\label{obs_decomposequasi4components}
Let $G$ be a graph and let $(T,\beta)$ be the tree decomposition of $G$ from \cref{thm_4components}.
Let $t\in V(T)$ and $G_t$ be the torso of $G$ at $t.$
Moreover, let $(T',\beta')$ be a tree decomposition for $G_t.$
Then for every $dt\in E(T)$ there exists a vertex $v\in V(T')$ such that $\beta(t)\cap\beta(d)\subseteq\beta'(v).$
\end{observation}

This observation allows us to refine the decomposition $(T,\beta)$ from \cref{thm_4components} by replacing any vertex $t$ with a tree decomposition for the torso of $G$ at $t$ and joining the neighbours of $t$ in $T$ to the respective vertices of this new tree decomposition which contain the corresponding adhesion set.

\paragraph{A planar wall.}

Step iii) of our strategy requires us to find a large planar area within a quasi-$4$-connected SCM-free graph of huge treewidth.
Moreover, this area should come equipped with a grid-like infrastructure.

Let $H$ be a $(6k\times 6k)$-grid for some positive integer $k$ and let $(C_1,\dots,C_{3k})$ be its centred layering.
Now let, for every $i\in[1,3],$ $H_i$ be the subgraph of $H$ induced by $\bigcup_{j\in[(i-1)k+1,ik]}V(C_j).$
We call $(H_1,H_2,H_3)$ the \emph{tripartition} of $H.$

\begin{lemma}\label{lemma_jumpandcrossfreegrid}
Let $k\in\N$ be some positive integer.
Let $G=H+P$ be a graph where $H$ is a $(12k\times12k)$-grid and $P$ is a path with endpoints $a$ and $b$ such that $V(P)\cap V(H)=\Set{a,b}.$
Let $(H_1,H_2,H_3)$ be the tripartition and $(C_1,\dots,C_{6k})$ be the centred layering of $H.$
Suppose $a\in V(H_1).$
If there exists some $j\in[1,6k]$ for which $C_j$ separates $a$ and $b$ within $H,$ then $G$ contains the single-crossing grid of order $k$ as a minor.
\end{lemma}

\begin{proof}
We distinguish two cases: either $b\in V(H_1)\cup V(H_2)$ or $b\in V(H_3).$
In the first case let $G'$ be the graph induced by $V(H_1)\cup V(H_2)\cup V(P)$ and let $\Omega$ be an ordering of the vertices of $C_{4k}$ obtained by traversing $C_{4k}$ in clockwise direction in a plane embedding of $H$ with $C_{6k}$ as the cycle bounding the outer face.
Observe that the society $(G',\Omega)$ must have a cross, in particular, the two paths $L$ and $R$ that form the cross can be chosen such that $L$ has endpoints $(2k+1,2k+1)$ and $(10k,10k),$ while $R$ has endpoints $(2k+1,10k)$ and $(10k,2k+1).$
This cross can then be extended to be a cross on the cycle $C_{2k}$ while being internally disjoint from $H_3.$
Observe that we are now able to find the single-crossing matching grid of order $k$ as a minor by using $H_3$ together with the paths $L$ and $R.$

In case   $b\in V(H_3),$ let $H'$ be the subgraph of $H_2$ induced by the vertices of $\CondSet{(i,j)}{i,j\in[2k+1,4k]}$ and let $C$ be the cycle in $H'$ on the vertices of $V(H')\setminus\CondSet{(i,j)}{i,j\in[2k+2,4k-1]}.$
Moreover, let $\Omega$ be an ordering of $V(C)$ obtained by traversing along $C$ in clockwise direction in a plane embedding of $H'$ where $C$ bounds the outer face.
Finally, $H''\coloneqq H-(H'-C).$
Then the society $(H'',\Omega)$ has a cross consisting of the paths $L$ and $R$ such that $L$ has endpoints $(2k+1,2k+1)$ and $(4k,4k),$ while $R$ has endpoints $(2k+1,4k)$ and $(4k,2k+1).$
Observe that $H'+L+R$ contains the inside-out crossing grid of order $2k$ as a minor and thus it also contains the single-crossing grid of order $k$ as a minor.
\end{proof}

With \cref{lemma_jumpandcrossfreegrid} we are able to exclude any kind of ``long jump'' that intersects with the innermost part of the tripartition of a grid.

An elementary $k$-wall $W$ has a unique face whose boundary contains more than six vertices.
The \emph{perimeter} of an elementary $k$-wall is defined to be the subgraph of $W$ induced by all vertices that lie on the unique face with more than six vertices.
Now let $W'$ be a $k$-wall obtained by subdividing each edge of $W$ an arbitrary (possibly zero) number of times.
The \emph{perimeter} of $W',$ denoted by $\Perimeter{W'},$ is the subgraph of $W'$ induced by the vertices of the perimeter of $W$ together with the subdivision vertices of the edges of the perimeter of $W.$
Observe that for elementary $2k$-walls we may extended our definition of centred layerings and for elementary $6k$-walls we may also extend the notion of the tripartition in the natural way (see \cref{fig_tripartitionofawall} for an illustration).
These definitions then carry over to general $6k$-walls.

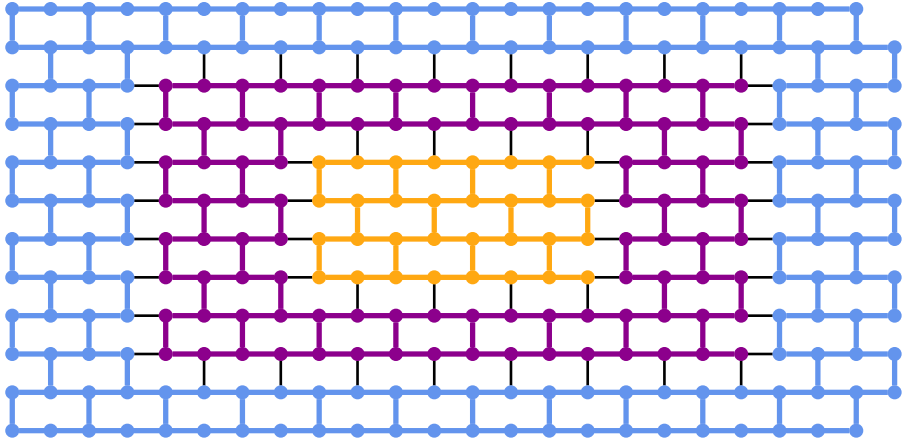
\begin{figure}
		\centering
\makefast{% !TEX root = ../single_comb.tex
% !TeX spellcheck = en_US	
	\begin{tikzpicture}[scale=0.85]
			\pgfdeclarelayer{background}
			\pgfdeclarelayer{foreground}
			\pgfsetlayers{background,main,foreground}
			\node[v:ghost] (C) {};
			
			\foreach \x in {1,...,23}
			\foreach \y in {1,12}
			{	
				%			\pgfmathsetmacro\X{}
				\node[v:main,color=CornflowerBlue] (v_\x_\y) at (\x*0.6,\y*0.6){};
			}
		
			\foreach \x in {1,...,24}
			\foreach \y in {2,11}
			{	
				%			\pgfmathsetmacro\X{}
				\node[v:main,color=CornflowerBlue] (v_\x_\y) at (\x*0.6,\y*0.6){};
			}
		
			\foreach \x in {1,...,4}
			\foreach \y in {3,...,10}
			{	
				%			\pgfmathsetmacro\X{}
				\node[v:main,color=CornflowerBlue] (v_\x_\y) at (\x*0.6,\y*0.6){};
			}
		
			\foreach \x in {21,...,24}
			\foreach \y in {3,...,10}
			{	
				%			\pgfmathsetmacro\X{}
				\node[v:main,color=CornflowerBlue] (v_\x_\y) at (\x*0.6,\y*0.6){};
			}
		
			\foreach \x in {5,...,20}
			\foreach \y in {3,4,9,10}
			{	
				%			\pgfmathsetmacro\X{}
				\node[v:main,color=DarkMagenta] (v_\x_\y) at (\x*0.6,\y*0.6){};
			}
			
			\foreach \x in {5,...,8}
			\foreach \y in {5,...,8}
			{	
				%			\pgfmathsetmacro\X{}
				\node[v:main,color=DarkMagenta] (v_\x_\y) at (\x*0.6,\y*0.6){};
			}
		
			\foreach \x in {17,...,20}
			\foreach \y in {5,...,8}
			{	
				%			\pgfmathsetmacro\X{}
				\node[v:main,color=DarkMagenta] (v_\x_\y) at (\x*0.6,\y*0.6){};
			}
			
			\foreach \x in {9,...,16}
			\foreach \y in {5,...,8}
			{	
				%			\pgfmathsetmacro\X{}
				\node[v:main,color=DarkTangerine] (v_\x_\y) at (\x*0.6,\y*0.6){};
			}
			
			\begin{pgfonlayer}{background}
				
				\foreach \x in {1,12}
				{
					\draw[e:thick,color=CornflowerBlue] (v_1_\x) to (v_23_\x);
				}
			
				\foreach \x in {2,11}
				{
					\draw[e:thick,color=CornflowerBlue] (v_1_\x) to (v_24_\x);
				}
			
				\foreach \y in {3,...,10}
				{
					\draw[e:thick,color=CornflowerBlue] (v_1_\y) to (v_4_\y);
					\draw[e:main] (v_4_\y) to (v_5_\y);
					\draw[e:thick,color=CornflowerBlue] (v_21_\y) to (v_24_\y);
					\draw[e:main] (v_20_\y) to (v_21_\y);
				}
				
				\foreach \y in {3,4,9,10}
				{
					\draw[e:thick,color=DarkMagenta] (v_5_\y) to (v_20_\y);
				}
			
				\foreach \y in {5,...,8}
				{
					\draw[e:thick,color=DarkMagenta] (v_5_\y) to (v_8_\y);
					\draw[e:main] (v_8_\y) to (v_9_\y);
					\draw[e:thick,color=DarkMagenta] (v_17_\y) to (v_20_\y);
					\draw[e:main] (v_16_\y) to (v_17_\y);
				}
			
				\foreach \y in {5,...,8}
				{
					\draw[e:thick,color=DarkTangerine] (v_9_\y) to (v_16_\y);
				}
			
				\foreach \x in {1,3,21,23}
				\foreach \y in {1,3,5,7,9,11}
				{	
					\pgfmathsetmacro\next{\y+1}
					\draw[e:thick,color=CornflowerBlue] (v_\x_\y) to (\x*0.6,\next*0.6);
				}
			
				\foreach \x in {2,4,22,24}
				\foreach \y in {2,4,6,8,10}
				{	
					\pgfmathsetmacro\next{\y+1}
					\draw[e:thick,color=CornflowerBlue] (v_\x_\y) to (\x*0.6,\next*0.6);
				}
			
				\foreach \x in {5,7,9,11,13,15,17,19}
				{
					\draw[e:thick,color=CornflowerBlue] (v_\x_1) to (v_\x_2);
					\draw[e:thick,color=CornflowerBlue] (v_\x_11) to (v_\x_12);
				}
			
				\foreach \x in {9,11,13,15}
				{
					\draw[e:thick,color=DarkMagenta] (v_\x_3) to (v_\x_4);
					\draw[e:thick,color=DarkMagenta] (v_\x_9) to (v_\x_10);
				}
			
				\foreach \x in {5,7,17,19}
				\foreach \y in {3,5,7,9}
				{	
					\pgfmathsetmacro\next{\y+1}
					\draw[e:thick,color=DarkMagenta] (v_\x_\y) to (\x*0.6,\next*0.6);
				}
				
				\foreach \x in {6,8,18,20}
				\foreach \y in {4,6,8}
				{	
					\pgfmathsetmacro\next{\y+1}
					\draw[e:thick,color=DarkMagenta] (v_\x_\y) to (\x*0.6,\next*0.6);
				}
				
				\foreach \x in {9,11,13,15}
				{
					\draw[e:thick,color=DarkTangerine] (v_\x_5) to (v_\x_6);
					\draw[e:thick,color=DarkTangerine] (v_\x_7) to (v_\x_8);
				}
			
				\foreach \x in {10,12,14,16}
				{
					\draw[e:thick,color=DarkTangerine] (v_\x_6) to (v_\x_7);
				}
			
				\foreach \x in {6,8,10,12,14,16,18,20}
				{
					\draw[e:main] (v_\x_2) to (v_\x_3);
					\draw[e:main] (v_\x_10) to (v_\x_11);
				}
			
				\foreach \x in {10,12,14,16}
				{
					\draw[e:main] (v_\x_4) to (v_\x_5);
					\draw[e:main] (v_\x_8) to (v_\x_9);
				}

			\end{pgfonlayer}
		\end{tikzpicture}}{\scalebox{.17}{\includegraphics{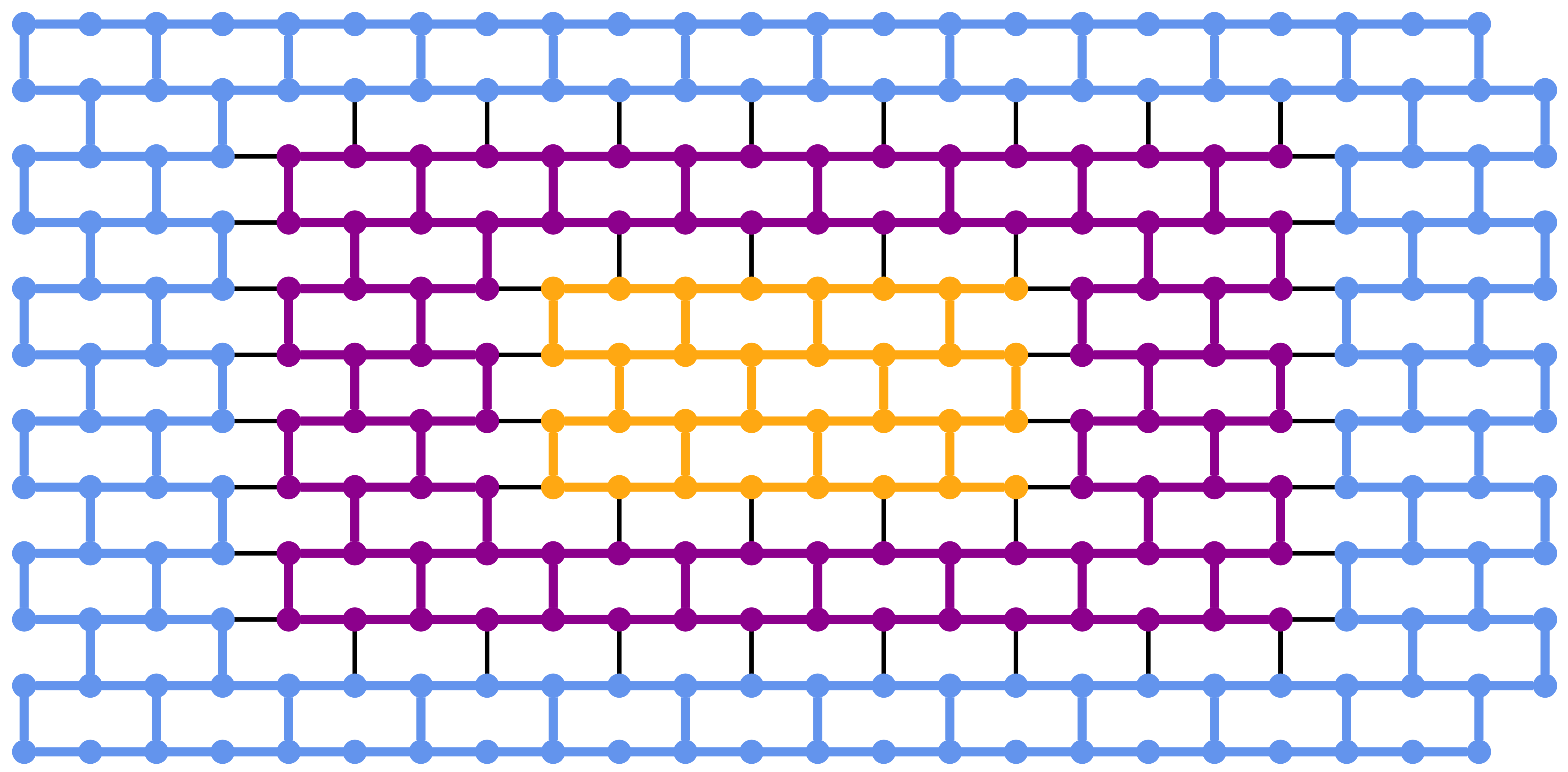}}}
	\caption{The tripartition of an elementary $12$-wall. The graphs $H_{1},$ $H_{2},$ and $H_{3}$ are drawn in orange, purple, and blue respectively.}
	\label{fig_tripartitionofawall}
\end{figure}

Let $G$ be a graph and $W$ be a wall in $G.$
The \emph{compass} of $W$ in $G,$ denoted by $\CompassG{G}{W},$ is the subgraph of $G$ induced by the vertices of $\Perimeter{W}$ together with the vertices of the unique component of $G-\Perimeter{W}$ that contains $W-\Perimeter{W}.$

\begin{lemma}\label{lemma_separatingcycle}
	Let $k\in \N$ be some integer and let $G$ be a quasi-$4$-connected graph.
	Next, let $W$ be a $(12k+6)$-wall in $G$ together with its tripartition $(W_1,W_2,W_3).$
	Then either $G$ contains the single-crossing grid of order $k$ as a minor or there is a separation $(A,B)$ in $G$ such that $V(W_1)\subseteq A,$ $\Perimeter{W}\subseteq B,$ and $\Perimeter{W_1}=A\cap B.$
\end{lemma}

\begin{proof}
Let $W_1'$ be the maximal subwall of $W_1$ contained in $W_1-\Perimeter{W_1}.$
Let us assume that there is no separation $(A,B)$ in $G$ such that $V(W_1)\subseteq A,$ $\Perimeter{W}\subseteq B,$ and $\Perimeter{W_1}=A\cap B.$
In this case there must exist a path $Q'$ with one endpoint in $W_1-\Perimeter{W_1},$ the other endpoint in $\Perimeter{W},$ and which is internally disjoint from $W_1.$
Moreover, this means we may find a path $Q$ with one endpoint, say $a$ in $W_1'$ and the other one, say $b,$ in $W-W_1$ which is internally disjoint from $W.$
In particular, $Q$ must be disjoint from $\Perimeter{W_1}.$
By \cref{lemma_jumpandcrossfreegrid} this means that $G$ contains the single-crossing grid of order $k$ as a minor and our claim follows.
\end{proof}

\begin{lemma}\label{lemma_planarwall}
Let $k\in \N$ be some integer and let $G$ be a quasi-$4$-connected graph.
Next, let $W$ be a $(12k+6)$-wall in $G$ together with its tripartition $(W_1,W_2,W_3)$ and let $W_1'$ be the maximal subwall of $W_1$ contained in $W_1-\Perimeter{W_1}.$
Then either $G$ contains the single-crossing grid of order $k$ as a minor, or $\CompassG{G}{W_1'}$ is a planar graph where $\Perimeter{W_1'}$ bounds a face.
\end{lemma}

\begin{proof}
Let $(C_1,\dots,C_{6k+3})$ be the centred layering of $W.$
By \cref{lemma_separatingcycle} we may assume there is a separation $(A,B)$ in $G$ such that $V(W_1)\subseteq A,$ $\Perimeter{W}\subseteq B,$ and $\Perimeter{W_1}=A\cap B.$
Hence, $W-W_1$ is disjoint from $\CompassG{G}{W_1}.$

Let $X\subseteq V(\Perimeter{W_1})$ be the set of vertices with neighbours in $V(W-W_1)$ within $W$ and let us fix an ordering $\Omega$ of $X$ obtained by traversing $\Perimeter{W_1}$ in clockwise order.
Note that this means that all vertices in $X$ have degree three in the graph $W.$
Then $(\CompassG{G}{W_1},\Omega)$ is a society.
Observe that, in case $(\CompassG{G}{W_1},\Omega)$ has a cross, we find the single-crossing grid of order $k$ as a minor.
Hence, $(\CompassG{G}{W_1},\Omega)$ has a vortex-free rendition $\rho$ in the disk by \cref{thm_twopaths}.
Moreover, we may assume for all $c\in C(\rho)$ with $\Abs{\widetilde{c}}=3$ that $\Abs{\sigma(c)}< 5.$

Next let $(U,T)$ be a separation of order at most three in $\CompassG{G}{W_1}$ such that in case $\Abs{U\cap T}=3$ we have $\min\Set{\Abs{U\setminus T},\Abs{T\setminus U}}\geq 2.$
Observe that $(U,T)$ can always be chosen such that $T\setminus U$ contains $C_i$ for some $i\in[1,2k+1],$ while $U$ does not contain any cycle of $\Set{C_1,\dots,C_{2k+1}}$ completely.
Suppose $U\cap T$ contains at most one vertex of $\Perimeter{W_1}$ and let $i$ be chosen as above.
Let $u\in U\setminus T$ be any vertex.
It follows that, in $G,$ every $u$-$C_i$-path must contain a vertex of $U\cap T.$
Moreover, there exist at least two distinct choices of $u$ and thus the existence of $(U,T)$ contradicts the fact that $G$ is quasi-$4$-connected.
It follows that $\rho$ contains a rendition $\rho'$ of $\CompassG{G}{W_1'}$ such that $\Abs{\widetilde{c}}\leq 2$ and $\sigma(c)\setminus \pi(\widetilde{c})=\emptyset$ for all cells $c\in C(\rho').$
Hence, $\CompassG{G}{W_1'}$ is a planar graph and one of its faces must be bounded by $\Perimeter{W_1'}.$
\end{proof}

\paragraph{Outside a planar wall.}

In \cref{lemma_planarwall} we have seen that within any large enough wall $W$ in a quasi-$4$-connected graph $G$ we may either find a large single-crossing grid as a minor, or the compass of a substantial portion of $W$ is planar.
The next step is to show that the existence of $W$ implies either the existence of a large single-crossing grid as a minor, or the entirety of $G$ is a planar graph.

\begin{lemma}\label{lemma_planaroutside}
	Let $k\in \N$ be some integer and let $G$ be a quasi-$4$-connected graph.
	Next, let $W$ be a $12(k+1)$-wall in $G$ together with its tripartition $(W_1,W_2,W_3)$ and let $H$ be the component of $G-\Perimeter{W_1}$ that contains $W_1-\Perimeter{W_1}.$
	Then either $G$ contains the single-crossing grid of order $k$ as a minor, or $G-H$ is a planar graph where $\Perimeter{W_1}$ bounds a face.
\end{lemma}

\begin{proof}
As before, by \cref{lemma_separatingcycle} we may assume there is a separation $(A,B)$ in $G$ such that $V(W_1)\subseteq A,$ $\Perimeter{W}\subseteq B,$ and $\Perimeter{W_1}=A\cap B.$
Hence, $W-W_1$ is disjoint from $\CompassG{G}{W_1}.$

Let $W_1'$ be the subwall of $W_1$ contained in $W_1-\Perimeter{W_1},$ let $H'\coloneqq \CompassG{G}{W_1'}-\Perimeter{W_1'},$ and let $X$ be the set of all vertices of $\Perimeter{W_1'}$ with neighbours in $W-W_1'$ within $W.$
Finally, let $\Omega$ be the cyclic ordering of $X$ obtained by traversing $\Perimeter{W_1'}$ in clockwise direction and let $\Complement{H}\coloneqq H-H'.$
If $(\Complement{H},\Omega)$ has a cross we find the inside-out-crossing grid of order $2k$ as a minor, and thus we have found our single-crossing grid of order $k.$
Thus, we may assume that $(\Complement{H},\Omega)$ has a vortex-free rendition $\rho$ in the disk by \cref{thm_twopaths}.
Moreover, we may assume for all $c\in C(\rho)$ with $\Abs{\widetilde{c}}=3$ that $\Abs{\sigma(c)}< 5.$

Similar to the proof of \cref{lemma_planarwall} we observe that there does not exist a separation $(U,T)$ of order at most three in $\Complement{H}$ such that $\Abs{U\cap T\cap V(\Perimeter{W_1'})}\leq 1$ and if $\Abs{U\cap T}=3$ we have $\min\Set{\Abs{U\setminus T},\Abs{T\setminus U}}\geq 2.$
It follows that $\rho$ contains a rendition $\rho'$ of $H\subsetneq \Complement{H}$ such that $\Abs{\widetilde{c}}\leq 2$ and $\sigma(c)\setminus \pi(\widetilde{c})=\emptyset$ for all cells $c\in C(\rho').$
Hence, $H$ is a planar graph and one of its faces must be bounded by $\Perimeter{W_1}.$
\end{proof}

Combining both \cref{lemma_planarwall} and \cref{lemma_planaroutside} yields the following ``local'' structural result.
To see this take the $12(k+1)$-wall from \cref{lemma_planaroutside} and choose a slightly smaller wall to apply \cref{lemma_planarwall} for the same cycle.
The results are two planar graphs, each with a face bounded by a cycle in which both graphs coincide.
Hence, both graphs can be seen as one half of a sphere each and their combination yields one planar graph -- or we find the single-crossing grid as a result of the two lemmas.

\begin{corollary}\label{cor_planarbags}
Let $k\in \N$ be some integer and let $G$ be a quasi-$4$-connected graph.
If $G$ contains a $12(k+1)$-wall, then either $G$ contains the single-crossing grid of order $k$ as a minor, or $G$ is planar.
\end{corollary}

\paragraph{Proof of the structure theorem.}

We state the Grid Theorem in the form of Chuzhoy and Tan who currently provide the best   bound on the function.
Moreover, since we are going to work with walls, which have the advantage that they can be found as subgraphs and not just as minors, we state the theorem in its wall version.
Note that any grid contains a wall of half the order and any wall contains a grid minor of the same order, thus restricting ourselves to walls does not change the order of the functions involved.

\begin{proposition}[\cite{chuzhoy2021towards}]\label{thm_undirectedgridtheorem}
	There exists a function $\mathsf{g}\colon\N\rightarrow\N$ with $\mathsf{g}(n)\in\mathcal{O}(n^9\operatorname{poly log}(n))$ such that for every $k\in \N$ and every graph $G,$ if $\tw{G}>\mathsf{g}(k)$ then $G$ contains a $k$-wall as a subgraph.
\end{proposition}

This theorem is the remaining ingredient for the structure theorem for SCM-free graphs.
The following is a refined version of the original result by R\&S on graphs excluding a single-crossing minor \cite{robertson1993excluding}.
In particular, we present exact bounds, depending only on the graph $H$ which is excluded and the function of the grid theorem.

\begin{theorem}\label{thm_singlecrossingwithbounds}
	Let $k\in\N$ be some integer and $H$ be a single-crossing minor which is a minor of the single-crossing grid of order $k.$
	Moreover, let $\mathsf{g}$ be the function from \cref{thm_undirectedgridtheorem}.
	Then every $H$-minor-free graph $G$ has a tree decomposition $(T,\beta)$ such that for every $t\in V(T),$ if $\Abs{\beta(t)}> \mathsf{g}(12(k+1))+1$ then for all $dt\in E(T)$ we have $\Abs{\beta(t)\cap\beta(d)}\leq 3$ and the torso of $G$ at $t$ is a planar quasi-$4$-component of $G.$
\end{theorem}

\begin{proof}
	Let $G$ be an $H$-minor-free graph.
	We start by applying \cref{thm_4components} to obtain a tree decomposition $(T',\beta')$ of $G$ into its quasi-$4$-components such that the adhesion between any two neighbouring bags is at most three.
	For each $t\in V(T')$ where $\Abs{\beta'(t)}>\mathsf{g}(12(k+1))+1$ and the torso $G_t$ of $G$ at $t$ has treewidth at most $\mathsf{g}(12(k+1)),$ we may apply \cref{obs_decomposequasi4components} to refine $(T',\beta')$ and finally obtain a tree decomposition $(T,\beta)$ where for all $t\in V(T)$ we have that, if $\Abs{\beta(t)}>\mathsf{g}(12(k+1))+1$ then the torso $G_t$ of $G$ at $t$ is a quasi-$4$-component of $G$ with treewidth more that $\mathsf{g}(12(k+1))$ and $\Abs{\beta(t)\cap\beta(d)}\leq 3$ for all $dt\in E(T).$
	Let $t\in V(T)$ be a vertex with $\Abs{\beta(t)}>\mathsf{g}(12(k+1))+1$ and let $G_t$ be the torso of $G$ at $t.$
	Then by \cref{thm_undirectedgridtheorem}, $G_t$ contains a $12(k+1)$-wall.
	Since $G$ is $H$-minor-free and $G_t$ is a minor of $G,$ $G_t$ cannot contain the single-crossing grid of order $k$ as a minor and thus \cref{cor_planarbags} implies the planarity of $G_t.$
\end{proof}

%%%%%%%%%%%%%%%%%%%%%%%%%%%%%%%%%%%%%%%%%%%%%%%%%%%%%%%%%%%%%%%%%%%%%%%%%%%%%%%%%%%%%%%%%%%%%%%%%%%%%%%%%%%%%%%%%%%%%%%%%%%%%%%%%%%%%%%%%%%%%%%%%%%%%%%%%%%%%%%%%%%%%%%%%%%%%%%%%%%%%%%%%%%%%%%%%%%%%%%%%%%%%%%%%%%%%%%%%%%%%%%%%%%%%%%%%%%%%%%%%%%%%%%%%%%%%%%%%%%%%%%%%%%%%%%%%%%%%%%%%%%%%%%%

\section{Matching theoretic background}\label{sec_background}

Before we start, let us take this section to formally introduce the matching theoretic concepts and tools we will use in the remainder of this paper.

\paragraph{Conventions for bipartite graphs.}

Since we are working almost exclusively on bipartite graphs, it makes sense to fix a bit of notation.
In most cases we will use $B$ to denote a bipartite graph, while $G$ generally denotes graphs which can be non-bipartite as well.

At times, we might break this convention if we run out of letters.

We assume that every bipartite graph $B$ is given together with a vertex partition into two sets $V_1$ and $V_2$ such that both of these sets induce independent sets in $B.$
The vertices of $V_1$ are always depicted as the black (or filled) vertices in our figures, while the vertices of $V_2$ are depicted as the white (or empty) ones.
For each $i\in[2]$ and any bipartite graph $B$ or set of vertices $X\subseteq V(B),$ we write $V_i(B),$ or $V_i(X),$ to denote $V_i\cap V(B),$ or $V_i\cap X$ respectively.
If the graph $B$ is understood from the context, we usually just write $V_1$ and $V_2$ instead of $V_1(B)$ and $V_2(B).$

\subsection{Basic definitions}
%
%\sw{Some sentence to lead into the paragraph}

We now introduce a series of concepts from matching theory.

\paragraph{Matching covered graphs.}

Let $G$ be a graph and $F\subseteq E(G).$
By $V(F)$ we denote the set $\bigcup_{e\in F}e$ of all endpoints of the edges in $F.$
We say that $F$ is a \emph{matching} if the edges in $F$ are pairwise-disjoint.
A matching $M\subseteq E(G)$ \emph{covers} a vertex $v\in V(G)$ if $v\in V(M)$ and a vertex $u\in V(G)$ is said to be \emph{exposed} by $M$ if $u\in V(G)\setminus V(M).$
A matching $M\subseteq E(G)$ is \emph{perfect} if it covers all vertices of $G.$
We denote by $\Perf{G}$ the set of all perfect matchings of $G.$

Let $G$ be a graph.
A set $X\subseteq V(G)$ is \emph{conformal} if $G-X$ has a perfect matching.
If $M$ is a perfect matching of $G$ and $\E{G-X}\cap M$ is a perfect matching of $G-X,$ we call $X$ \emph{$M$-conformal}.
A subgraph $H$ of $G$ for which $\V{H}$ is ($M$-)conformal is called \emph{($M$-)conformal}.
An edge $e\in\E{G}$ is \emph{admissible} if there exists a perfect matching $M'$ of $G$ with $e\in M'.$
A graph $G$ is \emph{matching covered} if it is connected and all of its edges are admissible.

We begin by defining an operation, together with a decomposition, that can be seen as an analogue of the block-decomposition\footnote{That is the tree-like decomposition of a graph into its maximal subgraphs without cut-vertices.} of graphs.
However, due to the more global structure of perfect matchings one needs to use a slightly more complicated operation than just taking subgraphs.
Hence, the decomposition, in spirit, is closer to \cref{thm_4components}.

\begin{definition}[Tight cuts and braces]\label{@inconvinientes}
	Let $G$ be a graph and $X\subseteq\V{G}.$
	We denote by $\Cut{X}$ the set of edges in $G$ with exactly one endpoint in $X$ and call $\Cut{X}$ the \emph{edge cut} around $X$ in $G.$
	Let us denote by $\Perf{G}$ the set of all perfect matchings in $G.$
	An edge cut $\Cut{X}$ is \emph{tight} if $\Abs{\Cut{X}\cap M}=1$ for all perfect matchings $M\in\Perf{G}.$
	If $\Cut{X}$ is a tight cut and $\Abs{X},\Abs{\V{G}\setminus X}\geq 2,$ it is \emph{non-trivial}.
	Identifying the \emph{shore} $X$ of a non-trivial tight cut $\Cut{X}$ into a single vertex is called a \emph{tight cut contraction} and the resulting graph $G'$ can be seen to be matching covered again.
	A bipartite matching covered graph without non-trivial tight cuts is called a \emph{brace} while a non-bipartite graph without non-trivial tight cuts is called a \emph{brick}.	
\end{definition}

The following is an observation first made by Lov\'asz (see the proof of Lemma 1.4 in \cite{lovasz1987matching}).

\begin{observation}\label{obs_tightcutminoritymajority}
	Let $B$ be a bipartite and matching covered graph and let $X\subseteq\V{B}$ be a set of vertices with $\Abs{X},\Abs{\V{B}\setminus X}\geq 3.$
	Then $\CutG{}{X}$ is a tight cut if and only if there exists $i\in[2]$ such that $\Abs{X\cap V_i}=\Abs{X\setminus V_i}-1,$ and $\Fkt{N_B}{X\cap V_i}\subseteq X.$
	We say that $V_i$ is the \emph{minority} of $X$ and $V_{3-i}$ is the \emph{majority} of $X.$
\end{observation}

One can observe that, if $X\subseteq Y\subseteq V(B)$ are two sets such that $\CutG{B}{X}$ and $\CutG{B}{Y}$ are both non-trivial tight cuts, then $\CutG{B_1}{X}$ is a tight-cut in the graph $B_1$ obtained by contracting $\Complement{Y}$ in $B.$
Hence, a maximal family of pairwise laminar non-trivial tight cuts in a matching covered graph $G$ defines a structure that decomposes $G$ in a tree-like way such that every vertex of this decomposition corresponds to a brick or brace of $G.$
We refer to this structure as the \emph{tight cut decomposition} of $G.$

It follows from a famous result of Lov\'asz \cite{lovasz1987matching} that the braces of a bipartite matching covered graph $B$ are uniquely determined.

\begin{definition}[Extendability]\label{@commemorative}
	Let $k$ be a positive integer.
	A graph $G$ is called \emph{$k$-extendable} if it has at least $2k+2$ vertices and for every matching $F\subseteq\E{G}$ of size $k$ there exists a perfect matching $M$ of $G$ with $F\subseteq M.$
\end{definition}

\begin{proposition}[\cite{lovasz2009matching}]\label{thm_braces}
	A bipartite graph $B$ is a brace if and only if it is either isomorphic to $C_4,$ or it is $2$-extendable.
\end{proposition}

The following theorem is a collection of several different characterisations of $k$-extendability in bipartite graphs.

\begin{definition}[Alternating path]\label{@sufficiently}
	Let $G$ be a graph with a perfect matching $M.$
	A path $P$ in $G$ is \emph{$M$-alternating} if there exists a set $S\subseteq\V{P}$ of endpoints of $P$ such that $P-S$ is a conformal subgraph of $G$, and we say that $P$ is \emph{alternating} if there exists a perfect matching $M$ of $G$ such that $P$ is $M$-alternating.
	$P$ is $M$-conformal if $S=\emptyset$ and $P$ is \emph{internally $M$-conformal} if $S$ contains both endpoints of $P.$
\end{definition}

\begin{proposition}[\cite{plummer1986matching,aldred2003m}]\label{thm_bipartiteextendibility}
	Let $B$ be a bipartite graph and $k\in\N$ a positive integer.
	The following statements are equivalent.
	\begin{enumerate}
		\item $B$ is $k$-extendable.
		\item $\Abs{V_1}=\Abs{V_2},$ and for all non-empty $S\subseteq V_1,$ $\Abs{\NeighboursG{B}{S}}\geq \Abs{S}+k.$
		\item For all sets $S_1\subseteq V_1$ and $S_2\subseteq V_2$ with $\Abs{S_1}=\Abs{S_2}\leq k$ the graph $B-S_1-S_2$ has a perfect matching.
		\item There is a perfect matching $M\in\Perf{B}$ such that for every $v_1\in V_1,$ every $v_2\in V_2$ there are $k$ pairwise internally disjoint internally $M$-conformal paths with endpoints $v_1$ and $v_2.$
		\item\label{more_stuf_con} For every perfect matching $M\in\Perf{B},$ every $v_1\in V_1,$ and  every $v_2\in V_2$ there are $k$ pairwise internally disjoint internally $M$-conformal paths with endpoints $v_1$ and $v_2.$
	\end{enumerate}
\end{proposition}

While ii) can be seen as a generalisation of Hall's Theorem, statements iv) and v) are more similar to Menger's Theorem.

\paragraph{Matching minors.}

An important restriction of tight cut contractions gives rise to the definition of matching minors.
Let $G$ be some matching covered graph and let $v\in V(G)$ be some vertex of degree two together with its two neighbours $N_{G}(v)=\{u_0,u_1\}.$
Notice that $\CutG{G}{\{ v,u_1,u_1\}}$ is a tight cut since every perfect matching of $G$ must use one of the two edges $vu_0$ or $vu_1.$
We call the operation of contracting $\{v,u_1,u_2 \}$ into a single vertex a \emph{bicontraction}.
Please note that all graphs in this paper are considered to be simple, and thus we remove parallel edges and loops wherever those might arise from contractions.

\begin{definition}[Matching minor]\label{def_matchingminor}
Let $G$ be a graph with a perfect matching.
A graph $H$ is a \emph{matching minor} of $G$ if it can be obtained from a conformal subgraph of $G$ by bicontractions.

Let $M$ be a perfect matching of $G.$
We say that $H$ is an \emph{$M$-minor} of $G$ if $H$ can be obtained from an $M$-conformal subgraph $H'$ of $G$ by repeated bicontractions where $M$ contains a perfect matching of $H'.$
\end{definition}

Similar to how every $2$-connected minor of a graph $G$ must be a minor of one of its blocks, every brace that is a matching minor of some bipartite matching covered graph $B$ is a matching minor of some brace of $B$ \cite{lucchesi2015thin}.

There exists a matching theoretic analogue of topological minors, or subdivisions, for the matching theoretic setting.
A graph $H$ is a \emph{bisubdivision} of some graph $H'$ if it can be obtained from $H'$ by subdividing every edge an even number of times.
Please note that we explicitly allow zero subdivisions in this definition.
A graph $G$ with a perfect matching contains a \emph{conformal bisubdivision} of a graph $H$ if $G$ has a conformal subgraph isomorphic to some bisubdivision of $H.$
For graphs of maximum degree three, there is no difference between the containment as a conformal bisubdivision or as a matching minor. 
A graph $G$ is said to be \emph{matching covered} if it is connected and each of its edges belongs to at least one perfect matching of $G.$

\begin{lemma}[\cite{lucchesi2018two}]\label{lemma_confmathingminors}
	Let $G$ and $H$ be matching covered graphs such that $\MaximumDegree{H}=3.$
	Then $G$ contains a conformal bisubdivision of $H$ if and only if it contains $H$ as a matching minor.
\end{lemma}

\subsection{Pfaffian graphs and the Two Paths Theorem}

\paragraph{The \textsc{Pfaffian Recognition} problem.}

To obtain a matching theoretic version of the \hyperref[thm_singlecrossingwithbounds]{structure theorem for SCM-free graphs} we will need a matching theoretic version of the \hyperref[thm_twopaths]{societal Two Paths Theorem}.
To understand where we derive such a theorem from, and particularly what kind of reductions we allow, we need to know a bit about the structure of matching covered bipartite graphs that exclude $K_{3,3}$ as a matching minor.

The graph $K_{3,3}$ plays a key role in the theory of matching minors.
It was found to be the singular obstruction, in the sense of matching minors, for a bipartite graph to have a Pfaffian Orientation relatively early \cite{kasteleyn1967graph}.
At the time however this did not yield a solution for the \textsc{Pfaffian Recognition} problem as no algorithm was known to check for the presence of a specific matching minor.
In the following $30$ years many equivalent problems would be discovered by various authors (see \cite{mccuaig2004polya} for a good overview) but it took a complete structural characterisation of bipartite graphs without a $K_{3,3}$-matching minor that resembles similar results from (regular) minor theory such as, for example, Wagner's description of $K_5$-minor-free graphs \cite{wagner1937eigenschaft}.

Given a bipartite graph $B$ with a perfect matching we say that $B$ is \emph{non-Pfaffian} if it contains $K_{3,3}$ as a matching minor.
If $B$ does not contain $K_{3,3}$ we say that $B$ is \emph{$K_{3,3}$-free} or \emph{Pfaffian}\footnote{We give a formal introduction to Pfaffian Orientations in \cref{sec_countingperfectmatchings}. For this part of the paper the structural definition suffices.}.

\begin{definition}[$4$-Cycle Sum]\label{@denomination}
	For every $i\in\Set{1,2,3}$ let $B_i$ be a bipartite graph with a perfect matching and $C_i$ be a conformal cycle of length four in $B_i.$
	A \emph{$4$-cycle-sum} of $B_1$ and $B_2$ at $C_1$ and $C_2$ is a graph $B'$ obtained by identifying $C_1$ and $C_2$ into the cycle $C'$ and possibly forgetting some of its edges.
	If a bipartite graph $B''$ is a $4$-cycle-sum of $B'$ and some bipartite and matching covered graph $B_3$ at $C'$ and $C_3,$ then $B''$ is called a \emph{trisum} of $B_1,$ $B_2$ and $B_3.$
\end{definition}

The \emph{Heawood graph} is the bipartite graph associated with the incidence matrix of the Fano plane, see \cref{fig_heawood} for an illustration.
It is the singular exceptional graph in the structure theorem for $K_{3,3}$-free braces, similar to how Wagner's graph is the only non-planar graph necessary to describe the structure of all $K_5$-minor-free graphs.

\begin{figure}[!ht]
	\centering
\makefast{% !TEX root = ../single_comb.tex
% !TeX spellcheck = en_UK	

\begin{tikzpicture}[scale=0.9]
		\pgfdeclarelayer{background}
		\pgfdeclarelayer{foreground}
		\pgfsetlayers{background,main,foreground}
		
		\foreach \x in {2,4,6,8,10,12,14}
		{
			\node[v:main] () at (\x*25.71:15mm){};
		}
		
		\foreach \x in {1,3,5,7,9,11,13}
		{
			\node[v:mainempty] () at (\x*25.71:15mm){};
		}
		
		\begin{pgfonlayer}{background}
			\foreach \x in {2,4,6,8,10,12,14}
			{
				\draw[e:mainthin] (\x*25.71:15mm) to (128.55+\x*25.71:15mm);
			}
			
			\foreach \x in {1,...,14}
			{
				\draw[e:main] (\x*25.71:15mm) to (25.71+\x*25.71:15mm);
			}
		\end{pgfonlayer}
	\end{tikzpicture}}{
			\scalebox{.46}{\includegraphics{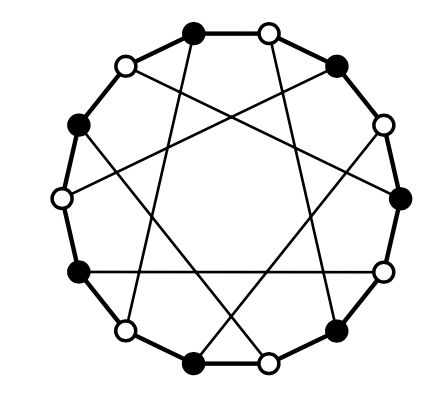}}}

	\caption{The Heawood graph $H_{14}.$}
	\label{fig_heawood}
\end{figure}
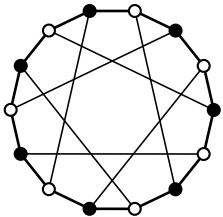

\begin{proposition}\cite{mccuaig2004polya,robertson1999permanents}\label{thm_trisums}
	A brace is $K_{3,3}$-free if and only if it either is isomorphic to the Heawood graph, or it can be obtained from planar braces by repeated application of the \hyperref[def:cyclesum]{trisum} operation.
\end{proposition}

\begin{corollary}\cite{mccuaig2004polya,robertson1999permanents}\label{cor_pfaffianalg}
	There exists an algorithm that decides, given a brace $B$ as input, whether $B$ contains $K_{3,3}$ as a matching minor in time $\Fkt{\mathcal{O}}{\Abs{\V{B}}^3}.$
\end{corollary}

\paragraph{A matching theoretic two paths theorem.}

For our purposes we cannot simply employ the Two Paths Theorem itself, as not every cross in the graph is compatible with its matching structure.
Instead, we need to discuss crosses which consist of alternating paths and which are accessible from the larger infrastructure provided by large grids\footnote{How these large grids are obtained is explained in \cref{sec_perfectmatchingwidth}.}.
As a starting point towards proving a matching theoretic version of the \hyperref[thm_twopaths]{\textsl{societal} Two Paths Theorem}, we introduce a theorem that links the existence of such a ``good'' cross to the presence of $K_{3,3}$ as a matching minor.
In $K_{3,3}$-free braces that are distinct from the Heawood graph we may replace any occurrences of non-planarity by $4$-cycles.
Moreover, if we are given a highly connected \hyperref[def:matchingminor]{matching minor} within a $K_{3,3}$-free brace, say for example a large grid, we know that this grid cannot be contained in many different parts which are separated by $4$-cycles, and thus we must be able to find a planar brace within the tree-structure provided by \cref{thm_trisums} that contains our large grid.
This gives a solid idea of what the absence of ``good'' crosses  might be like in bipartite matching covered graphs.
So next we need a tool that allows us to argue that, in the absence of a large complete bipartite graph as a matching minor, we will find an area within every grid which is, essentially, $K_{3,3}$-free.
 \cite{giannopoulou2021two} proposed an analogue of the \textsl{Two Paths Theorem} for crosses formed by alternating paths over conformal cycles in bipartite graphs with perfect matchings based on \hyperref[def:cyclesum]{$4$-cycle sums}.

\begin{definition}[Conformal cross]\label{@breathtaking}
	Let $G$ be a graph with a perfect matching and let $C$ be a conformal cycle in $G.$
	Two paths $P_1$ and $P_2$ where $P_i$ has endpoints $s_i$ and $t_i$ for each $i\in[2]$ are said to form a \emph{conformal cross} over $C$ if they are disjoint, internally disjoint from $C,$ their endpoints $s_1,$ $s_2,$ $t_1,$ and $t_2$ occur on $C$ in the order listed, and the graph $C+P_1+P_2$ is a conformal subgraph of $G.$
\end{definition}

%Please note that, if $P_1$ and $P_2$ form a \hyperref[def:matchingcross]{conformal cross} over some conformal cycle $C$ in a graph $G,$ then there exists a perfect matching $M$ of $G$ such that both $P_1$ and $P_2$ are internally $M$-conformal and $C$ is an $M$-conformal cycle.
The following two results are the key to the main theorem of \cite{giannopoulou2021two} and will serve as the main ingredient of a  societal Two Paths Theorem appropriate for the matching setting.

\begin{lemma}[\cite{giannopoulou2021two}]\label{lemma_goodcrossesmeanK33}
	Let $B$ be a brace and $C$ a $4$-cycle in $B,$ then there is a conformal cross over $C$ in $B$ if and only if $C$ is contained in a conformal bisubdivision of $K_{3,3}.$	
\end{lemma}

\begin{proposition}[\cite{giannopoulou2021two}]\label{thm_4cycleK33}
	Let $B$ be a brace containing $K_{3,3}$ and $C$ a $4$-cycle in $B,$ then there exists a conformal bisubdivision of $K_{3,3}$ with $C$ as a subgraph.
\end{proposition}

\subsection{Perfect Matching Width}\label{sec_perfectmatchingwidth}

What is necessary for our plan is an analogue of treewidth for the study of matching minors in graphs with perfect matchings.
To fill this gap, Norin \cite{norine2005matching} introduced the notion of ``perfect matching width''.

%%%%%%%%%%%%%%%%%%%%%%%%
%%%%%%% Definition of Perfect matching width
%%%%%%%%%%%%%%%%%%%%%%%%

%Let $G$ be a graph with at least one matching.  We use ${\cal M}(G)$ for the set of all perfect matchings of $G.$
%Given a $X\subseteq V(G),$ we denote by $\partial_{G}(X)$ the set of edges of $G$ that have one endpoint
%in $X$ and one in $V(G)\setminus X.$  
The \emph{matching porosity} of $X$ is defined as $$\mathbf{mp}_{G}(X)=\max\{|M\cap \partial_{G}(X)| \mid M\in \Perf{G}\}.$$
Notice that $\partial_{G}(X)=\partial_{G} (V(G)\setminus X),$ therefore $\mathbf{mp}_{G}(X)=\mathbf{mp}_{G} (V(G)\setminus X).$

\begin{definition}[Perfect matching width]\label{def_pmw}
	Let $G$ be a graph with at least one matching.
	A \emph{perfect matching decomposition} of $G$ is a pair $(T,\delta)$ 
	where $T$ is a cubic tree and $\delta$ is a bijection from $L(T)$ to $V(G).$
%	Notice that as $|\mathcal{M}|>0,$ $|T|$ should have at least two leaves.
	Let $e=t_1t_2\in E(T)$ and let $T_{1},T_{2}$ be the connected components of $T-e,$ assuming that $t_1\in V(T_{1}).$
	We define $X_i=\delta(V(T_{i})\cap L(T)), i\in[2]$ and observe that $\{X_{1},X_{2}\}$ is a partition of $V(G),$ 
	where $\partial_{G}(X_{1})=\partial_{G}(X_{2})$ defines an edge cut of $G,$ which we will denote as $\partial(e).$
	The \emph{width} of $(T,σ)$ is defined as $\max\{\mathbf{mp}_{G}(\partial_{G}(X))\mid e\in E(T)\}.$ Moreover, the \emph{perfect matching width} of $G,$ denoted as $\pmw{G},$ is defined to be the minimum width over all perfect matching decompositions of $G.$ 
\end{definition}

%%%%%%%%%%%%%%%%%%%%%%%%
%%%%%%% END of definition of Perfect matching width
%%%%%%%%%%%%%%%%%%%%%%%%

For our purposes we will also need an analogue of \cref{thm_undirectedgridtheorem} that provides us with a large planar infrastructure whenever we encounter a brace of large perfect matching width.

\begin{definition}[Cylindrical matching grid]\label{@wholeheartedly}
	The \emph{cylindrical matching grid} $CG_k$ of order $k$ is defined as follows.
	Let $C_1,\dots,C_k$ be $k$ vertex disjoint cycles of length $4k.$
	For every $i\in[k]$ let $C_i=\Brace{v_1^i,v_2^i,\dots,v_{4k}^i},$ $V_1^i\coloneqq\CondSet{v_j^i}{j\in\Set{1,3,5,\dots,4k-1}},$ $V_2^i\coloneqq\Fkt{V}{C_i}\setminus V_1^i,$ and $M_i\coloneqq\CondSet{v_j^iv_{j+1}^i}{v_j^i\in V_1^i}.$
	Then $CG_k$ is the graph obtained from the union of the $C_i$ by adding
	\begin{align*}
		\CondSet{v_j^iv_{j+1}^{i+1}}{i\in[k-1]~\text{and}~j\in\Set{1,5,9,\dots,4k-3}}&\text{, and}\\
		\CondSet{v_j^iv_{j+1}^{i-1}}{i\in[2,k]~\text{and}~j\in\Set{3,7,11,\dots,4k-1}}&
	\end{align*}
	to the edge set.
	We call $M\coloneqq\bigcup_{i=1}^kM_i$ the \emph{canonical matching} of $CG_k.$
	See \cref{fig_cylindricalgrid} for an illustration.
\end{definition}

Please note that the cylindrical matching grid is indeed a subcubic graph and thus, by \cref{lemma_confmathingminors}, one can always find a conformal bisubdivision of $CG_k$ within a bipartite graph $B$ if it contains $CG_k$ as a matching minor.

\begin{figure}
	\centering
\makefast{% !TEX root = ../single_comb.tex
% !TeX spellcheck = en_US	
		\begin{tikzpicture}

		\pgfdeclarelayer{background}
		\pgfdeclarelayer{foreground}
		\pgfsetlayers{background,main,foreground}

		\draw[e:main] (0,0) circle (11mm);
		\draw[e:main] (0,0) circle (16mm);
		\draw[e:main] (0,0) circle (21mm);
		\draw[e:main] (0,0) circle (26mm);
		
		\foreach \x in {1,...,4}
		{
			\draw[e:main] (\x*90:16mm) -- (\x*90+22.5:11mm);
			\draw[e:main] (\x*90:21mm) -- (\x*90+22.5:16mm);
			\draw[e:main] (\x*90:26mm) -- (\x*90+22.5:21mm);
		}
		
		\foreach \x in {1,...,4}
		{
			\draw[e:main] (\x*90-22.5:16mm) -- (\x*90-45:11mm);
			\draw[e:main] (\x*90-22.5:21mm) -- (\x*90-45:16mm);
			\draw[e:main] (\x*90-22.5:26mm) -- (\x*90-45:21mm);
			
		}
		
		\foreach \x in {1,...,8}
		{
			\draw[e:coloredthin,color=BostonUniversityRed,bend right=13] (\x*45:11mm) to (\x*45+22.5:11mm);
			\draw[e:coloredthin,color=BostonUniversityRed,bend right=13] (\x*45:16mm) to (\x*45+22.5:16mm);
			\draw[e:coloredthin,color=BostonUniversityRed,bend right=13] (\x*45:21mm) to (\x*45+22.5:21mm);
			\draw[e:coloredthin,color=BostonUniversityRed,bend right=13] (\x*45:26mm) to (\x*45+22.5:26mm);
		}
		
		\foreach \x in {1,...,8}
		{
			\node[v:main] () at (\x*45:11mm){};
			\node[v:main] () at (\x*45:16mm){};
			\node[v:main] () at (\x*45:21mm){};
			\node[v:main] () at (\x*45:26mm){};
			\node[v:mainempty] () at (\x*45+22.5:11mm){};
			\node[v:mainempty] () at (\x*45+22.5:16mm){};
			\node[v:mainempty] () at (\x*45+22.5:21mm){};
			\node[v:mainempty] () at (\x*45+22.5:26mm){};
		}

		\begin{pgfonlayer}{background}
			\foreach \x in {1,...,8}
			{
				\draw[e:coloredborder,bend right=13] (\x*45:11mm) to (\x*45+22.5:11mm);
				\draw[e:coloredborder,bend right=13] (\x*45:16mm) to (\x*45+22.5:16mm);
				\draw[e:coloredborder,bend right=13] (\x*45:21mm) to (\x*45+22.5:21mm);
				\draw[e:coloredborder,bend right=13] (\x*45:26mm) to (\x*45+22.5:26mm);
			}
		\end{pgfonlayer}
		
	\end{tikzpicture}}{
\scalebox{.174}{\includegraphics{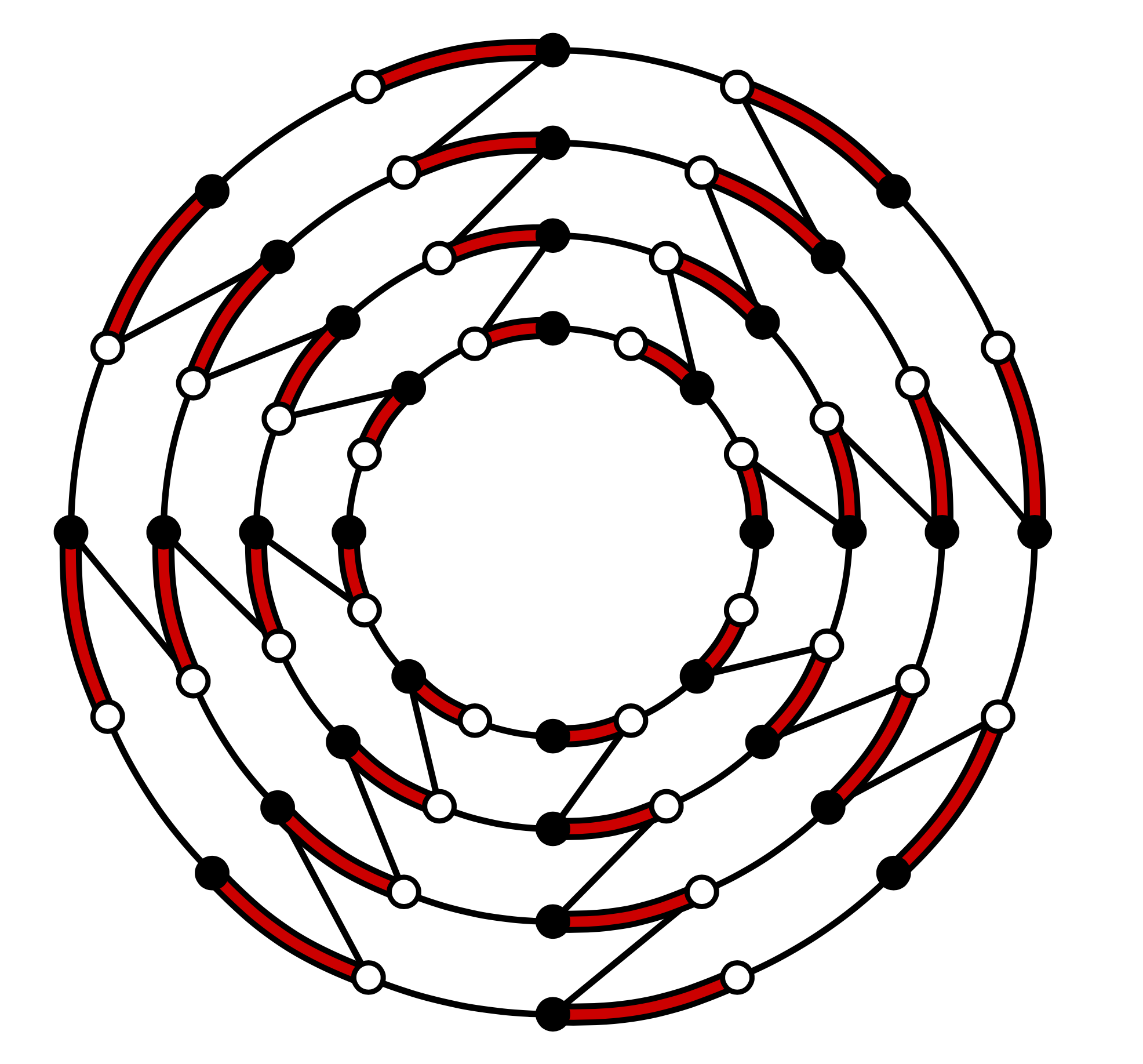}}}
	\caption{The \hyperref[def:matchinggrid]{cylindrical matching grid} of order $4$ with the canonical matching.}
	\label{fig_cylindricalgrid}
\end{figure}
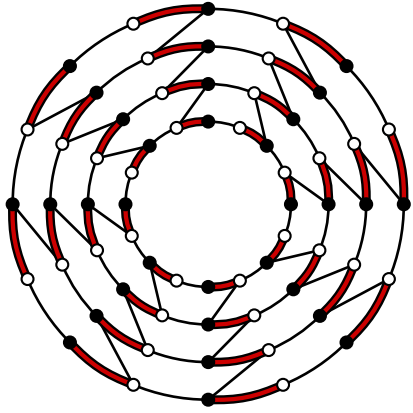

\begin{proposition}[\cite{hatzel2019cyclewidth}]\label{thm_matchinggrid}
There exists a function $\mathsf{mg}\colon\N\rightarrow\N$ such that for every $k\in\N,$ every bipartite graph $B$ with a perfect matching and $\pmw{B}\geq \mathsf{mg}(k)$ contains a conformal bisubdivision of the cylindrical matching grid of order $k.$
\end{proposition}

%%%%%%%%%%%%%%%%%%%%%%%%%%%%%%%%%%%%%%%%%%%%%%%%%%%%%%%%%%%%%%%%%%%%%%%%%%%%%%%%%%%%%%%%%%%%%%%%%%%%%%%%%%%%%%%%%%%%%%%%%%%%%%%%%%%%%%%%%%%%%%%%%%%%%%%%%%%%%%%%%%%%%%%%%%%%%%%%%%%%%%%%%%%%%%%%%%%%%%%%%%%%%%%%%%%%%%%%%%%%%%%%%%%%%%%%%%%%%%%%%%%%%%%%%%%%%%%%%%%%%%%%%%%%%%%%%%%%%%%%%%%%%%%%

\section{Societies, Renditions and Flatness}\label{sec_rendition}

Next we lay out the foundation for the concurrent steps in the development of a theory of matching minors.
We will need robust and versatile matching theoretic analogues of $Σ$-decompositions and societies.
In particular, this entails \hyperref[def_society]{societies}, \hyperref[def_sigmadecomposition]{$\Sigma$-decompositions}, \hyperref[def_vortex]{vortices}, and \hyperref[def_rendition]{renditions}.
Moreover, in the second part of this section, we lift \cref{lemma_goodcrossesmeanK33} to a societal version.
Please note that some of the concepts we introduce here do not find full-depth applications in this paper, they are rather designed to remain useful in future applications.

\subsection{Matching renditions}

\begin{definition}[Extended $\Sigma$-decomposition]\label{@circunspetta}
Let $\Sigma$ be a surface.
An \emph{extended $\Sigma$-decomposition} of a matching covered bipartite graph $B$ is a triple $(\Gamma,\mathcal{V},\mathcal{D}),$ where
\begin{itemize}
	\item $\Gamma$ is a drawing with crossings of $B$ on $\Sigma,$
	\item $\mathcal{V}$ and $\mathcal{D}$ are collections of closed disks, each a subset of $\Sigma,$ we call the disks in $\mathcal{V}$ the \emph{big vertices}, and
	\item $\Gamma,$ $\mathcal{V},$ and $\mathcal{D}$ satisfy axioms \textbf{\textsf{ED1}} to \textbf{\textsf{ED11}}.
\end{itemize}
Given a big vertex $\mathsf{v}\in\mathcal{V},$ the edges of $B$ drawn by $\Gamma$ such that they have at least one point on either side\footnote{If $C$ is a contractible curve in a surface $Σ,$ its {\em sides} are the connected components of $Σ\setminus C.$} of $\Boundary{\mathsf{v}}$ are called the \emph{crossing edges of $\mathsf{v}$}.
\begin{description}
	\item[ED1] The disks in $\mathcal{V}$ (respectively $\mathcal{D}$) have pairwise-disjoint interiors.
	\item[ED2] For every $Δ\in\mathcal{D}$ and every $\mathsf{v}\in\mathcal{V},$ $Δ\cap \mathsf{v}$ is a connected subset of the intersection of the boundaries of $Δ$ and $\mathsf{v}.$
	\item[ED3] For every  $Δ\in\mathcal{D},$ if $C$ is a connected component of $\Boundary{Δ\cap U(Γ)},$ then either $C\in V(Γ)$ or there exists some edge $e\in E(Γ)$ such that (i) $C\subseteq e$ and (ii) there exists some $\mathsf{v}\in\mathcal{V}$ such that $e$ is crossing for $\mathsf{v}.$
	\item[ED4] For every $\mathsf{v},$ no vertex of $B$ is drawn on the boundary of $\mathsf{v}$ by $\Gamma.$
	Moreover, for every $\mathsf{v},$ the collection $X$ of all vertices drawn by $\Gamma$ in the interior of $\mathsf{v}$ induces a non-trivial tight cut in $B.$
	We say that $\mathsf{v}$ \emph{belongs to the colour class} $V_i$ if $X\cap V_i$ is the \hyperref[obs_tightcutminoritymajority]{majority} of $X.$
	\item[ED5] If $\Delta_1,\Delta_2\in\mathcal{D}$ are distinct, then $\Delta_1\cap \Delta_2\subseteq V(\Gamma).$
	\item[ED6] If $\mathsf{u},\mathsf{v}\in\mathcal{V}$ are distinct, let us denote by $E(\mathsf{u},\mathsf{v})$ the set of edges which are crossing for both $\mathsf{u}$ and $\mathsf{v}.$
	In case $E(\mathsf{u},\mathsf{v})\neq\emptyset$ there exists a disk $\Delta\in\mathcal{D}$ which intersects both $\mathsf{u}$ and $\mathsf{v}$ but does not intersect any other big vertex, nor any vertex of $\Gamma.$
	Moreover, $\Delta$ contains all points of the drawings of all edges from $E(\mathsf{u},\mathsf{v})$ in $\Gamma$ that are not in the interiors of $\mathsf{u}$ or $\mathsf{v}.$
	Finally, $\Delta$ is disjoint from the drawings of all edges of $B$ which do not belong to $E(\mathsf{u},\mathsf{v}).$
	\item[ED7] If $\mathsf{v}\in\mathcal{V}$ is a big vertex and $u$ is a vertex of $\Gamma$ which is not contained in the interior of some disk in $\mathcal{V}\cup\mathcal{D},$ we denote by $E(u,\mathsf{v})$ the collection of all edges of $\Gamma$ which are incident with $u$ and crossing for $\mathsf{v}.$
	In case $E(u,\mathsf{v})\neq\emptyset$ there exists a disk $\Delta\in\mathcal{V}$ such that $\Delta$ intersects $\mathsf{v},$ $u$ is drawn on the boundary of $\Delta,$ and $\Delta$ is disjoint from any other big vertex or vertex of $\Gamma.$
	Moreover, $\Delta$ contains all points of the drawings of all edges from $E(u,\mathsf{v})$ in $\Gamma$ that are not in the interiors of $\mathsf{v}.$
	Finally, $\Delta$ is disjoint from the drawings of all edges of $B$ which do not belong to $E(u,\mathsf{v}).$
	\item[ED8] Every edge $e\in E(\Gamma)$ either belongs to the interior of one of the disks in $\mathcal{D}\cup\mathcal{V},$ there exist $\mathsf{u},\mathsf{v}\in\mathcal{V}$ such that $e\in E(\mathsf{u},\mathsf{v}),$ or there exist $u\in V(\Gamma)$ and $\mathsf{v}\in\mathcal{V}$ such that $e\in E(u,\mathsf{v}).$
	\item[ED9] If $u$ is a vertex of $\Gamma$ drawn in the interior of some disk $\Delta\in\mathcal{D}$ and $\mathsf{v}$ is a big vertex such that $u$ is incident to some crossing edge $e$ of $\mathsf{v},$ then $e\subseteq\mathsf{v}\cup\Delta.$
	\item[ED10] For every $\Delta\in\mathcal{D},$ the set $X$ of all vertices drawn by $\Gamma$ in the interior of $\Delta$ is conformal and $\InducedSubgraph{B}{X}$ has a perfect matching.
	\item[ED11] For every $\Delta\in\mathcal{D}$ let $X$ be the set of all vertices drawn on the boundary of $\Delta$ together with all big vertices which intersect $\Delta$ and let $\Omega_{\Delta}$ be a cyclic permutation induced by the order in which the members of $X$ appear\footnote{Note that for every member of $X,$ its intersection with $\Delta$ is a unique connected subset of $\Boundary{\Delta},$ these subsets are pairwise-disjoint and thus $\Omega_{\Delta}$ is well-defined up to the choice of the direction of traversing $\Boundary{\Delta}.$} on $\Boundary{\Delta}.$
	Then, if $x_1\in X$ is an immediate predecessor or successor of some $x_2\in X$ with respect to $\Omega_{\Delta},$ $x_1$ and $x_2$ belong to different colour classes of $B.$
	Moreover, $(X\cap V(\Gamma))\cup(\cupall(X\cap\mathcal{V})\cap V(\Gamma))$ is conformal.
%	Moreover, there exists a perfect matching $M$ of $B$ such that every edge of $M$ is either completely contained in $(X\cap V(\Gamma))\cup(X\cap\cupall\mathcal{V})$ or disjoint from this set and each member of $X$ is matched to an immediate predecessor or successor w.\@r.\@t.\@ $\Omega.$

	\begin{figure}[ht]
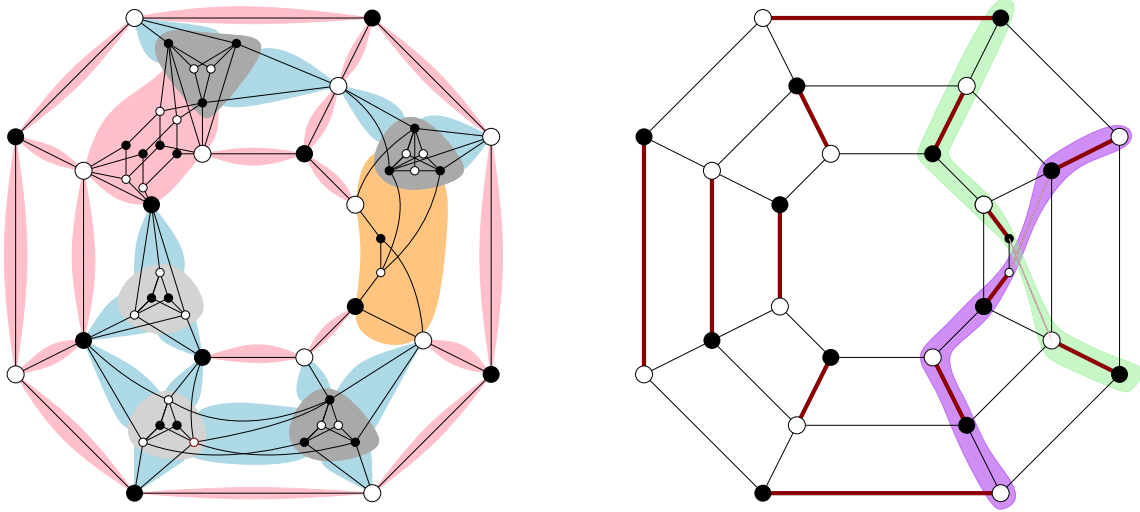

		\centering
\makefast{\scalebox{.8}{\hspace{-3cm}\input{figures/struc_all2.tex}}}{\includegraphics[scale=0.8]{figures/struc_all2}}
		\caption{An extended $\Sigma$-decomposition of a bipartite graph $B$ (on the left) and its reduction, drawn along with a perfect matching $M$ (on the right). 
The big vertices (i.e., the elements of $\mathcal{V}$) are drawn in grey and in light grey according to the colour class they belong to.	The elements of $\mathcal{D}$ are drawn blue or red. The blue ones 
correspond to edges of the reduction obtained by contracting the big vertices. The red ones are the cells of the   extended  $\Sigma$-decomposition, among which the orange one is a vortex.
		The society $(B,Ω)$ defined by the cyclic ordering of the outermost vertices is not matching flat.  A  cross of  $(B,Ω)$ is indicated by the purple and the green path in the reduction.}
		\label{fig_cylindercross}
	\end{figure}

\end{description}
We define containment, isomorphism, and restrictions of extended $\Sigma$-decompositions in the natural way.
Let $N$ be the set of all vertices of $\Gamma$ that do not belong to the interior of the disks in $\mathcal{D}\cup\mathcal{V}$ together with all big vertices of $\delta.$
We will refer to the elements of $N$ as the \emph{nodes} of $\delta.$
Note that any big vertex can be identified into a single vertex using a tight cut contraction, which justifies us treating the elements of $\mathcal{V}$ as vertices for the purpose of an extended $\Sigma$-decomposition.
Let $\Delta\in\mathcal{D}$ and $X\subseteq N.$
By $\Delta-X$ we denote the set $\Delta-(X\cap V(\Gamma))-\cupall(\mathcal{V}\cap X)$, and we refer to the set $\Delta-N$ as a \emph{cell} of $\delta.$
We denote the set of nodes of $\delta$ by $N(\delta)$ and the set of cells by $C(\delta).$
For a cell $c\in C(\delta),$ we denote by $\widetilde{c}$ the set of vertices of $\Gamma$ that lie on the boundary of the closure of $c,$ together with all big vertices that are intersected by the closure of $c.$
Thus, the cells $c$ of $\delta$ with $\widetilde{c}\neq\emptyset$ form the hyperedges of a hypergraph with vertex set $N(\delta),$ where $\widetilde{c}$ is the set of vertices incident with $c\in C(\delta).$
For a cell $c\in C(\delta),$ we define $\sigma_{\delta}(c)$ (or $\sigma(c)$ when $\delta$ is understood from the context) to be the subgraph of $B$ consisting of all vertices and edges drawn in the union of the closure of $c$ and all big vertices it intersects.
We define $\pi_{\delta}\colon N(\delta)\rightarrow\V{B}\cup (\V{B}\times 2^{V(B)})$ to be the mapping that assigns to every vertex of $\Gamma$ in $N(\delta)$ its identity and to every big vertex $\mathsf{v}\in\mathcal{V}$ the set $X$ of vertices of $B$ corresponding to the vertices contained in the closure of $\mathsf{v}.$

Let $x\in N(\delta).$
If $x$ is a big vertex, we say that an edge is \emph{incident with $x$} if it is a crossing edge of $x.$
\end{definition}
Given the above, we also say that the boundary of a disk $\Delta$ \emph{intersects $\Gamma$ in a big vertex $\mathsf{v}\in\mathcal{V}$} if $\Boundary{\Delta}\cap\Boundary{\mathsf{v}}$ consists of a single connected component.\smallskip

Let $\delta=(\Gamma,\mathcal{V},\mathcal{D})$ be an extended $\Sigma$-decomposition of some bipartite matching covered graph $B.$
Consider a big vertex $\mathsf{v}\in\mathcal{V}.$ 
Since the collection $X$ of vertices drawn by $\Gamma$ in the interior of $\mathsf{v}$ induce a non-trivial tight cut in $B,$ we may perform a tight cut contraction to reduce $X$ to a single vertex $r_{\mathsf{v}}$ which belongs to the same colour class as $\mathsf{v}.$
We call this operation the \emph{contraction of $\mathsf{v}$}.
We will usually work on the \emph{reduction} $\hat{\delta}=(\hat{B},\hat{\Gamma},\hat{\mathcal{D}})$ of $\delta$ obtained by contracting all members of $\mathcal{V}$ where $\hat{\Gamma}$ and $\hat{\mathcal{D}}$ are obtained from $\Gamma$ and $\mathcal{D}$ respectively such that $(\hat{\Gamma},\emptyset,\hat{\mathcal{D}})$ is an extended $\Sigma$-decomposition of $\hat{B}$ without any big vertices.
We usually identify $(\hat{\Gamma},\emptyset,\hat{\mathcal{D}})$ and $\hat{\delta}.$
Note that any cell $c\in C(\delta)$ corresponds to a disk $\Delta\in\mathcal{D}.$
By $\Omega_c$ we refer to the ordering $\Omega_{\Delta}$ which we usually identify with the induced ordering of the vertices in the graph $\hat{B}$
It follows that $(\Omega_c,\sigma_{\hat{\delta}}(c))$ is a society with a series of additional properties.

\begin{definition}[Matching society]\label{@klammerausdrucks}
	A \emph{matching society} is a pair $(B,\Omega),$ where $B$ is a bipartite graph with a perfect matching and $\Omega=(a_1,b_1,a_2,b_2,\dots,a_{\ell},b_{\ell})$ is a cyclic permutation  of some vertices of $B$ such that $\Set{a_1b_1,a_2b_2,\dots,a_{\ell}b_{\ell}}$ is contained in \textsl{some} perfect matching of $B.$
\end{definition}

In the case of $\Sigma$-decompositions, a cell would either have a small boundary, which contains at most three vertices, or it would be declared a vortex.
The number three here originates from the Two Paths Theorem, as here any separation of order at most three can be replaced by a tiny clique that maintains the existence of disjoint paths but, since it is so small, does not allow for crossings to occur.
In the matching setting we have to take care of this in two ways.
First of all, crossings cannot pass through tight cuts and return and thus we may allow for tight cut contractions.
Since these correspond to extremely small separations but may contain large obstructions to embeddability in the plane they have to be dealt with through contractions instead of vertex separations.
This is the role of the big vertices.
After tight cuts are removed, the only kind of separation we have to pay special attention to is the $4$-cycle sum.
Here two cases can occur:
If the graph ``above'' the separating $4$-cycle is Pfaffian, it cannot produce a cross suitable for matching theoretic applications, but still might create a cross in the usual sense.
So, in this case, we have to allow for these ``flaps'' to exist as we cannot use them to create an excluded minor.
However, if whatever is separated by the $4$-cycle contains a $K_{3,3}$ that cannot be separated from this cycle with a tight cut, then \cref{lemma_goodcrossesmeanK33} guarantees the existence of a matching theoretic cross.
Hence, it is crucial to differentiate between the two cases.
Any cell with a larger boundary will immediately be classified as a vortex.
Recall that $|\widetilde{c}|$ must be even for all cells of an extended $\Sigma$-decomposition.

\begin{definition}[Vortex]\label{@cowardliness}
	Let $B$ be a bipartite matching covered graph, $\Sigma$ be a surface and $\delta=(\Gamma,\mathcal{V},\mathcal{D})$ be an extended $\Sigma$-decomposition of $B$ with reduction $\hat{\delta}=(\hat{B},\hat{\Gamma},\hat{\mathcal{D}}).$
	A cell $c\in C(\delta)$ is called a \emph{vortex} if it meets one of the following requirements:
	\begin{itemize}
		\item $\Abs{\widetilde{c}}=2$ and $|V(\sigma(c))|\geq 3,$
		\item $\Abs{\widetilde{c}}=4$ and, if $\hat{c}=\Set{a_1,b_1,a_2,b_2}$ is the cell of $\hat{\delta}$ which corresponds to $c,$ $\sigma_{\hat{\delta}}(\hat{c})+\Set{a_1b_1, b_1a_2, a_2b_2, b_2a_1}$ contains a conformal bisubdivision of $K_{3,3}$ which contains the $4$-cycle $(a_1,b_1,a_2,b_2)$ as a subgraph, or
		\item $\Abs{\widetilde{c}}\geq 6.$
	\end{itemize} 
\end{definition}

With this preparation we are finally ready to define a matching theoretic version of renditions.
This will almost be the last step towards a society version of the matching theoretic Two Paths Theorem.

\begin{definition}[Matching rendition]\label{@condescension}
	Let $(B,\Omega)$ be a \hyperref[def:matchingsociety]{matching society} and $\Sigma$ be a surface with one boundary component $G.$
	A \emph{matching rendition} of $(B,\Omega)$ in $\Sigma$ is an extended $\Sigma$-decomposition $\rho$ of $B$ such that no vertex of $V(\Omega)$ is contained in a big vertex and the mapping of the image under $\pi_{\rho}$ of $N(\rho)\cap G$ is $V(\Omega),$ mapping one of the cyclic orders of $B$ to the order of $\Omega.$
	If $(B,\Omega)$ has a vortex-free matching rendition in the disk, we say that $(G,\Omega)$ is $\emph{matching flat}.$
\end{definition}

\subsection{A matching theoretic society lemma}

The problem with a direct application of \cref{lemma_goodcrossesmeanK33} to a \hyperref[def:matchingsociety]{matching society} $(B,\Omega)$ is that it requires  $B$ to be a brace and either $V(\Omega)$ to induce a $4$-cycle or the addition of artificial edges  to  create a $4$-cycle using only vertices of $V(\Omega).$
While $B$ being a brace can probably be handled by using big vertices, the other two requirements are either extremely restrictive or will make it hard to fit two crossing alternating paths on the society $(B,\Omega)$ back into the framework of a larger graph.
Hence, we need to find some kind of workaround to allow for an application of \cref{lemma_goodcrossesmeanK33} without changing the graph too much.
The idea is to augment a society with a small grid-like structure that allows us to ``project'' a conformal cross over an artificial four-cycle onto a cycle of arbitrary length.

\begin{definition}[Reinforced society]\label{@investigates}
A \emph{reinforced society} is a tuple $\mathscr{S}=(B,H,\Psi,\Omega)$ where 
\begin{itemize}
	\item $B$ and $H$ are bipartite matching covered graphs with $H$ being a conformal subgraph of $B,$
	\item $(B,\Psi)$ and $(B-(H-V(\Omega)),\Omega)$ are matching societies,
	\item $H$ contains a conformal bisubdivision of the \hyperref[def:matchinggrid]{cylindrical matching grid} of order $3$ with concentric cycles $C_1,$ $C_2,$ and $C_3$ such that $V(\Psi)=V(C_1)$ and $V(\Omega)=V(C_3),$
	\item $(V(B)-(H-V(\Omega)),H)$ is a separation with $V(\Omega)$ being the separator, and
	\item $H$ has a vortex-free extended $\Delta$-decomposition $\rho$ in some disk $Δ,$ which is simultaneously a matching rendition of $(H,\Psi)$ and of $(H,\Omega).$
\end{itemize}

We denote by $B_{\mathscr{S}}$ the graph $B-(H-V(\Omega)).$
We say that $\mathscr{S}$ has a \emph{cross}, if the cycle $C_1$ with vertex set $V(\Psi)$ has a \hyperref[def:matchingcross]{conformal cross} in $B.$
\end{definition}

With these definitions in place we are finally ready to state our key result, that is a society based Two Paths Theorem for bipartite matching covered graphs.

\begin{theorem}\label{thm_societytwopaths}
	Let $\mathscr{S}=(B,H,\Psi,\Omega)$ be a reinforced society.
	The matching society $(B,\Psi)$ is matching flat if and only if $\mathscr{S}$ does not have a cross.
\end{theorem}

Before we can prove \cref{thm_societytwopaths} we need to discuss two important tools.
First we describe how out choice for the existence of a matching cylindrical grid of order $3$ as a conformal subdivision allows us to modify the ``outer'' parts of $H$ without any influence on $(B_{\mathscr{S}},\Omega).$ 

\begin{observation}\label{lemma_tightcutsandconformalcycles}
Let $G$ be a matching covered graph and $C_1,$ $C_2$ be two disjoint cycles such that there exists a perfect matching $M$ of $G$ for which both $C_1$ and $C_2$ are $M$-conformal.
Now let $X\subseteq V(G)$ be a set of vertices with $\Cut{X}\cap E(C_i)\neq\emptyset$ for both $i\in[2],$ then $\Cut{X}$ is not a tight cut of $G.$
\end{observation}

\begin{proof}
Observe that if $\Cut{X}$ is a tight cut, it cannot contain an edge of $M\cap E(C_i)$ for both $i\in[2].$
Hence, there exists $j\in[2]$ such that $\Cut{X}\cap E(C_j)\cap M=\emptyset.$
However, $\Abs{\Cut{X}\cap E(C_j)}\geq 2$ as $C_j$ is a cycle, and $M'=(M\setminus E(C_j))\cup (E(C_j)\setminus M)$ is a perfect matching of $G$ with $\Abs{\Cut{X}\cap M'}\geq 2.$
Hence, $\Cut{G}$ cannot be tight.
\end{proof}
Let $\mathscr{S}=(B,H,\Psi,\Omega)$ be a reinforced society, let $\rho=(\Gamma,\mathcal{V},\mathcal{D})$ be a matching rendition of $(H,\Psi)$ in a disk $\Delta$ which can also be interpreted as a matching rendition of $(H,\Omega)$ in some disk $\Delta'.$
Observe that there exists a closed curve $\gamma$ in $\Delta$ which follows\footnote{We call such a curve the \emph{trace} of the cycle, similar to the corresponding notion from \cite{kawarabayashi2020quickly}.} the way the cycle $C_2$ of the conformal bisubdivision of the cylindrical matching grid $CG_3$ in $H$ is embedded into $\Delta$ by $\Gamma.$
Let $\Delta''$ be obtained from $\Delta$ by deleting the component of $\Delta-\gamma$ which is disjoint from the boundary of $\Delta$ and let $H''$ be the subgraph of $H$ which is drawn in $\Delta''$ by $\Gamma.$
Finally, let $H'\coloneqq H-(H''-C_2).$
Then $H'$ and $H''$ intersect exactly in the cycle $C_2.$
Let $\Psi'$ be the cyclic permutation of $V(C_2)$ obtained by traversing along $C_2$ in clockwise order.
Then $(H',\Psi')$ is a society that has a matching rendition in a disk.
Finally, let $B'=B-(H''-C_2).$
We call $\mathscr{S}^-=(H',B',\Psi',\Omega)$ the \emph{reduct} of $\mathscr{S}.$

It follows from \cref{lemma_tightcutsandconformalcycles} that any non-trivial tight cut of $B'$ that does not correspond to a non-trivial tight cut of $B$ must have a shore $X\subseteq V(B')$ for which $X\subseteq (H'-C_3).$
Hence, a set $X\subseteq (B'-(H'-C_3))$ induces a non-trivial tight cut in $B'$ if and only if it induces a non-trivial tight cut in $B.$ 

We now enhance the graphs $H'$ and $B'$ by introducing four fresh vertices $a,c\in V_1$ and $b,d\in V_2$ together with the edges $ab,$ $bc,$ $cd,$ $ad,$ $av^2_{2},$ $bv^2_{5},$ $cv^2_{6},$ and $dv^2_{1},$ where the vertices $v^2_i$ are the branch vertices of $CG_3$ of the corresponding vertices on $C_2.$
Let $H^+$ and $B^+$ be the resulting graphs.
Observe that $H^+$ and $B^+$ are matching minors of $H$ and $B$ respectively.
Hence, any matching minor of $B^+$ will also be a matching minor of $B.$
Finally, let $\Psi^+$ be the cyclical permutation $(a,b,c,d).$
We call $\mathscr{S}^+=(B^+,H^+,\Psi^+,\Omega)$ the \emph{enhancement} of $\mathscr{S}^{-}.$
%Since the reduct $\mathscr{S}^{-}$ is uniquely determined by $\mathscr{S}$, so is the enhancement $\mathscr{S}^{+}$ of $\mathscr{S}^{-}$.
%This gives us the right to refer to $\mathscr{S}^{+}$ as the \emph{enhancement of $\mathscr{S}$} wherever this is more convenient.
See \cref{fig_enhancedsociety} for an illustration.

\begin{figure}
	\centering
\makefast{% !TEX root = ../single_comb.tex
% !TeX spellcheck = en_UK	

	\begin{subfigure}{0.32\textwidth}
		\centering
	\begin{tikzpicture}[scale=0.68]

		\pgfdeclarelayer{background}
		\pgfdeclarelayer{foreground}
		\pgfsetlayers{background,main,foreground}

		\draw[e:main] (0,0) circle (16mm);
		\draw[e:main] (0,0) circle (23mm);
		\draw[e:main] (0,0) circle (30mm);
		
		\foreach \x in {1,...,3}
		{
			\draw[e:main] (\x*120-15:23mm) -- (\x*120+30-15:16mm);
			\draw[e:main] (\x*120-15:30mm) -- (\x*120+30-15:23mm);
		}
		
		\foreach \x in {1,...,3}
		{
			\draw[e:main] (\x*120-30-15:23mm) -- (\x*120-60-15:16mm);
			\draw[e:main] (\x*120-30-15:30mm) -- (\x*120-60-15:23mm);
			
		}
		
		\foreach \x in {1,...,6}
		{
			\draw[e:coloredthin,color=BostonUniversityRed,bend right=13] (\x*60-15:16mm) to (\x*60+30-15:16mm);
			\draw[e:coloredthin,color=BostonUniversityRed,bend right=13] (\x*60-15:23mm) to (\x*60+30-15:23mm);
			\draw[e:coloredthin,color=BostonUniversityRed,bend right=13] (\x*60-15:30mm) to (\x*60+30-15:30mm);
		}
		
		\foreach \x in {1,...,6}
		{
			\node[v:main] () at (\x*60-15:16mm){};
			\node[v:main] () at (\x*60-15:23mm){};
			\node[v:main] () at (\x*60-15:30mm){};
			\node[v:mainempty] () at (\x*60+30-15:16mm){};
			\node[v:mainempty] () at (\x*60+30-15:23mm){};
			\node[v:mainempty] () at (\x*60+30-15:30mm){};
		}
		
		\node[v:ghost] () at (90:33mm) {};
		\node[v:ghost] () at (270:40mm) {$\mathscr{S}=(B,H,\Psi,\Omega)$}; 
		
		\begin{pgfonlayer}{background}
			\path[fill=AO,opacity=0.2] (0,0) circle [radius=3];
			\path[fill=white] (0,0) circle [radius=1.6];
			\path[fill=CornflowerBlue,opacity=0.4] (0,0) circle [radius=1.6];
			
			\node () at (0,0) {$B_{\mathscr{S}}$};
			\node () at (90:26.3mm) {\textcolor{AO}{$H$}};
			
			\foreach \x in {1,...,6}
			{
				\draw[e:coloredborder,bend right=13] (\x*60-15:16mm) to (\x*60+30-15:16mm);
				\draw[e:coloredborder,bend right=13] (\x*60-15:23mm) to (\x*60+30-15:23mm);
				\draw[e:coloredborder,bend right=13] (\x*60-15:30mm) to (\x*60+30-15:30mm);
			}
		\end{pgfonlayer}
		
	\end{tikzpicture}
	\end{subfigure}
	\begin{subfigure}{0.32\textwidth}
		\centering
		\begin{tikzpicture}[scale=0.68]

			\pgfdeclarelayer{background}
			\pgfdeclarelayer{foreground}
			\pgfsetlayers{background,main,foreground}

			\draw[e:main] (0,0) circle (16mm);
			\draw[e:main] (0,0) circle (23mm);
			%			\draw[e:main] (0,0) circle (30mm);
			
			\foreach \x in {1,...,3}
			{
				\draw[e:main] (\x*120-15:23mm) -- (\x*120+30-15:16mm);
				%				\draw[e:main] (\x*120:30mm) -- (\x*120+30:23mm);
			}
			
			\foreach \x in {1,...,3}
			{
				\draw[e:main] (\x*120-30-15:23mm) -- (\x*120-60-15:16mm);
				%				\draw[e:main] (\x*120-30:30mm) -- (\x*120-60:23mm);
				
			}
			
			\foreach \x in {1,...,6}
			{
				\draw[e:coloredthin,color=BostonUniversityRed,bend right=13] (\x*60-15:16mm) to (\x*60+30-15:16mm);
				\draw[e:coloredthin,color=BostonUniversityRed,bend right=13] (\x*60-15:23mm) to (\x*60+30-15:23mm);
				%				\draw[e:coloredthin,color=BostonUniversityRed,bend right=13] (\x*60:30mm) to (\x*60+30:30mm);
			}
			
			\foreach \x in {1,...,6}
			{
				\node[v:main] () at (\x*60-15:16mm){};
				\node[v:main] () at (\x*60-15:23mm){};
				%				\node[v:main] () at (\x*60:30mm){};
				\node[v:mainempty] () at (\x*60+30-15:16mm){};
				\node[v:mainempty] () at (\x*60+30-15:23mm){};
				%				\node[v:mainempty] () at (\x*60+30:30mm){};
			}
			
			\node[v:ghost] () at (90:33mm) {};
			\node[v:ghost] () at (270:40mm) {$\mathscr{S}^-=(B',H',\Psi',\Omega)$}; 
			
			\begin{pgfonlayer}{background}
				\path[fill=AO,opacity=0.2] (0,0) circle [radius=2.3];
				\path[fill=white] (0,0) circle [radius=1.6];
				\path[fill=CornflowerBlue,opacity=0.4] (0,0) circle [radius=1.6];
				
				\node () at (0,0) {$B_{\mathscr{S}}$};
				\node () at (90:19.3mm) {\textcolor{AO}{$H'$}};
				
				\foreach \x in {1,...,6}
				{
					\draw[e:coloredborder,bend right=13] (\x*60-15:16mm) to (\x*60-15+30:16mm);
					\draw[e:coloredborder,bend right=13] (\x*60-15:23mm) to (\x*60-15+30:23mm);
					%					\draw[e:coloredborder,bend right=13] (\x*60:30mm) to (\x*60+30:30mm);
				}
			\end{pgfonlayer}
			
		\end{tikzpicture}
	\end{subfigure}
	\begin{subfigure}{0.32\textwidth}
		\centering
		\begin{tikzpicture}[scale=0.68]

			\pgfdeclarelayer{background}
			\pgfdeclarelayer{foreground}
			\pgfsetlayers{background,main,foreground}

			\draw[e:main] (0,0) circle (16mm);
			\draw[e:main] (0,0) circle (23mm);
			%			\draw[e:main] (0,0) circle (30mm);
			
			\foreach \x in {1,...,3}
			{
				\draw[e:main] (\x*120-15:23mm) -- (\x*120-15+30:16mm);
				%				\draw[e:main] (\x*120:30mm) -- (\x*120+30:23mm);
			}
			
			\foreach \x in {1,...,3}
			{
				\draw[e:main] (\x*120-30-15:23mm) -- (\x*120-60-15:16mm);
				%				\draw[e:main] (\x*120-30:30mm) -- (\x*120-60:23mm);
				
			}
			
			\foreach \x in {1,...,6}
			{
				\draw[e:coloredthin,color=BostonUniversityRed,bend right=13] (\x*60-15:16mm) to (\x*60+30-15:16mm);
				\draw[e:coloredthin,color=BostonUniversityRed,bend right=13] (\x*60-15:23mm) to (\x*60+30-15:23mm);
				%				\draw[e:coloredthin,color=BostonUniversityRed,bend right=13] (\x*60:30mm) to (\x*60+30:30mm);
			}
			
			\foreach \x in {1,...,6}
			{
				\node[v:main] () at (\x*60-15:16mm){};
				\node[v:main] () at (\x*60-15:23mm){};
				%				\node[v:main] () at (\x*60:30mm){};
				\node[v:mainempty] () at (\x*60+30-15:16mm){};
				\node[v:mainempty] () at (\x*60+30-15:23mm){};
				%				\node[v:mainempty] () at (\x*60+30:30mm){};
			}
		
			\node[draw, circle, scale=1.1, thick,color=DarkTangerine,fill=white,inner sep=0.7mm] () at (45:32mm){};
			\node[draw, circle, scale=1.1, thick,color=DarkTangerine,fill=DarkTangerine,inner sep=0.7mm] () at (135:32mm){};
			\node[draw, circle, scale=1.1, thick,color=DarkTangerine,fill=white,inner sep=0.7mm] () at (165:32mm){};
			\node[draw, circle, scale=1.1, thick,color=DarkTangerine,fill=DarkTangerine,inner sep=0.7mm] () at (15:32mm){};
			
			\node[v:ghost] () at (90:33mm) {};
			\node[v:ghost] () at (270:40mm) {$\mathscr{S}^+=(B^+,H^+,\Psi^+,\Omega)$}; 
			
			\begin{pgfonlayer}{background}
				\path[fill=AO,opacity=0.2] (0,0) circle [radius=2.3];
				\path[fill=white] (0,0) circle [radius=1.6];
				\path[fill=CornflowerBlue,opacity=0.4] (0,0) circle [radius=1.6];
				
				\node () at (0,0) {$B_{\mathscr{S}}$};
				\node () at (90:19.3mm) {\textcolor{AO}{$H'$}};
				
				\draw[e:colored,color=DarkTangerine] (0,0) circle (32mm);
				
				\draw[e:colored,color=DarkTangerine] (135:32mm) to (150-15:23mm);
				\draw[e:colored,color=DarkTangerine] (165:32mm) to (180-15:23mm);
				\draw[e:colored,color=DarkTangerine] (15:32mm) to (30-15:23mm);
				\draw[e:colored,color=DarkTangerine] (45:32mm) to (60-15:23mm);
				
				\draw[e:coloredborder,bend right=13,color=DarkTangerine] (135:32mm) to (165:32mm);
				\draw[e:coloredborder,bend right=13,color=DarkTangerine] (15:32mm) to (45:32mm);
				
				\draw[e:colored,bend right=13] (135:32mm) to (165:32mm);
				\draw[e:colored,bend right=13] (15:32mm) to (45:32mm);
				
				\foreach \x in {1,...,6}
				{
					\draw[e:coloredborder,bend right=13] (\x*60-15:16mm) to (\x*60+30-15:16mm);
					\draw[e:coloredborder,bend right=13] (\x*60-15:23mm) to (\x*60+30-15:23mm);
					%					\draw[e:coloredborder,bend right=13] (\x*60:30mm) to (\x*60+30:30mm);
				}
			\end{pgfonlayer}
			
		\end{tikzpicture}
	\end{subfigure}}{\scalebox{.55}{\includegraphics{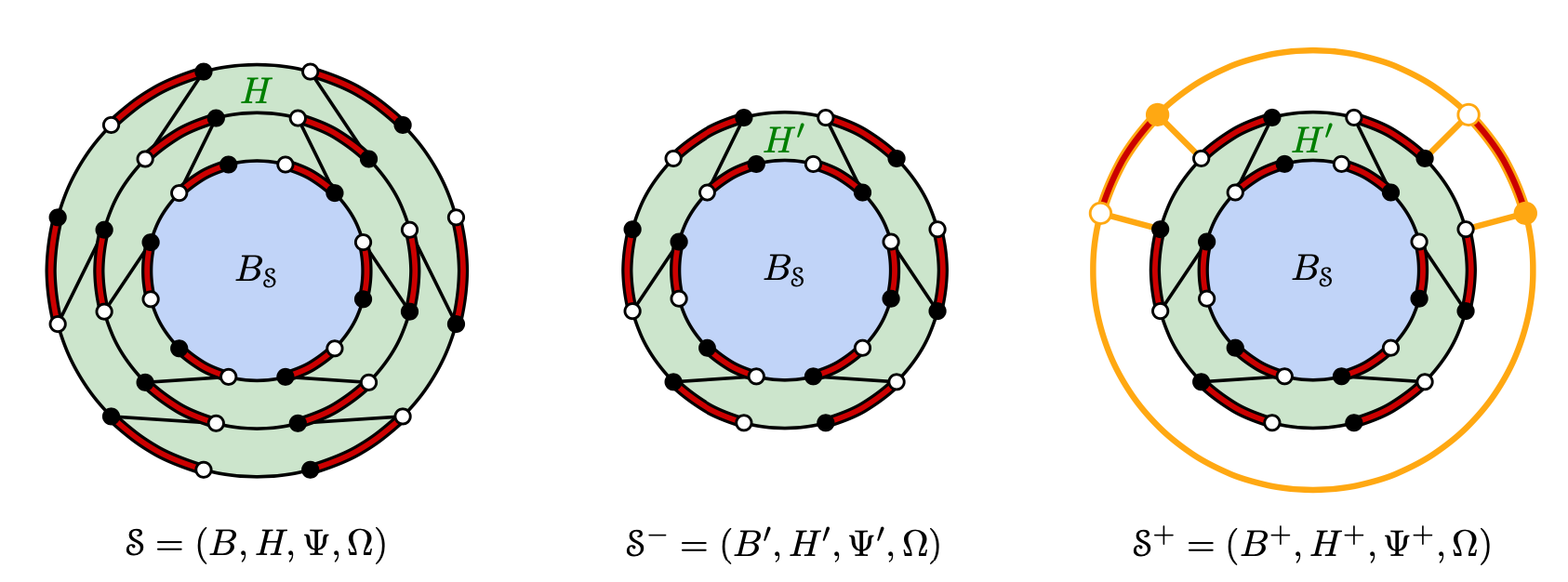}}}
	\caption{A reinforced society $\mathscr{S}=(B,H,\Psi,\Omega)$ (leftmost). The blue area indicates the graph $B_{\mathscr{S}},$ while the green are indicates the graph $H$ together with its conformal bisubdivision of a cylindrical matching grid of order $3.$ The \emph{reduct} $\mathscr{S}^{-}$ (middle) is obtained from $\mathscr{S}$ by deleting the subgraph of $H$ captured between the cycles $C_1$ and $C_2.$
	Finally, the \emph{enhancement} $\mathscr{S}^{+}$ (rightmost) of $\mathscr{S}^-$ is obtained from $\mathscr{S}^{-}$ by adding a $4$-cycle which can be found as a conformal bisubdivision in the part of $H$ that was removed to obtain $\mathscr{S}^{-}.$}
	\label{fig_enhancedsociety}
\end{figure}
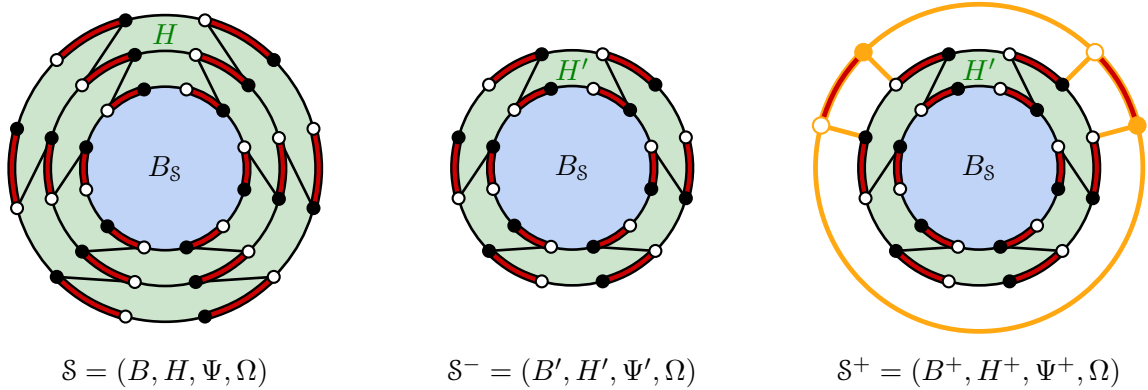

A second tool is needed to identify the ``boundaries'' of two matching societies with vortex-free renditions by joining these renditions.
This lemma will find additional applications later on.

\begin{lemma}\label{lemma_joiningoverlappingsocieties}
	Let $(B_1,\Omega)$ and $(B_2,\Omega)$ be two matching societies with $V(B_1)\cap V(B_2)=V(\Omega).$
	Moreover, let $\Sigma_1$ and $\Sigma_2$ be two surfaces such that for each $i\in[2]$ the surface $\Sigma_i$ has a boundary component $G_i$ and the society $(B_i,\Omega)$ has a matching rendition $\rho_i$ in $\Sigma_i.$
	Let $\Sigma$ be the surface obtained from $\Sigma_1$ and $\Sigma_2$ by identifying $G_1$ and $G_2,$ then $B_1\cup B_2$ has a vortex-free extended $\Sigma$-decomposition $\delta$ which contains both $\rho_1$ and $\rho_2.$
\end{lemma}

\begin{proof}
	This lemma follows immediately from the definition of \hyperref[def:matchingrendition]{matching renditions}.
	To see this observe that neither in $\rho_1$ nor in $\rho_2$ a vertex of $V(\Omega)$ may be drawn within a big vertex.
	Hence, for both $\rho_i$ the entirety of $V(\Omega)$ is drawn properly on the respective boundary component $G_i$ of $\Sigma_i.$
	Thus, $\rho_1$ and $\rho_2$ coincide in the way the vertices of $\Omega$ are drawn which allows us to unify both renditions into an extended $\Sigma$-decomposition $\delta.$
	Any vortex of $\delta$ would be a disk that is either completely contained in $\Sigma_1$ or $\Sigma_2$, and thus it would correspond to a vortex of one of the $\rho_i$ which cannot exist.
\end{proof}

\begin{proof}[Proof of \Cref{thm_societytwopaths}]
Consider the reinforced society $\mathscr{S}=(B,H,\Psi,\Omega)$ and the enhancement $\mathscr{S}^+=(B^+,H^+,\Psi^+,\Omega)$ of its reduct.
%\sed{Is here the enhancement correct? We agreed that we "we call $\mathscr{S}^+=(B^+,H^+,\Psi^+,\Omega)$ the \rred{\emph{enhancement} of $\mathscr{S}^{-}.$}" and not of $\mathscr{S}.$ I changed it everywhere but here I need some verification.}
%\sw{Seems fine to me}
We prove the result by induction on $\Abs{V(B)}+\Abs{E(B)}.$\smallskip

As a first step, we show that we may assume  $B$ to be a brace.
To see this suppose there exists a non-trivial tight cut $\Cut{X}$ in $B.$
By \cref{lemma_tightcutsandconformalcycles} either $X$ or $\Complement{X},$ w.\@l.\@o.\@g.\@ we may assume it to be $\Complement{X},$ contains at least two of the three cycles $C_1,$ $C_2,$ and $C_3$ of the conformal bisubdivision $J$ of $CG_3$ in $H.$
Let $B'$ be obtained from $B$ by contracting $X$ into a single vertex $v_X$ and observe that $B'$ still contains a conformal bisubdivision of $CG_3.$
Moreover, we may adjust $H,$ $\Psi,$ and $\Omega$ to generate a reinforced society $\mathscr{S}'=(B',H,',\Psi',\Omega').$
By induction $(B',\Psi')$ is matching flat, that is it has a vortex-free matching rendition in the disk, if and only if $\mathscr{S}'$ does not have a cross.

\begin{claim}\label{@contrasentidos}
$\mathscr{S}$ has a cross if and only if $\mathscr{S}'$ has a cross.
\end{claim}

\noindent\emph{Proof of Claim 1:}
In case $\mathscr{S}'$ has a cross, then we may also find a cross for $\mathscr{S}$ by simply expanding $v_X$ back to the set $X$ and ``patching'' either $C_1$ or one of the two paths of the cross with an appropriate path through $X,$ which has to exist since $B'$ is a matching minor of $B$ \cite{mccuaig2001brace}.

Now suppose $\mathscr{S}$ has a cross.
Then there exists a perfect matching $M$ together with two disjoint paths $P_1,$ $P_2$ such that $K=C_1+P_1+P_2$ is an $M$-conformal subgraph of $B$ and $P_1$ and $P_2$ form a cross over $C_1.$
Observe that for each of the two paths $P_i$ there exists a subpath $Q_i$ of $C_1$ such that $P_i+Q_i$ is an $M$-conformal cycle.
Moreover, for each $i\in[2]$ there exists a path $Q_i'$ such that $Q_i$ and $Q_i'$ coincide in their endpoints only and $C_1=Q_i+Q_i'.$
Let $M_i$ denote the perfect matching $(M\setminus E(P_{3-i}+Q_{3-i}))\cup(E(P_{3-i}+Q_{3-i})\setminus M)$ of $K.$
Let $\Cut{Y}$ be any non-trivial tight cut of $B.$
Notice that $\Cut{Y}$ either contains exactly two edges of $C_1$ or none and, if it contains edges of $C_1,$ exactly one of those belongs to $M.$
As before, we may assume $\Complement{Y}$ to contain at least two of the three cycles $C_1,$ $C_2,$ and $C_3.$
Suppose $Y$ contains vertices of both $P_1$ and $P_2.$
From the observation above it follows that $Y$ contains at most one of the four endpoints of the two paths.
If there exists $j\in[2]$ such that $\Cut{Y}\cap E(P_j)$ contains two non-$M$-edges, then it must contain at least two edges of the perfect matching $M_{3-j}.$
Thus, each $P_i$ has at least one edge in $\Cut{Y}$ and if $\Cut{Y}$ has at least two edges of $P_i,$ then these are exactly two, one of them belonging to $M.$
It follows that there exists $j\in[2]$ such that $\Cut{Y}$ contains exactly two edges of $P_j$ and exactly one edge of $P_{3-j}$ which does not belong to $M.$
However, this also means that  $\Cut{Y}$ contains an edge of $M_j\cap E(P_{3-j})$ and the edge in $E(P_j)\cap M\cap M_j$ contradicting $\Cut{Y}$ to be a tight cut.
Hence, contracting $Y$ into a single vertex preserves the existence of a conformal cross over $C_1$ (or what remains\footnote{Which is at least a cycle of length four.} of $C_1$ after contracting $Y$).
Thus, the assumption that $\mathscr{S}$ has a cross implies that $\mathscr{S}'$ has a cross.
This completes the proof of \cref{@contrasentidos}.

Hence, we may assume $(B',\Psi')$ to be matching flat.
Let $\rho'$ be a vortex-free rendition of $(B',\Psi')$ in the disk.
If $V(C_1)\cap Y=\emptyset$ we can obtain a vortex-free rendition of $(B,\Psi)$ by replacing all occurrences of $v_X$ in $\rho'$ by a big vertex in whose interior we draw all the vertices of $X.$
Thus, in this case we are done and therefore we may assume $V(C_1)\cap X\neq\emptyset.$
However, in this case we know that $X\subseteq V(H)\setminus V(H^+)$ by \cref{lemma_tightcutsandconformalcycles} as otherwise this would contradict the fact that $(H,\Psi)$ is matching flat.
So we may use the vortex-free rendition of $(H,\Psi)$ in the disk to obtain from $\rho'$ a vortex-free rendition of $(B,\Psi)$ in the disk by adjusting $\rho'$ to also properly embed the vertices of $X.$

Hence, from now on we may assume $B$ to be a brace.\medskip

From now on let $\mathscr{S}^-=(B',H',\Psi',\Omega)$ be the \textsl{reduct} of $\mathscr{S}.$

\begin{claim}\label{@extravaganza}
	$B^+$ contains $K_{3,3}$ as a matching minor if and only if $B$ contains $K_{3,3}$ as a matching minor.
\end{claim}

\noindent\emph{Proof of Claim 2:}
Let $\hat{B}^+$ be the unique brace of $B^+$ that contains $B_{\mathscr{S}}.$
Notice that $\hat{B}^+$ must still contain the artificial $4$-cycle $C=(a,b,c,d)$ for which   $B^+-C=B'$ holds.
This is a consequence of \cref{lemma_tightcutsandconformalcycles}.
Since every non-trivial tight cut of $B^+$ must contain edges of $C_2$ and one of its shores must be completely contained within $H',$ which does not contain $K_{3,3}$ as a matching minor, it follows that $\hat{B}^+$ contains a $K_{3,3}$-matching minor if and only if $B^+$ does so.

So instead of \cref{@extravaganza} it suffices to show \cref{@memorialization}.

\begin{claim}\label{@memorialization}
	$\hat{B}^+$ contains $K_{3,3}$ as a matching minor if and only if $B$ contains $K_{3,3}$ as a matching minor.
\end{claim}

\noindent\emph{Proof of Claim 3:}
Since $\hat{B}^+$ is a matching minor of $B,$ any matching minor of $\hat{B}^+$ is also one of $B.$
Hence, we may assume $\hat{B}^+$ to be $K_{3,3}$-matching-minor-free.
By \cref{thm_trisums} it follows that $\hat{B}^+$ can be obtained from planar braces via trisums.
We need to show that $C$ bounds a face in one of these braces in order to obtain matching flatness for  
the matching society $(\hat{\Psi},\hat{B}^+).$

Once we have established this we may use the resulting matching rendition $\hat{\rho}$ together with \cref{lemma_joiningoverlappingsocieties} to extend the matching rendition $\tilde{\rho}$ of $(H,\Omega),$ which is guaranteed by the definition of reinforced societies, to one of $(B,\Psi).$
This is possible since $C_3$ does not contain an edge of any non-trivial tight cut of $B^+$ by \cref{lemma_tightcutsandconformalcycles} and thus $\hat{\rho}$ contains a matching rendition of $(B-(H-V(\Omega)),\Omega)$ which can be combined with $(H,\Omega).$
From the existence of a matching rendition of $(B,\Psi)$ it follows that $B$ may be obtained from planar braces by trisums and thus, by \cref{thm_trisums}, $B$ does not contain $K_{3,3}$ as a matching minor which completes the proof of \cref{@memorialization}.

Therefore, the next step is to prove \cref{@pronunciarle}.

\begin{claim}\label{@pronunciarle}
	Let $G_1,\dots,G_{\ell}$ be planar braces such that $\hat{B}^+$ is isomorphic to a graph that can be obtained from the $G_i$ by trisums.
	Then there exists $i\in[\ell]$ such that $C\subseteq G_i$ and $C$ bounds a face of $G_{i}.$
\end{claim} 

\noindent\emph{Proof of Claim 4:}
Suppose there exists a set $S\subseteq V(\hat{B}^+)$ such that $\Abs{S\cap V_1}=\Abs{S\cap V_2}=2$ and $\hat{B}^+$ contains two distinct components, both of which contain a vertex of $C.$
First, observe that this is possible only if there exists $j\in[2]$ such that $S\cap V(C)=V_j\cap V(C).$
Moreover, notice that there exist two internally vertex disjoint paths between the two neighbours of the vertices in $V_{3-j}\cap V(C)$ in what remains of $H'$ in $\hat{B}^+.$
Hence, $S$ must contain a vertex from each of these paths.
Indeed, notice that two such paths exist on the remainder $\hat{C_2}$ of $C_2$ within $\hat{B}^+.$
Without loss of generality let us assume $j=2$ and let $a_1,a_2$ be the two vertices of $V(C)\setminus S\subseteq V_1$ while $\Set{b_1,b_2}=S\cap V(C).$
Let $v_1,v_2\in V_2$ be the neighbours of $a_1,a_2$ on the remainder of $C_2$ respectively.
Finally, let $s_1,s_2\in S\cap V_1$ the two remaining vertices of $S.$
Notice that $K\coloneqq \hat{B}^++b_1s_1+b_1s_2+b_2s_1+b_2s_2$ must still be $K_{3,3}$-matching-minor-free if $S$ is the vertex set of some four cycle that occurs in the construction of $\hat{B}^+$ from the $G_i.$
Without loss of generality, we may assume that $s_1$ and $s_2$ lie on $\hat{C_2}.$
Indeed, we may even further assume that there is some perfect matching $\hat{M}$ of $\hat{B}^+,$ which can be obtained from a perfect matching $M$ of $B$ which contains the canonical matching of  the  cylindrical matching grid of order three that spans $C_1,$ $C_2,$ and $C_3,$ for which the two disjoint paths that connect $v_1$ to $s_1$ and $v_2$ to $s_2$ respectively are $\hat{M}$-conformal.
Notice that $s_1$ and $s_2$ must be vertices of the $\hat{M}$-conformal paths which are the bisubdivisions of the two matching edges of $C_2$ which cover $v_1$ and $v_2.$
This holds as $v_1$ and $v_2$ were selected to be degree three vertices of the bisubdivision of $CG_3.$
Now we may change $\hat{M}$ to the perfect matching $N$ as follows.
Let $a_1b_1\in N$ as well as $a_2v_2$ and $s_2b_2.$
Furthermore, add all edges of $\hat{M}\cap E(\hat{B}^+)\setminus E(\hat{C_2}).$
Finally, note that $\hat{C_2}-v_2-s_2$ consists of the disjoint paths of odd length since $v_2$ and $s_2$ belong to different colour classes.
So finally we add the edges of the unique perfect matchings of these two paths.
It is straightforward to confirm that we may now find an $N$-conformal cross over the cycle $C$ in $K.$
By \cref{lemma_goodcrossesmeanK33} this means that $K$ contains $K_{3,3}$ as a matching minor.
This is a contradiction and thus the proof of \cref{@pronunciarle} is complete.\medskip

We are now ready to finalise the proof of our theorem.
Suppose $B$ contains $K_{3,3}$ as a matching minor.
Then, by \cref{@memorialization}, so does $\hat{B}^+.$
Hence, \cref{thm_4cycleK33} provides a conformal bisubdivision of $K_{3,3}$ in $\hat{B}^+$ which contains $C$ as a subgraph.
\Cref{lemma_goodcrossesmeanK33} in turn yields a conformal cross over $C$ in $\hat{B}^+.$
Since $\hat{B}^+$ is a matching minor of $B$ and $C,$ in particular, is represented by the cycle $C_1,$ we may expand this conformal cross over $C$ to a conformal cross over $C_1$ in $B$ and thus $\mathscr{S}$ has a cross.

Suppose now that  $B$ does not contain $K_{3,3}$ as a matching minor. 
Then, $(B,\Psi)$ is matching flat as a consequence of \cref{@pronunciarle} and \cref{lemma_joiningoverlappingsocieties},  described above.
To complete the proof, it remains to show that, in this situation, $\mathscr{S}$ does not have a cross.
Suppose there exists a perfect matching $M$ together with two disjoint paths $P_1,$ $P_2$ such that $K=C_1+P_1+P_2$ is an $M$-conformal subgraph of $B$ and $P_1$ and $P_2$ form a cross over $C_1.$
Let us delete from $B$ everything that belongs to $V(H)\setminus (V(H^+)\setminus V(K))$ and everything that belongs to $E(H)\setminus (E(H^+)\setminus E(K)).$
Let $G$ be the unique brace of the resulting graph that contains $B_{\mathscr{S}}$ and notice that there exists a $4$-cycle $C'$ which is all that remains of $C_1.$
Hence, $G$ is a brace with a $4$-cycle $C'$ and a conformal cross over $C'.$
By \cref{thm_4cycleK33} this means that $G$ contains $K_{3,3}$ as a matching minor.
This, however, implies that also $B$ contains a $K_{3,3}$-matching minor, which  contradicts our assumption.
\end{proof}

%%%%%%%%%%%%%%%%%%%%%%%%%%%%%%%%%%%%%%%%%%%%%%%%%%%%%%%%%%%%%%%%%%%%%%%%%%%%%%%%%%%%%%%%%%%%%%%%%%%%%%%%%%%%%%%%%%%%%%%%%%%%%%%%%%%%%%%%%%%%%%%%%%%%%%%%%%%%%%%%%%%%%%%%%%%%%%%%%%%%%%%%%%%%%%%%%%%%%%%%%%%%%%%%%%%%%%%%%%%%%%%%%%%%%%%%%%%%%%%%%%%%%%%%%%%%%%%%%%%%%%%%%%%%%%%%%%%%%%%%%%%%%%%%

\section{Proof of the structure theorem}
\label{sec_proof_struct}

Recall the four steps for the proof of the structure theorem for SCM-free graphs.
The tight cut decomposition provides the means necessary for a matching theoretic analogue of step ii), which was the decomposition of the graph into its quasi-$4$-connected components.
While braces are not as highly connected as these components, for the purposes of our proof they provide enough rigidity.
%Step iii), which incorporates the tree decomposition of the quasi-$4$-connected components, is only necessary for algorithmic purposes.
%We deal with this step in the form of \cref{thm_splicegeneratingfunctions} which allows us to reduce the problem of counting the perfect matchings of a bipartite matching covered graph to counting the perfect matchings of its braces.

So what is left to do are the steps i), iii), and iv), where i) is the identification of a universal obstruction, iii) is a refinement of \cref{thm_matchinggrid} to find a structure that is \hyperref[def_matchingflat]{matching flat} and contains a large wall, and finally iv) which shows that also the brace outside this flat area must be matching flat.
We combine steps iii) and iv) into a single step. The proofs for this step can be found in \cref{subsec_flatgrid}.

\subsection{A universal obstruction}\label{subsec_obstruction}

Let us start by formally introducing our obstruction.

\begin{definition}[single-crossing matching grid]\label{def_singlecrossingmatchinggrid}
Let $k\in\N$ be some positive integer.
The \emph{single-crossing matching grid} of order $k$ is obtained from the $(2k\times 2k)$-grid by identifying the vertices (and edges) of the $4$-cycle $((k,k)(k,k+1)(k+1,k+1)(k+1,k))$ with the vertices (and edges) of a $4$-cycle of $K_{3,3}.$
\end{definition}

For an illustration see \cref{fig_singlecrossingmatchinggrid}.
We say that a bipartite graph $B$ is a \emph{single-crossing matching minor} (or \emph{SCMM} for short) if there exists some $k$ such that $B$ is a matching minor of the single-crossing matching grid of order $k.$

Similar to the case of SCM-free graphs, we will encounter the single-crossing matching grid not directly in our proof.
Instead, in many cases one of three equivalent graphs might appear.

Let $k\in\N$ be a positive integer
The \emph{inside-out single-crossing matching grid} of order $k$ is obtained from the $(2k\times 2k)$-grid and $K_{3,3}$ by selecting a $4$-cycle in $K_{3,3},$ identifying its vertices with the vertices $(1,1),$ $(1,2k),$ $(2k,2k),$ and $(2k,1)$ and removing its edges.
The \emph{cylindrical single-crossing grid} of order $k$ is obtained from the cylindrical matching grid of order $4k$ by adding the edges $v^{1}_{1}v^{1}_{8k+4}$ and $v^{1}_{4k+1}v^{1}_{12k+4}.$
Finally, the \emph{cylindrical single-jump grid} of order $k$ is obtained from the cylindrical matching grid of order $k$ by adding the edge $v^k_1v^1_2$ (see \cref{fig_gridbrothers}).

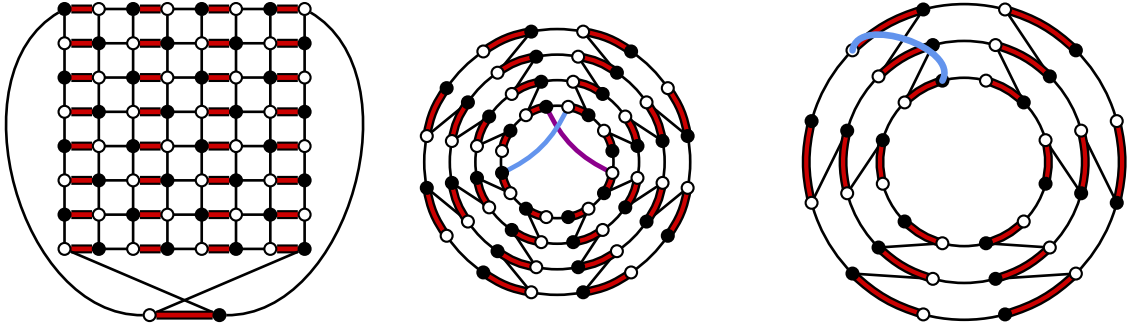
\begin{figure}[ht]
		\centering
\makefast{% !TEX root = ../single_comb.tex
% !TeX spellcheck = en_US	
\begin{subfigure}{0.32\textwidth}
		\centering
	\begin{tikzpicture}[scale=0.76]
		\pgfdeclarelayer{background}
		\pgfdeclarelayer{foreground}
		\pgfsetlayers{background,main,foreground}
		\node[v:ghost] (C) {};
		
		\foreach \x in {1,3,5,7}
		\foreach \y in {2,4,6,8}
		{	
			%			\pgfmathsetmacro\X{}
			\node[v:main] (v_\x_\y) at (\x*0.6,\y*0.6){};
		}
		
		\foreach \x in {2,4,6,8}
		\foreach \y in {1,3,5,7}
		{	
			%			\pgfmathsetmacro\X{}
			\node[v:main] (v_\x_\y) at (\x*0.6,\y*0.6){};
		}
		
		\foreach \x in {1,3,5,7}
		\foreach \y in {1,3,5,7}
		{	
			%			\pgfmathsetmacro\X{}
			\node[v:mainempty] (v_\x_\y) at (\x*0.6,\y*0.6){};
		}
		
		\foreach \x in {2,4,6,8}
		\foreach \y in {2,4,6,8}
		{	
			%			\pgfmathsetmacro\X{}
			\node[v:mainempty] (v_\x_\y) at (\x*0.6,\y*0.6){};
		}
		
		\node[v:main,position=285:12mm from v_5_1] (a) {};
		\node[v:mainempty,position=255:12mm from v_4_1] (b) {};
		
		\begin{pgfonlayer}{background}
			
			\foreach \x in {1,...,8}
			{
				\draw[e:main] (v_\x_1) to (v_\x_8);
				\draw[e:main] (v_1_\x) to (v_8_\x);
			}
			
			\foreach \z in {1,...,8}
			{
				\draw[e:coloredborder] (v_1_\z) to (v_2_\z);
				\draw[e:coloredborder] (v_3_\z) to (v_4_\z);
				\draw[e:coloredborder] (v_5_\z) to (v_6_\z);
				\draw[e:coloredborder] (v_7_\z) to (v_8_\z);
			}
			
			\foreach \z in {1,...,8}
			{
				\draw[e:colored] (v_1_\z) to (v_2_\z);
				\draw[e:colored] (v_3_\z) to (v_4_\z);
				\draw[e:colored] (v_5_\z) to (v_6_\z);
				\draw[e:colored] (v_7_\z) to (v_8_\z);
			}
			
			\draw[e:coloredborder] (a) to (b);
			\draw[e:main] (b) to (v_8_1);
			\draw[e:main,bend left=75] (b) to (v_1_8);
			\draw[e:main] (a) to (v_1_1);
			\draw[e:main,bend right=75] (a) to (v_8_8);
			
			\draw[e:colored] (a) to (b);
			
		\end{pgfonlayer}
	\end{tikzpicture}
	\end{subfigure}
	\begin{subfigure}{0.32\textwidth}
		\centering
			\begin{tikzpicture}[scale=0.68]

			\pgfdeclarelayer{background}
			\pgfdeclarelayer{foreground}
			\pgfsetlayers{background,main,foreground}

			\draw[e:main] (0,0) circle (11mm);
			\draw[e:main] (0,0) circle (16mm);
			\draw[e:main] (0,0) circle (21mm);
			\draw[e:main] (0,0) circle (26mm);
			
			\foreach \x in {1,...,4}
			{
				\draw[e:main] (\x*90-11.25:16mm) -- (\x*90-22.5-11.25:11mm);
				\draw[e:main] (\x*90-11.25:21mm) -- (\x*90-22.5-11.25:16mm);
				\draw[e:main] (\x*90-11.25:26mm) -- (\x*90-22.5-11.25:21mm);
			}
			
			\foreach \x in {1,...,4}
			{
				\draw[e:main] (\x*90-67.5-11.25:16mm) -- (\x*90-45-11.25:11mm);
				\draw[e:main] (\x*90-67.5-11.25:21mm) -- (\x*90-45-11.25:16mm);
				\draw[e:main] (\x*90-67.5-11.25:26mm) -- (\x*90-45-11.25:21mm);
				
			}
		
				\draw[e:colored,color=DarkMagenta,bend left=20] (90-11.25-67.5-22.5:11mm) to (270-22.5-11.25-45-67.5-22.5:11mm);
				\draw[e:colored,color=CornflowerBlue,bend right=20] (-11.25-67.5-67.5-22.5:11mm) to (180-11.25-67.5-22.5:11mm);
			
			\foreach \x in {1,...,8}
			{
				\draw[e:coloredthin,color=BostonUniversityRed,bend right=13] (\x*45-22.5-11.25:11mm) to (\x*45-11.25:11mm);
				\draw[e:coloredthin,color=BostonUniversityRed,bend right=13] (\x*45-22.5-11.25:16mm) to (\x*45-11.25:16mm);
				\draw[e:coloredthin,color=BostonUniversityRed,bend right=13] (\x*45-22.5-11.25:21mm) to (\x*45-11.25:21mm);
				\draw[e:coloredthin,color=BostonUniversityRed,bend right=13] (\x*45-22.5-11.25:26mm) to (\x*45-11.25:26mm);
			}
			
			\foreach \x in {1,...,8}
			{
				\node[v:main] () at (\x*45-22.5-11.25:11mm){};
				\node[v:main] () at (\x*45-22.5-11.25:16mm){};
				\node[v:main] () at (\x*45-22.5-11.25:21mm){};
				\node[v:main] () at (\x*45-22.5-11.25:26mm){};
				\node[v:mainempty] () at (\x*45-11.25:11mm){};
				\node[v:mainempty] () at (\x*45-11.25:16mm){};
				\node[v:mainempty] () at (\x*45-11.25:21mm){};
				\node[v:mainempty] () at (\x*45-11.25:26mm){};
			}

			\begin{pgfonlayer}{background}
				\foreach \x in {1,...,8}
				{
					\draw[e:coloredborder,bend right=13] (\x*45-22.5-11.25:11mm) to (\x*45-11.25:11mm);
					\draw[e:coloredborder,bend right=13] (\x*45-22.5-11.25:16mm) to (\x*45-11.25:16mm);
					\draw[e:coloredborder,bend right=13] (\x*45-22.5-11.25:21mm) to (\x*45-11.25:21mm);
					\draw[e:coloredborder,bend right=13] (\x*45-22.5-11.25:26mm) to (\x*45-11.25:26mm);
				}
			\end{pgfonlayer}
			
		\end{tikzpicture}
	\end{subfigure}
	\begin{subfigure}{0.32\textwidth}
		\centering
		\begin{tikzpicture}[scale=0.70]

			\pgfdeclarelayer{background}
			\pgfdeclarelayer{foreground}
			\pgfsetlayers{background,main,foreground}

			\draw[e:main] (0,0) circle (16mm);
			\draw[e:main] (0,0) circle (23mm);
			\draw[e:main] (0,0) circle (30mm);
			
			\foreach \x in {1,...,3}
			{
				\draw[e:main] (\x*120-15:23mm) -- (\x*120+30-15:16mm);
				\draw[e:main] (\x*120-15:30mm) -- (\x*120+30-15:23mm);
			}
			
			\foreach \x in {1,...,3}
			{
				\draw[e:main] (\x*120-30-15:23mm) -- (\x*120-60-15:16mm);
				\draw[e:main] (\x*120-30-15:30mm) -- (\x*120-60-15:23mm);
				
			}
			
			\foreach \x in {1,...,6}
			{
				\draw[e:coloredthin,color=BostonUniversityRed,bend right=13] (\x*60-15:16mm) to (\x*60+30-15:16mm);
				\draw[e:coloredthin,color=BostonUniversityRed,bend right=13] (\x*60-15:23mm) to (\x*60+30-15:23mm);
				\draw[e:coloredthin,color=BostonUniversityRed,bend right=13] (\x*60-15:30mm) to (\x*60+30-15:30mm);
			}
			
			\foreach \x in {1,...,6}
			{
				\node[v:main] () at (\x*60-15:16mm){};
				\node[v:main] () at (\x*60-15:23mm){};
				\node[v:main] () at (\x*60-15:30mm){};
				\node[v:mainempty] () at (\x*60+30-15:16mm){};
				\node[v:mainempty] () at (\x*60+30-15:23mm){};
				\node[v:mainempty] () at (\x*60+30-15:30mm){};
			}
			
			\draw[e:colored,color=CornflowerBlue,bend right=100,line width=2.5pt] (105:16mm) to (135:30mm);
			
			\begin{pgfonlayer}{background}

				\foreach \x in {1,...,6}
				{
					\draw[e:coloredborder,bend right=13] (\x*60-15:16mm) to (\x*60+30-15:16mm);
					\draw[e:coloredborder,bend right=13] (\x*60-15:23mm) to (\x*60+30-15:23mm);
					\draw[e:coloredborder,bend right=13] (\x*60-15:30mm) to (\x*60+30-15:30mm);
				}
			\end{pgfonlayer}
			
		\end{tikzpicture}
	\end{subfigure}}{\scalebox{.2}{\includegraphics{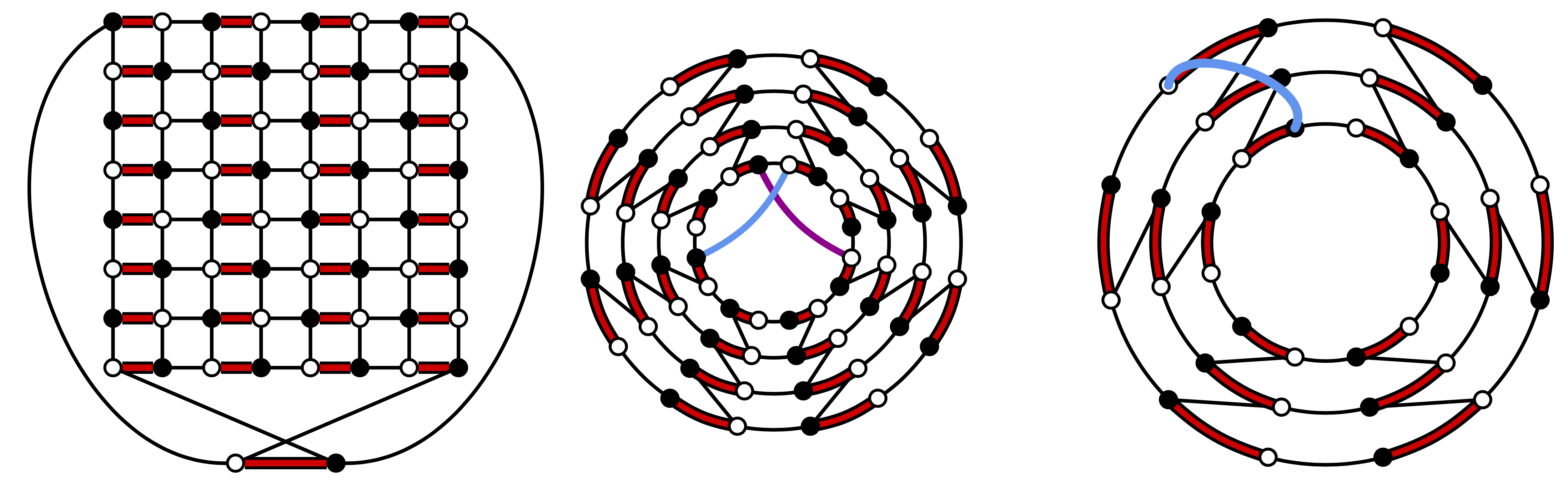}}}
	\caption{The inside-out single-crossing matching grid of order $4$ (left), the cylindrical single-crossing grid of order $1$ (middle), and the cylindrical single-jump grid of order $3$ (right).}
	\label{fig_gridbrothers}
\end{figure}

In what follows we show that both, the cylindrical single-crossing grid and the cylindrical single-jump grid, if chosen to be of adequate size, contain a large inside-out single-crossing  matching grid as a matching minor.
In a last step we then show that a large inside-out single-crossing matching grid contains a still large single-crossing matching grid as a matching minor.
In later stages of the proof this will give us the right to stop whenever we encounter one of the four structures.

As an intermediate step, let us first show that we can obtain a $(2k\times 2k)$-grid as a matching minor in a sufficiently large cylindrical matching grid such that only a small part of the cylindrical grid is used.
This will allow us to then make use of the remaining infrastructure to attach the crossing paths to the four corners of the grid.

Let, for some $k\in\N,$ $CG_k$ be the \hyperref[def_matchinggrid]{cylindrical matching grid} of order $k$ together with its canonical matching $M.$
Moreover, let $h\leq k$ be another positive integer.
A \emph{slice} $S$ of $CG_k$ of \emph{breadth $h$} is a subgraph of $CG_k$ induced by $h$ consecutive cycles $C_i,\dots, C_{i+h-1}$ of $CG_k.$
We say that $S$ \emph{starts at $i$}.
If $B$ is a bisubdivision of $CG_k,$ we denote by $B_{[i,h]}$ the subgraph of $B$ corresponding to the slice of $CG_k$ of breadth $h$ starting at $i.$

Now suppose  that $k$ is an even number.
Then we denote by $M^*$ the perfect matching of $CG_k$ which is obtained from $M$ by exchanging $M\cap E(C_i)$ with $E(C_i)\setminus M$ for all odd $i\in[k].$
We say that $M^*$ is the \emph{alternating matching} of $CG_k.$
Moreover, if $M$ is the matching of $B$ corresponding to the canonical matching of $CG_k,$ then by $M^*$ we denote the perfect matching of $B$ which corresponds to the alternating matching of $CG_k.$

Note that $CG_k$ contains $k$ pairs $\mathscr{P}_i=(P_i,P_{i+1})$ of disjoint conformal paths where $P_i=(v^1_{2i},v^1_{2i-1},v^2_{2i},v^2_{2i-1},\dots,v^k_{2i},v^k_{2i-1})$ and $P_{i+1}=(v^1_{2i+1},v^1_{2i+2},v^2_{2i+1},v^2_{2i+2},\dots,v^k_{2i+1},v^k_{2i+2})$ for every $i\in[k].$
We call $\mathscr{P}_i$ the \emph{$i$th row of $CG_k$} and extend this definition in the natural way to bisubdivisions of $CG_k.$
Moreover, we denote by $V(\mathscr{P}_i)$ the vertex set $V(P_i)\cup V(P_{i+1}).$ 

\begin{lemma}\label{lemma_gridinthecylinder}
Let $k\in\N$ be some positive integer and let $G$ be a cylindrical matching grid of order $h\geq 8k.$
Moreover, let $S$ be a slice of breadth $2k$ of $G$ and consider the alternating matching $M^*$ of $CG_k.$
Next, let $\ell\in[h-8k+1]$ be some integer.
Finally, let $S'$ be the maximal $M^*$-conformal subgraph of $S$ which is contained in the subgraph of $S$ induced by the vertices $(\bigcup_{i=\ell}^{\ell+8k-1}V(\mathscr{P}_i))\cap V(S).$
Then $S'$ contains the $(2k\times 2k)$-grid as a matching minor which spans the entirety of $S'.$
\end{lemma}

\begin{proof}
In case $k=1$ a single $M^*$-conformal cycle suffices.
Observe that such in this case $S'$ contains such a cycle as a spanning subgraph.
Now suppose we are given $S'$ for the case $k+1.$
Imagine that $S'$ is drawn as in \cref{@darstellungsmittel} and observe that, starting from the upper left corner, one can find an $M^*$-conformal subgraph $S''\subseteq S'$ which meets the requirements for the case $k.$
In \cref{@darstellungsmittel} this is the subgraph in the pink box.
By induction, we may assume that $S''$ contains a spanning matching minor model of the $(2k\times 2k)$-grid as a conformal subgraph.
Moreover, let us assume this model was created by the inductive procedure we outline here.
Compare the subgraph in green from \cref{@darstellungsmittel}.
We now complete this model to a matching minor model of the $(2(k+1)\times2(k+1))$-grid as illustrated in \cref{@darstellungsmittel}.
Notice that the result is again a spanning subgraph of $S'.$
\begin{figure}[ht]
	\centering
\makefast{% !TEX root = ../single_comb.tex
% !TeX spellcheck = en_GB
	\begin{tikzpicture}[scale=1.15]
		\pgfdeclarelayer{background}
		\pgfdeclarelayer{foreground}
		\pgfsetlayers{background,main,foreground}
%		\node[v:ghost] (C) {};
		
		\foreach \x in {1,...,12}
		\foreach \y in {1,...,3}
		{	
			\node[v:mainempty] (v_\x_\y) at (\x*1.2,\y*0.8){};
			\node[v:main] (u_\x_\y) at (0.6+\x*1.2,\y*0.8){};
		}
		\foreach \x in {1,...,11}
		\foreach \y in {1,...,3}
		{	
			\node[v:main] (a_\x_\y) at (0.6+\x*1.2,\y*0.8-0.4){};
			\node[v:mainempty] (b_\x_\y) at (1.2+\x*1.2,\y*0.8-0.4){};
		}
	
		\begin{pgfonlayer}{background}
		
		\node[rectangle, draw = DarkMagenta,fill = DarkMagenta,opacity=0.15,minimum width = 10.85cm, minimum height = 1.9cm] (r) at (5.7,1.8) {};
		
		\draw[e:marker,color=Amber] (0.6+8*1.2,3*0.8) to (15,3*0.8);
		\draw[e:marker,color=Amber] (8*1.2,3*0.8-0.4) to (14.4,3*0.8-0.4);
		\draw[e:marker,color=Amber] (0.6+8*1.2,2*0.8) to (15,2*0.8);
		\draw[e:marker,color=Amber] (8*1.2,2*0.8-0.4) to (14.4,2*0.8-0.4);
		\draw[e:marker,color=Amber] (1*1.2,1*0.8) to (15,1*0.8);
		\draw[e:marker,color=Amber] (0.6+1*1.2,1*0.8-0.4) to (14.4,1*0.8-0.4);

		\draw[e:marker,color=AO,opacity=0.25] (1*1.2,3*0.8) to (0.6+8*1.2,3*0.8);
		\draw[e:marker,color=AO,opacity=0.25] (0.6+1*1.2,3*0.8-0.4) to (8*1.2,3*0.8-0.4);
		\draw[e:marker,color=AO,opacity=0.25] (1*1.2,2*0.8) to (0.6+8*1.2,2*0.8);
		\draw[e:marker,color=AO,opacity=0.25] (0.6+1*1.2,2*0.8-0.4) to (8*1.2,2*0.8-0.4);
		
		\foreach \x in {1,5}
		\foreach \y in {2,3}
		{
			\draw[e:marker,color=AO,opacity=0.25] (1.2*\x,0.8*\y) to (0.6+1.2*\x,0.8*\y-0.4);
		}
		\foreach \x in {4,8}
		\foreach \y in {2,3}
		{
			\draw[e:marker,color=AO,opacity=0.25] (1.2*\x+0.6,0.8*\y) to (1.2*\x,0.8*\y-0.4);
		}
		\foreach \x in {2,6}
		\foreach \y in {2}
		{
			\draw[e:marker,color=AO,opacity=0.25] (1.2*\x+0.6,0.8*\y+0.4) to (1.2*\x,0.8*\y);
		}
		\foreach \x in {3,7}
		\foreach \y in {2}
		{
			\draw[e:marker,color=AO,opacity=0.25] (1.2*\x,0.8*\y+0.4) to (0.6+1.2*\x,0.8*\y);
		}

		\foreach \x in {1,5,9}
		\foreach \y in {1}
		{
			\draw[e:marker,color=Amber] (1.2*\x,0.8*\y) to (0.6+1.2*\x,0.8*\y-0.4);
		}
		\foreach \x in {9}
		\foreach \y in {2,3}
		{
			\draw[e:marker,color=Amber] (1.2*\x,0.8*\y) to (0.6+1.2*\x,0.8*\y-0.4);
		}
		\foreach \x in {11}
		\foreach \y in {2}
		{
			\draw[e:marker,color=Amber] (1.2*\x,0.8*\y+0.4) to (0.6+1.2*\x,0.8*\y);
		}
		\foreach \x in {3,7,11}
		\foreach \y in {1}
		{
			\draw[e:marker,color=Amber] (1.2*\x,0.8*\y+0.4) to (0.6+1.2*\x,0.8*\y);
		}
		\foreach \x in {4,8,12}
		\foreach \y in {1}
		{
			\draw[e:marker,color=Amber] (1.2*\x+0.6,0.8*\y) to (1.2*\x,0.8*\y-0.4);
		}
		\foreach \x in {12}
		\foreach \y in {2,3}
		{
			\draw[e:marker,color=Amber] (1.2*\x+0.6,0.8*\y) to (1.2*\x,0.8*\y-0.4);
		}
		\foreach \x in {10}
		\foreach \y in {2}
		{
			\draw[e:marker,color=Amber] (1.2*\x+0.6,0.8*\y+0.4) to (1.2*\x,0.8*\y);
		}
		\foreach \x in {2,6,10}
		\foreach \y in {1}
		{
			\draw[e:marker,color=Amber] (1.2*\x+0.6,0.8*\y+0.4) to (1.2*\x,0.8*\y);
		}

		\foreach \x in {1,...,3}
		{
			\draw[e:mainthin] (1.2,\x*0.8) to (15,\x*0.8);
			\draw[e:mainthin] (1.8,\x*0.8-0.4) to (14.4,\x*0.8-0.4);
		}
		\foreach \x in {1,3,5,7,9,11}
		\foreach \y in {1,...,3}
		{
			\draw[e:mainthin] (1.2*\x,0.8*\y) to (0.6+1.2*\x,0.8*\y-0.4);
		}
		\foreach \x in {2,4,6,8,10,12}
		\foreach \y in {1,...,3}
		{
			\draw[e:mainthin] (1.2*\x+0.6,0.8*\y) to (1.2*\x,0.8*\y-0.4);
		}
		\foreach \x in {3,5,7,9,11}
		\foreach \y in {1,...,2}
		{
			\draw[e:mainthin] (1.2*\x,0.8*\y+0.4) to (0.6+1.2*\x,0.8*\y);
		}
		\foreach \x in {2,4,6,8,10}
		\foreach \y in {1,...,2}
		{
			\draw[e:mainthin] (1.2*\x+0.6,0.8*\y+0.4) to (1.2*\x,0.8*\y);
		}
		\foreach \x in {1,...,12}
		\foreach \y in {1,...,3}
		{
			\draw[e:coloredborder] (1.2*\x,0.8*\y) to (0.6+1.2*\x,0.8*\y);
			\draw[e:colored,color=BostonUniversityRed] (1.2*\x,0.8*\y) to (0.6+1.2*\x,0.8*\y);
		}
		\foreach \x in {1,...,11}
		\foreach \y in {1,...,3}
		{
			\draw[e:coloredborder] (0.6+1.2*\x,0.8*\y-0.4) to (1.2+1.2*\x,0.8*\y-0.4);
			\draw[e:colored,color=CornflowerBlue] (0.6+1.2*\x,0.8*\y-0.4) to (1.2+1.2*\x,0.8*\y-0.4);
		}
		
		\end{pgfonlayer}
\end{tikzpicture}}{
	\includegraphics[scale=0.57]{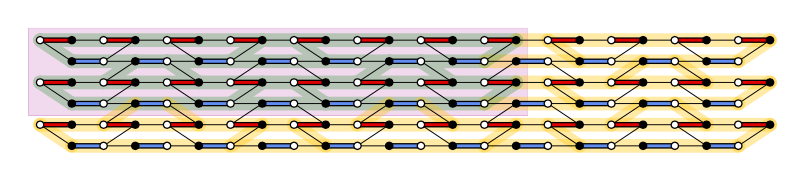}}
	\caption{An illustration for the cases $k=2$ (in the pink box) and $k+1=3$ for the construction from \cref{lemma_gridinthecylinder}. The orange part shows how to extend the matching minor model from the case $k$ to the case $k+1.$}
	\label{@darstellungsmittel}
\end{figure}
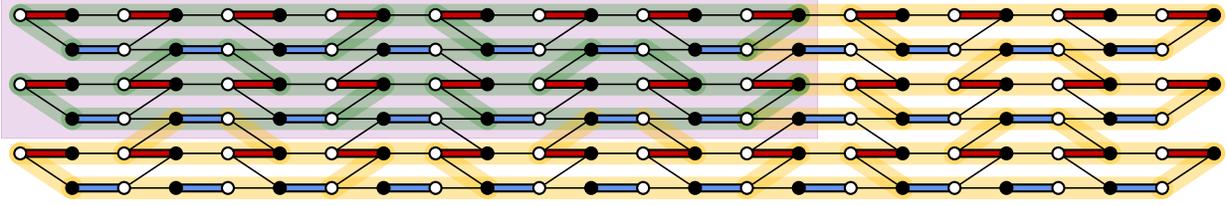
\end{proof}

\begin{lemma}\label{lemma_singlecrossingcylinder}
Let $k\in\N$ be some positive integer and let $G$ be a cylindrical single-crossing matching grid of order $h\geq 8k.$
Then $G$ contains the inside-out matching grid of order $k$ as a matching minor.\end{lemma}

\begin{proof}
	Notice that, since $G$ is a cylindrical single-crossing matching grid of order $h,$ it consists of the cylindrical matching grid $CG_{4h}$ together with two crossing edges $e_1$ and $e_2.$
	Let $S$ be the slice of breath $2k$ of $CG_{4h}$ which starts at $9.$
	Moreover, consider the rows $\mathscr{P_i}$ for $i\in[1,8k].$
	By \cref{lemma_gridinthecylinder}, $\InducedSubgraph{S}{V(\bigcup_{i=1}^{8k}V(\mathscr{P}_i))}$ contains an $M^*$-conformal subgraph $S'$ which hosts a spanning matching minor model of the $(2k\times 2k)$-grid.
	Observe that $e_1$ and $e_2$ each contain an endpoint $x_1,$ $x_2$ respectively, such that the subpath $Q$ of $C_1,$ which is otherwise disjoint from $e_1$ and $e_2,$ is $M^*$-conformal.
	It suffices now to find the correct internally $M^*$-conformal paths that connect each endpoint of $Q$ to two non-adjacent corners of the grid-matching minor in $S'.$
	Since $e_1$ and $e_2$ are crossing, all of their endpoints are mutually far apart within $CG_k,$ and $S'$ starts with the cycle $C_{8k}$ it is relatively straight forward to find these paths.
	See \cref{fig_cylindercross} for an example.
	\begin{figure}[ht]
		\centering
		\makefast{\input{figures/cylindercross.tex}}{\includegraphics[scale=0.48]{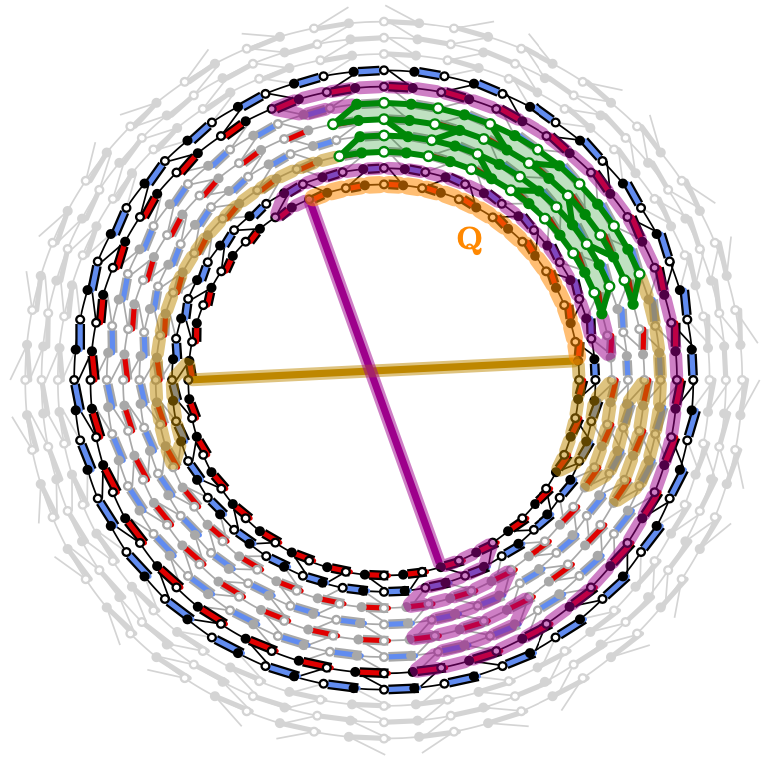}}
		\caption{A model of the inside-out single-crossing matching grid of order $2$ as a matching minor within the cylindrical single-crossing grid of order $8.$}
		\label{fig_cylindercross}
	\end{figure}
\end{proof}

\begin{lemma}\label{lemma_jumpinggrid}
Let $k\in\N$ be some positive integer and let $G$ be a cylindrical single-jump grid of order $8(k+2).$
Then $G$ contains the inside-out single-crossing matching grid of order $k$ as a matching minor.
\end{lemma}

\begin{proof}
The construction is quire similar to the one from \cref{lemma_singlecrossingcylinder}.
First notice that there exists a unique edge $e$ in $G$ such that $G-e$ is the cylindrical matching grid of order $8(k+2).$
Let $M$ be its canonical matching and $M^*$ be its alternating matching.%\sed{No pic? no...}
Let $S$ be the slice of breath $2k$ of $G-e$ which starts at $9.$
By \cref{lemma_gridinthecylinder}, $\InducedSubgraph{S}{V(\bigcup_{i=1}^{8k}V(\mathscr{P}_i))}$ contains an $M^*$-conformal subgraph $S'$ which hosts a spanning matching minor model of the $(2k\times 2k)$-grid.
Let $Q$ be the unique $M^*$ conformal subpath of $C_{1}$ whose first edge belongs to $\mathscr{P}_1,$ whose last edge belongs to $\mathscr{P}_{8k},$ and which is disjoint from $V(\mathscr{P}_9).$
Notice $Q$ shares an endpoint with $e.$
All that is left to do is to find four pairwise internally $M^*$ conformal paths which are internally disjoint from $S'$ such that two of them link the endpoint $Q$ shares with $e$ to the two corners of the $(2k\times 2k)$-grid-matching minor in $S'$ that lie in $V_2,$ while the other two paths joint the other endpoint of $Q$ to the remaining two corners.
As $e$ connects $C_1$ to $C_{8(k+2)}$ and $S'$ is disjoint from the innermost and the outermost $C_i$ while also avoiding at least $8$ rows of $G-e$ this is clearly possible.
\end{proof}

It remains to discuss how to find the single-crossing matching grid within a sufficiently large inside-out single-crossing matching grid.

\begin{lemma}\label{lemma_reversinginsideout}
Let $k\in\N$ be some positive integer and let $G$ be an inside-out single-crossing matching grid of order $h\geq 2k,$ then $G$ contains the single-crossing matching grid of order $k$ as a matching minor.
\end{lemma}

\begin{proof} 
Without loss of generality we may assume $h=2k.$
Let $e$ be the unique edge which does not have an endpoint in the $(4k\times4k)$-grid $G'$ in $G.$
%Let $G'$ be the conformal subgraph of $G-V(e)$ isomorphic to the $(4k \times 4k)$-grid after the removal of the two outer most cycles.
It suffices to show that $G'$ contains a matching minor model of the $(2k)\times2k)$-grid $G''$ such that the outermost cycle of $G'$ is mapped to the central $4$-cycle of $G''$ in such a way that the four paths which contain the four corners of the grid will be contracted into vertices of the same colour class of their respective corner.
We fix a particular perfect matching of $G'$:
Removing the outer cycle $C_1$ of $G'$ leaves a graph which is isomorphic to the $((4k-2)\times(4k-2))$-grid.
Hence, $C_1$ is a conformal cycle of $G'.$
Moreover, this procedure can be repeated $2k$ times to obtain a family of pairwise disjoint cycles $C_1,\dots,C_{2k}.$
In fact, this cycle family is a $2$-factor of $G'.$
For each of these cycles we may now choose a perfect matching $M_i$ such that $\bigcup_{i=1}^{2k}M_i$ is a perfect matching of $G'.$
Moreover, let the $M_i$ be chosen such that the two endpoints of any $e'\in M_i$ are adjacent to endpoints of two different edges of $M_{i+1}$ for all $i\in[2k-1].$
Finally, let $M_1$ be chosen such that the path between the upper two corners is internally $M_1$-conformal.
See \cref{fig_insideoutcross} for an illustration.
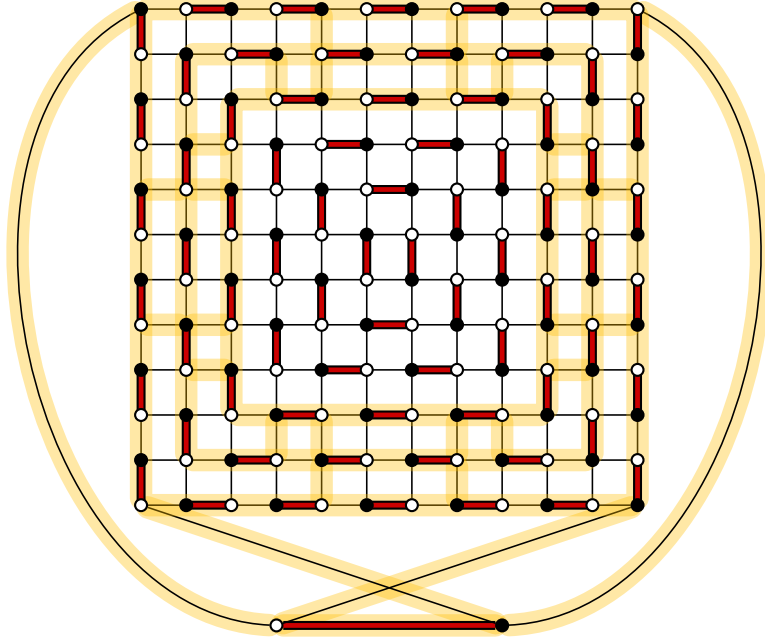
\begin{figure}[ht]
	\centering
	\makefast{	% !TEX root = ../single_comb.tex
% !TeX spellcheck = en_GB
	\begin{tikzpicture}[scale=1]
		\pgfdeclarelayer{background}
		\pgfdeclarelayer{foreground}
		\pgfsetlayers{background,main,foreground}
		\node[v:ghost] (C) {};

		\foreach \x in {1,3,5,7,9,11}
		\foreach \y in {1,3,5,7,9,11}
		{
			\node[v:mainempty] (v_\x_\y) at (0.6*\x,0.6*\y) {};
		}
	
		\foreach \x in {2,4,6,8,10,12}
		\foreach \y in {2,4,6,8,10,12}
		{
			\node[v:mainempty] (v_\x_\y) at (0.6*\x,0.6*\y) {};
		}
	
		\foreach \x in {2,4,6,8,10,12}
		\foreach \y in {1,3,5,7,9,11}
		{
			\node[v:main] (v_\x_\y) at (0.6*\x,0.6*\y) {};
		}
	
		\foreach \x in {1,3,5,7,9,11}
		\foreach \y in {2,4,6,8,10,12}
		{
			\node[v:main] (v_\x_\y) at (0.6*\x,0.6*\y) {};
		}

		\node[v:mainempty] (b) at (2.4,-1) {};
		\node[v:main] (a) at (5.4,-1) {};

		\begin{pgfonlayer}{background}
			
		\draw[e:marker,color=Amber] (a) to (b);
		\draw[e:marker,color=Amber] (a) to (v_1_1);
		\draw[e:marker,color=Amber] (b) to (v_12_1);
		
		\draw[e:marker,color=Amber,bend right=75] (a) to (v_12_12);
		\draw[e:marker,color=Amber,bend left=75] (b) to (v_1_12);

		\draw[e:coloredborder] (a) to (b);
		\draw[e:colored,color=BostonUniversityRed] (a) to (b);
		
		\draw[e:mainthin] (a) to (v_1_1);
		\draw[e:mainthin] (b) to (v_12_1);
		
		\draw[e:mainthin,bend right=75] (a) to (v_12_12);
		\draw[e:mainthin,bend left=75] (b) to (v_1_12);

		\foreach \x in {1,...,12}
		{
			\draw[e:mainthin] (v_\x_1) to (v_\x_12);
			\draw[e:mainthin] (v_1_\x) to (v_12_\x);
		}
	
		\draw [e:marker,color=Amber] (v_1_1) to (v_1_12) to (v_12_12) to (v_12_1) to (v_1_1);
		\draw [e:marker,color=Amber] (v_2_2) to (v_2_11) to (v_11_11) to (v_11_2) to (v_2_2);
		\draw [e:marker,color=Amber] (v_3_3) to (v_3_10) to (v_10_10) to (v_10_3) to (v_3_3);
		
		\draw[e:marker,color=Amber] (v_1_5) to (v_2_5);
		\draw[e:marker,color=Amber] (v_1_8) to (v_2_8);
		\draw[e:marker,color=Amber] (v_5_1) to (v_5_2);
		\draw[e:marker,color=Amber] (v_8_1) to (v_8_2);
		\draw[e:marker,color=Amber] (v_12_5) to (v_11_5);
		\draw[e:marker,color=Amber] (v_5_12) to (v_5_11);
		\draw[e:marker,color=Amber] (v_12_8) to (v_11_8);
		\draw[e:marker,color=Amber] (v_8_12) to (v_8_11);
		
		\draw[e:marker,color=Amber] (v_2_5) to (v_3_5);
		\draw[e:marker,color=Amber] (v_2_8) to (v_3_8);
		\draw[e:marker,color=Amber] (v_5_2) to (v_5_3);
		\draw[e:marker,color=Amber] (v_8_2) to (v_8_3);
		\draw[e:marker,color=Amber] (v_11_5) to (v_10_5);
		\draw[e:marker,color=Amber] (v_5_11) to (v_5_10);
		\draw[e:marker,color=Amber] (v_11_8) to (v_10_8);
		\draw[e:marker,color=Amber] (v_8_11) to (v_8_10);
		
		\draw[e:marker,color=Amber] (v_2_4) to (v_3_4);
		\draw[e:marker,color=Amber] (v_2_9) to (v_3_9);
		\draw[e:marker,color=Amber] (v_4_2) to (v_4_3);
		\draw[e:marker,color=Amber] (v_9_2) to (v_9_3);
		\draw[e:marker,color=Amber] (v_11_4) to (v_10_4);
		\draw[e:marker,color=Amber] (v_4_11) to (v_4_10);
		\draw[e:marker,color=Amber] (v_11_9) to (v_10_9);
		\draw[e:marker,color=Amber] (v_9_11) to (v_9_10);
		
		\foreach \x in {1,3,5,7,9,11}
		{
			\draw[e:coloredborder] (0.6*1,0.6*\x) to (0.6*1,0.6*\x+0.6);
			\draw[e:coloredborder] (0.6*12,0.6*\x) to (0.6*12,0.6*\x+0.6);
			
			\draw[e:colored,color=BostonUniversityRed] (0.6*1,0.6*\x) to (0.6*1,0.6*\x+0.6);
			\draw[e:colored,color=BostonUniversityRed] (0.6*12,0.6*\x) to (0.6*12,0.6*\x+0.6);
			
		}
		\foreach \x in {2,4,6,8,10}
		{
			\draw[e:coloredborder] (0.6*\x,0.6*1) to (0.6*\x+0.6,0.6*1);
			\draw[e:coloredborder] (0.6*\x,0.6*12) to (0.6*\x+0.6,0.6*12);
			
			\draw[e:colored,color=BostonUniversityRed] (0.6*\x,0.6*1) to (0.6*\x+0.6,0.6*1);
			\draw[e:colored,color=BostonUniversityRed] (0.6*\x,0.6*12) to (0.6*\x+0.6,0.6*12);
		}
		
		\foreach \x in {2,4,6,8,10}
		{
			\draw[e:coloredborder] (0.6*2,0.6*\x) to (0.6*2,0.6*\x+0.6);
			\draw[e:coloredborder] (0.6*11,0.6*\x) to (0.6*11,0.6*\x+0.6);
			
			\draw[e:colored,color=BostonUniversityRed] (0.6*2,0.6*\x) to (0.6*2,0.6*\x+0.6);
			\draw[e:colored,color=BostonUniversityRed] (0.6*11,0.6*\x) to (0.6*11,0.6*\x+0.6);
		}
		\foreach \x in {3,5,7,9}
		{
			\draw[e:coloredborder] (0.6*\x,0.6*2) to (0.6*\x+0.6,0.6*2);
			\draw[e:coloredborder] (0.6*\x,0.6*11) to (0.6*\x+0.6,0.6*11);
			
			\draw[e:colored,color=BostonUniversityRed] (0.6*\x,0.6*2) to (0.6*\x+0.6,0.6*2);
			\draw[e:colored,color=BostonUniversityRed] (0.6*\x,0.6*11) to (0.6*\x+0.6,0.6*11);
		}
		
		\foreach \x in {3,5,7,9}
		{
			\draw[e:coloredborder] (0.6*3,0.6*\x) to (0.6*3,0.6*\x+0.6);
			\draw[e:coloredborder] (0.6*10,0.6*\x) to (0.6*10,0.6*\x+0.6);
			
			\draw[e:colored,color=BostonUniversityRed] (0.6*3,0.6*\x) to (0.6*3,0.6*\x+0.6);
			\draw[e:colored,color=BostonUniversityRed] (0.6*10,0.6*\x) to (0.6*10,0.6*\x+0.6);
		}
		\foreach \x in {4,6,8}
		{
			\draw[e:coloredborder] (0.6*\x,0.6*3) to (0.6*\x+0.6,0.6*3);
			\draw[e:coloredborder] (0.6*\x,0.6*10) to (0.6*\x+0.6,0.6*10);
			
			\draw[e:colored,color=BostonUniversityRed] (0.6*\x,0.6*3) to (0.6*\x+0.6,0.6*3);
			\draw[e:colored,color=BostonUniversityRed] (0.6*\x,0.6*10) to (0.6*\x+0.6,0.6*10);
		}
	
		\foreach \x in {4,6,8}
		{
			\draw[e:coloredborder] (0.6*4,0.6*\x) to (0.6*4,0.6*\x+0.6);
			\draw[e:coloredborder] (0.6*9,0.6*\x) to (0.6*9,0.6*\x+0.6);
			
			\draw[e:colored,color=BostonUniversityRed] (0.6*4,0.6*\x) to (0.6*4,0.6*\x+0.6);
			\draw[e:colored,color=BostonUniversityRed] (0.6*9,0.6*\x) to (0.6*9,0.6*\x+0.6);
		}
		\foreach \x in {5,7}
		{
			\draw[e:coloredborder] (0.6*\x,0.6*4) to (0.6*\x+0.6,0.6*4);
			\draw[e:coloredborder] (0.6*\x,0.6*9) to (0.6*\x+0.6,0.6*9);
			
			\draw[e:colored,color=BostonUniversityRed] (0.6*\x,0.6*4) to (0.6*\x+0.6,0.6*4);
			\draw[e:colored,color=BostonUniversityRed] (0.6*\x,0.6*9) to (0.6*\x+0.6,0.6*9);
		}
	
		\foreach \x in {5,7}
		{
			\draw[e:coloredborder] (0.6*5,0.6*\x) to (0.6*5,0.6*\x+0.6);
			\draw[e:coloredborder] (0.6*8,0.6*\x) to (0.6*8,0.6*\x+0.6);
			
			\draw[e:colored,color=BostonUniversityRed] (0.6*5,0.6*\x) to (0.6*5,0.6*\x+0.6);
			\draw[e:colored,color=BostonUniversityRed] (0.6*8,0.6*\x) to (0.6*8,0.6*\x+0.6);
		}
		\foreach \x in {6}
		{
			\draw[e:coloredborder] (0.6*\x,0.6*5) to (0.6*\x+0.6,0.6*5);
			\draw[e:coloredborder] (0.6*\x,0.6*8) to (0.6*\x+0.6,0.6*8);
			
			\draw[e:colored,color=BostonUniversityRed] (0.6*\x,0.6*5) to (0.6*\x+0.6,0.6*5);
			\draw[e:colored,color=BostonUniversityRed] (0.6*\x,0.6*8) to (0.6*\x+0.6,0.6*8);
		}
		
		\foreach \x in {6}
		{
			\draw[e:coloredborder] (0.6*6,0.6*\x) to (0.6*6,0.6*\x+0.6);
			\draw[e:coloredborder] (0.6*7,0.6*\x) to (0.6*7,0.6*\x+0.6);
			
			\draw[e:colored,color=BostonUniversityRed] (0.6*6,0.6*\x) to (0.6*6,0.6*\x+0.6);
			\draw[e:colored,color=BostonUniversityRed] (0.6*7,0.6*\x) to (0.6*7,0.6*\x+0.6);
		}
		\end{pgfonlayer}
\end{tikzpicture}}{\includegraphics[scale=0.57]{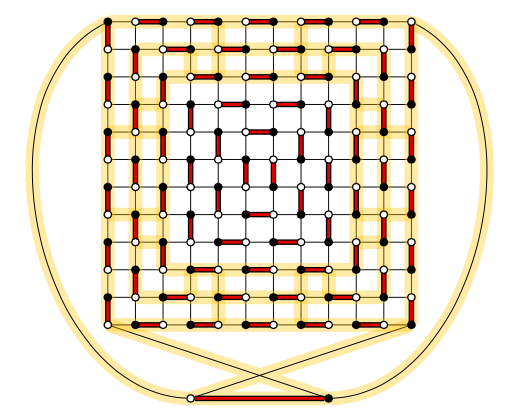}}
	\caption{A model of the single-crossing matching grid of order $k=3$ as a matching minor within the inside-out single-crossing matching grid of order $2k=6.$}
	\label{fig_insideoutcross}
\end{figure}
Now, for each cycle $C_i$ we define a set of paths of even length which will represent the vertices on this cycle.
Consider some cycle $C_i$ and notice that $C_i$ is the outer cycle of a $((4k-2(i-1))\times(4k-2(i-1)))$-subgrid $H_i$ of $G'.$
A \emph{side} of $C_i$ is a shortest path on $C_i$ connecting two consecutive corners of $H_i.$
Each of the four corners of $H_i$ belongs to two sides of $C_i,$ each of these sides has length $4k-2(i-1)-1$
Let $c$ be a corner of $H_i$ and $S$ a side of $H_i$ which contains $c,$ then let $Q^i_{c,S}$ be the subpath of $S$ of length $2(k-i)$ which contains $c.$
Finally, let $Q^i_c$ be the union of the two paths $Q^i_{c,S},$ $Q^i_{c,S'},$ where $S$ and $S'$ are the two sides that contain $c.$
Observe that each $Q^i_{c,S}$ has even length and the two endpoints of $Q^i_c$ belong to the same colour class as $c$ itself.
We call the $Q^i_c$ the \emph{corner paths} of $C_i.$

Given $i\in[k-1]$ we say that a side $S_i$ of $C_i$ and a side $S_{i+1}$ of $C_{i+1}$ are \emph{aligned} if there exists an edge of $G'$ which connects an internal vertex of $S_i$ to an internal vertex of $S_{i+1}.$
Let $S_i$ be a side of $C_i$ and $S_{i+1}$ be the unique side of $C_{i+1}$ which is aligned with $S_i.$
Moreover, let $P$ be the subpath of $S_i$ which shares exactly its endpoints with the two corner paths of $C_i$ that intersect $S_i.$
Now we choose two disjoint subpaths of $P,$ let us call them $Q_1$ and $Q_2,$ of length $i-1,$ each of them containing an endpoint of $P.$
Each of the two $Q_j$ has exactly $i$ vertices, and we may select a perfect matching between $Q_j$ and $S_{i+1}.$
Observe that, for $i\geq 2$ this choice ensures that all but one vertex of each of the $Q_j$ is also covered by an edge chosen in the previous step.

Let $G''$ be the union of the cycles $C_1,\dots,C_k$ together with all matchings for all sides chosen as above.
Notice that all corner paths may be contracted into single vertices in $G''.$
Afterwards some degree two vertices remain.
Four of those are the four corner paths of $C_k$ which all have length zero.
Bicontract vertices of degree two until the only remaining vertices of degree two are exactly those four.
The result is a $(2k\times 2k)$-grid and our construction is complete.
\end{proof}

Using the inverse argument of the construction in the proof of \cref{lemma_reversinginsideout} one can also observe that a large single-crossing matching grid contains a still large inside-out single-crossing matching grid as a matching minor.
Moreover, a large inside-out single-crossing matching grid contains a cylindrical single-crossing matching grid as a matching minor which shows that, as universal obstructions, all three can be seen as being (parametrically) equivalent.

\subsection{A structure theorem for SCMM-free braces}\label{subsec_flatgrid}

We proceed by showing that any brace with a large enough cylindrical matching grid as a matching minor must either contain a single-crossing matching grid as a matching minor or be Pfaffian.
In combination with the \hyperref[thm_matchinggrid]{matching theoretic grid theorem} this yields our structural \hyperref[thm_matchingmainthm]{main result}.

Let $H$ be a cylindrical matching grid of order $3k$ for some positive integer $k.$
For every $i\in[3]$ let $H_i$ be the subgraph of $H$ induced by $\bigcup_{j\in[(i-1)k+1,ik]}V(C_j).$
We call $(H_1,H_2,H_3)$ the \emph{tripartition} of $H.$
We extend the definition of tripartitions in the natural way to subdivisions of the cylindrical matching grid of order $3k.$

\begin{lemma}\label{lemma_matchinglongjump}
Let $k\in\N$ be some positive integer.
Let $B=H+P$ be a bipartite matching covered graph where $H$ is the cylindrical matching grid of order $96k,$ $M$ is a perfect matching of $B$ which contains the canonical matching of $H,$ and $P$ is an internally $M$-conformal path with endpoints $a$ and $b$ such that $V(P)\cap V(H)=\Set{a,b}.$
Let $(H_1,H_2,H_3)$ be the tripartition of $H.$
Suppose $a\in V(H_2).$
If there exists some $j\in[96k]$ such that $C_j$ separates $a$ from $b$ within $H,$ then $B$ contains the single-crossing matching grid of order $k$ as a matching minor.
\end{lemma}

\begin{proof}
The proof is similar to the one of \cref{lemma_jumpandcrossfreegrid}.
We distinguish two cases: either there exists $j\in\{1,3\}$ such that $b\in V(H_j)$ or $b\in V(H_2).$

In the later case it is possible to select two vertices $a_1$ and $a_2$ of $V_1$ and two vertices $b_1$ and $b_2$ of $V_2$ from the cycle $C_{8k}$ which are almost equidistant such that $a_1,$ $a_2,$ $b_1,$ and $b_2$ appear on $C_{32k}$ in the order listed.
We can then use the infrastructure of $H_2+P$ to route two disjoint internally $M$-conformal paths, one joining $a_1$ and $b_1,$ the other joining $a_2$ and $b_2$ while both are internally disjoint from $H_1.$
Observe that the resulting graph is a bisubdivision of the cylindrical single-crossing grid of order $8k.$

In case $b\in V(H_j)$ for $j\in\{1,3\}$ let $i\in\{1,3\}$ be distinct from $j.$
We may now repeat the construction from before for $H_i$ by using the infrastructure of $H_j+H_2$ and again obtain an $M$-conformal bisubdivision of the cylindrical single-crossing grid of order $8k$ as a result.

The claim now follows from \cref{lemma_singlecrossingcylinder}.
\end{proof}

\begin{lemma}\label{lemma_matchinglongjump2}
Let $k\in\N$ be some positive integer.
Let $B=H+P$ be a bipartite matching covered graph where $H$ is the cylindrical matching grid of order $96(k+2),$ $M$ is a perfect matching of $B$ which contains the canonical matching of $H,$ and $P$ is an internally $M$-conformal path with endpoints $a$ and $b$ such that $V(P)\cap V(H)=\Set{a,b}.$
Let $(H_1,H_2,H_3)$ be the tripartition of $H.$
If $P$ has one endpoint in $H_1$ and the other in $H_3,$ then $B$ contains the single-crossing matching grid of order $k$ as a matching minor.
\end{lemma}

\begin{proof}
First observe that $H_2$ contains $32(k+2)$ many concentric cycles of $H.$
Hence, we may find an $M$-conformal bisubdivision $W$ of the cylindrical matching grid of order $8(2k+2)=16k+16$ within $H_2$ such that $W$ is contained between the cycles $C_{35(k+2)}$ and $C_{61(k+2)}.$
Without loss of generality let us assume that $P$ has one endpoint, say $a_1,$ in $V_1\cap V(H_1)$ and let $b_1\in V_2$ be the other endpoint of the edge of $M$ covering $a_1.$
Then the other endpoint, say $b_2$ belongs to $V_2\cap V(H_3).$
Let $a_2$ be the other endpoint of the edge of $M$ which covers $b_2.$
It is now easy to see that we can link any vertex $a$ of $V_1\cap V(C_{35(k+2)})$ which has a neighbour not on this cycle but on a path of $W$ joining $C_{35(k+2)}$ to $C_{35(k+2)-1}$ to $b_1$ with an internally $M$-conformal path which is otherwise disjoint from $W$ and $H_3.$
Similarly, we may join some vertex $b$ of $V_2\cap V(C_{61(k+2)})$ to $a_2$ without otherwise touching $W$ or $H_1.$
Let $Q$ be the resulting internally $M$-conformal path joining $a$ and $b.$
Notice that $W+Q$ is a bisubdivision of the cylindrical single-jump grid of order $8(2k+2).$
By applying \cref{lemma_jumpinggrid} we may not find the inside-out single-crossing matching grid of order $2k$ as a matching minor.
This graph, finally, may be transformed by \cref{lemma_reversinginsideout} into the single-crossing matching grid of order $k$ and thus our proof is complete.
\end{proof}

\begin{lemma}\label{lemma_matchingseparatinggridexternally}
Let $k$ be some integer and $B$ be a brace.
Let $W$ be a conformal bisubdivision of the cylindrical matching grid of order $96(k+2)+6$ together with a perfect matching $M$ of $B$ such that $M$ contains the canonical matching of $W$ and the tripartition $(W_1,W_2,W_3)$ of $W.$
Then either $B$ contains the single-crossing matching grid of order $k$ as a matching minor or there exist separations $(X_1,Y_2)$ and $(X_2,Y_2)$ of $B$ such that $V(W_2)\subseteq X_1\cap X_2,$ $V(C_1)\subseteq Y_1,$ $V(V_{96(k+2)+6})\subseteq Y_2,$ $V(C_{32(k+2)+2})\cup V(C_{32(k+2)+3})= X_1\cap Y_1,$ and $V(C_{64(k+2)+4})\cup V(C_{64(k+2)+5})= X_2\cap Y_2.$
\end{lemma}

\begin{proof}
By symmetry, it suffices to show the existence of the separation $(X_1,Y_1).$
Towards a contradiction, suppose there does not exist such a separation and let $Z\coloneqq V(C_{32(k+2)+2})\cup V(C_{32(k+2)+3}).$
Moreover, $H$ be the maximal matching covered subgraph\footnote{Hence, $H$ is the so called \emph{elementary component} of $B-Z$ which contains the cycle $C_{32(k+2)+4}.$} of $B-Z$ which contains the cycle $C_{32(k+2)+4}.$

Suppose there exists an internally $M$-conformal path $P$ which is internally disjoint from $W,$ disjoint from $Z,$ has one endpoint in $W_1-Z$ and the other endpoint in $W-W_1-Z.$
Then one of the cycles $C_{32(k+2)+2},$ $C_{32(k+2)+3}$ separates the endpoints of $P$ within $W.$
If the endpoint of $P$ in $W-W_1-Z$ belongs to $W_2,$ \cref{lemma_matchinglongjump} implies the existence of the single-crossing matching grid of order $k$ as a matching minor.
Otherwise, it belongs to $W_3$ and thus \cref{lemma_matchinglongjump2} yields the desired matching minor.
Hence, we may assume no such path exists.

Since $Z$ does not separate $C_{32(k+2)+4}$ from $W_1-Z$ by assumption, there must still be some path between the two in $B-Z.$

Recall the \emph{Dulmage-Mendelsohn structure theorem} \cite{dulmage1958coverings}.
The theorem says that for the elementary components\footnote{As above, these are the maximal matching covered and conformal subgraphs of a graph with a perfect matching.} of a bipartite graph with a perfect matching there exists a partial order $\leq_{\text{DM}}$ such that for any two components $H_1,$ $H_2$ with $H_1\leq_{\text{DM}} H_2$ and any perfect matching $M'$ of the graph there exists an internally $M'$-conformal path from a vertex of $V(H_1)\cap V_1$ to a vertex of $V(H_2)\cap V_2.$
Moreover, there is no edge between $V(H_2)\cap V_1$ and $V(H_1)\cap V_2$ and there are no edges between incomparable elementary components.

Let $H_1$ be the elementary component of $B-Z$ that contains $C_1$ and let $H_2$ be the elementary component of $B-Z$ which contains $C_{32(k+2)+4}.$
From the discussion above it follows that $H_1$ and $H_2$ are incomparable w.\@r.\@t.\@ $\leq_{\text{DM}}.$
However, the two must have a common upper or lower bound w.\@r.\@t.\@ $\leq_{\text{DM}}.$
Without loss of generality let us assume that there exists an elementary component $G$ of $B-Z$ which is a common upper bound for both $H_1$ and $H_2.$
Since $B$ is a brace there must exist an internally $M$-conformal $Q$ path starting in $V_1\cap V(G)$ and ending in $V(W)\cap V_2.$
Moreover, we may assume $Q$ to be internally disjoint from $W.$
Since $H_1$ and $H_2$ are distinct elementary components of $B-Z$ and there is no internally $M$-conformal path in $B-Z$ connecting then, it follows that $Q$ must end on one of the two cycles $C_{32(k+2)+2},$ or $C_{32(k+2)+3},$ or somewhere in $W-H_1-H_2.$
Let $b$ be the endpoint of $Q$ in $W.$
Furthermore, let $P_1$ be an internally $M$-conformal path from $V(H_1)\cap V_1$ to $b$ which is internally disjoint from $W\cap(H_1+H_2)$ and let $P_2$ be an internally $M$-conformal path from $V(H_1)\cap V_1$ to $b$ which is internally disjoint from $W.$
For each $i\in[1,2]$ let us denote the other endpoint of $P_i,$ the one that is not $b,$ by $a_i.$
The existence of both paths is guaranteed by $B$ being a brace, \cref{thm_bipartiteextendibility}, the existence of $Q,$ and the definition of $\leq_{\text{DM}}.$
We now distinguish three cases depending on the position of $b.$
In all cases we will find an internally $M$-conformal path with one endpoint in $H_1\cap W$ and the other in $H_2\cap W$ such that the path is internally disjoint from $W$ and its endpoints are separated by one of the cycles $C_{32(k+2)+2}$ or $C_{32(k+2)+3}.$
Whenever we find such a configuration we have found our single-crossing matching grid of order $k$ as a matching minor by applying \cref{lemma_matchinglongjump} or \cref{lemma_matchinglongjump2}.

\textbf{Case 1:} $b\in V(C_{32(k+2)+2}).$
In this case $C_{32(k+2)+3}$ separates $a_2$ and $b$, and thus we are done as $P_2$ is the desired path.

\textbf{Case 2:} $b\in V(C_{32(k+2)+3}).$
Here $C_{2(k+2)k+2}$ separates $a_1$ and $b,$ and $b$ belongs to $H_2\cap W.$
Hence, $P_1$ is the desired path.

\textbf{Case 3:} $b\in V(W-H_1-H_2).$
In this case we have to discuss two subcases as every vertex of $V(W-H_1-H_2)$ belongs to some internally $M$-conformal path of $W$ which joins $C_{32(k+2)+2}$ and $C_{32(k+2)+3}.$
Let $R$ be the path which contains $b.$
In case $R$ has one endpoint in $V_1\cap V(C_{32(k+2)+2})$ we may imagine $R$ as being ``oriented towards'' $C_{32(k+2)+3}.$
Hence, in this case, the cycle separating the vertices $a_1$ and $b$ is $C_{32(k+2)+2}$ and since $C_{32(k+2)+3}$ belongs to $W_2$ we are done.
Otherwise, $R$ has one endpoint in $V_2\cap C_{32(k+2)+2}$ and thus one may imagine it being  ``oriented towards'' $C_{32(k+2)+2}.$
In this case, finally we may choose $C_{32(k+2)+3}$ to be the cycle that separates $a_2$ and $b$, and thus we are done.
\end{proof}

\begin{theorem}\label{thm_pfaffianbrace}
Let $k\in\N$ be some positive integer and $B$ be a brace with a perfect matching $M.$
If $B$ contains an $M$-conformal bisubdivision of the cylindrical matching wall of order $96(k+2)+6,$ then it either contains the single-crossing matching grid of order $k$ as a matching minor, or it is Pfaffian.
\end{theorem}

\begin{proof}
Let $W$ be an $M$-conformal bisubdivision of the cylindrical matching wall of order $96(k+2)+6$ such that $M$ contains its canonical matching.
Let $(W_1,W_2,W_3)$ be the tripartition of $W.$
Moreover, let $U_1$ be an $M$-conformal bisubdivision of the cylindrical matching grid containing the cycles $C_{32(k+2)+1},C_{32(k+2)+2},$ and $C_{32(k+2)+3}.$ 
We say that $C_{32(k+2)+1}$ is the \emph{outer cycle} of $U_1,$ while $C_{32(k+2)+3}$ is its \emph{inner cycle}.
Similarly, let $U_2$ be an $M$-conformal bisubdivision of the cylindrical matching grid containing the cycles $C_{64(k+2)+4},C_{64(k+2)+5},$ and $C_{64(k+2)+6}.$ 
We say that $C_{64(k+2)+6}$ is the \emph{outer cycle} of $U_2,$ while $C_{64(k+2)+4}$ is its \emph{inner cycle}.

Next we call upon \cref{lemma_matchingseparatinggridexternally}.
The lemma either provides us with the single-crossing matching grid of order $k$ as a matching minor, in which case we are done, or we obtain two separations $(X_1,Y_1)$ and $X_2,Y_2$ such that $V(W_2)\subseteq X_1\cap X_2,$ $V(C_1)\subseteq Y_1,$ $V(V_{96(k+2)+6})\subseteq Y_2,$ $V(C_{32(k+2)+2})\cup V(C_{32(k+2)+3})= X_1\cap Y_1,$ and $V(C_{64(k+2)+4})\cup V(C_{64(k+2)+5})= X_2\cap Y_2.$
Notice that $V(B)=X-1\cup X_2.$
For each $i\in[1,2]$ let $H_i\coloneqq U_i\cup \InducedSubgraph{B}{X_i}.$
Moreover, let $\Psi_i$ be the cyclic permutation of the vertex set of the outer cycle of $U_i$ obtained by traversing it in clockwise direction.
Also, let $\Omega_i$ be the cyclic permutation of the vertex set of the inner cycle of $U_i$ obtained by traversing it in clockwise direction.
Notice that for both $i\in[1,2]$ the tuple $(H_i,U_i,\Psi_i,\Omega_i)$ is a reinforced society.

\textbf{Claim 1:} For each $i\in[1,2],$ the matching society $(H_i,\Psi_i)$ is matching flat or $B$ contains the single-crossing matching grid of order $k$ as a matching minor.

By symmetry, it suffices to prove the claim for $i=1.$
Let $W'$ be the union of the component of $W_1-C_{32(k+2)+1}$ which contains the cycle $C_1$ with the cycle $C_{32(k+2)+1}.$
Suppose $(H_1,\Psi_1)$ is not matching flat, then \cref{thm_societytwopaths} provides us with a conformal cross over $C_{32(k+2)+1}$ within $H_1.$
Let $P$ and $Q$ be the two paths of this cross and let $M''$ be a perfect matching of $C_{32(k+2)+1}+P+Q.$
We may now choose a perfect matching $M'$ of $W'+P+Q$ such that $M'\cap E(W'-C_{32(k+2)+1})\subseteq M$ and $M''\subseteq M'.$
Observe that $W'+P+Q$ contains the cylindrical single-crossing grid of order $8k$ as a matching minor and thus, by \cref{lemma_singlecrossingcylinder} we find the single-crossing matching grid of order $k$ as a matching minor in $B.$
Hence, \textbf{Claim 1} follows.\smallskip

So from now on we may assume the matching societies $(H_i,\Psi_i)$ to be matching flat. 
For each $i\in [1,2]$ let $\rho_i$ be the matching rendition of $H_i$ in the disk provided by the matching flatness of $(H_i,\Psi_i).$
Now let $C\coloneqq C_{48(k+2)+3}.$
Notice that $C\subseteq H_1\cap H_2.$
Moreover, as $B$ is a brace it follows from \cref{lemma_tightcutsandconformalcycles} that, for both $i\in[1,2],$ no non-trivial tight cut of $H_i$ contains an edge of $C.$
Let $G_i$ be the subgraph obtained from the union of all components of $H_i-C$ whose vertices are drawn inside the disk $\Delta_i$ bounded by the trace of $C$ in $\rho_i$ together with $C,$ where $\Delta_i$ does not contain the trace of the inner cycle of $U_i.$
Notice that every component of $H_i-C$ is either fully drawn within $\Delta_i$ or has no intersection with the disk since $\rho_i$ is a matching rendition.
Let now $\Phi$ be the cyclic permutation of the vertices of $C$ obtained by traversing $C$ in clockwise direction.
It follows from the matching flatness of both $(H_1,\Psi_1)$ and $(H_2,\Psi_2)$ that the matching societies $(G_1,\Phi)$ and $(G_2,\Phi)$ are matching flat.
By \cref{lemma_joiningoverlappingsocieties} this means that $B=G_1\cup G_2$ has a vortex-free extended $\Sigma$-decomposition $\delta$ where $\Sigma$ is the sphere.
Hence, since $B$ is a brace, $\delta$ is free of big vertices.  
Therefore,  $B$ may be obtained from planar braces by repeated applications of the trisum operation.
Thus, by \cref{thm_trisums}, $B$ is a Pfaffian brace.
\end{proof}

With this we are ready to prove \cref{thm_matchingmainthm}.

\begin{proof}[Proof of \Cref{thm_matchingmainthm}]
Let $S\in\mathcal{S}$ be some single-crossing matching minor and let $k$ be the smallest integer such that $S$ is a matching minor of the single-crossing matching grid of order $k.$
We set
\begin{align*}
	h_1(S)\coloneqq \mathsf{mg}(96(k+2)+6),
\end{align*}
where $\mathsf{mg}$ denotes the function from \cref{thm_matchinggrid}.

Let $B$ be some brace with a perfect matching $M$ such that $B$ excludes $S$ as a matching minor.
Now suppose $\pmw{B}>h_1(S).$
Then, by \cref{thm_matchinggrid}, $B$ contains the cylindrical matching grid $W$ of order $96(k+2)+6$ as an $M$-conformal bisubdivision such that $M$ contains the canonical matching of $W.$

Observe that \cref{thm_pfaffianbrace} has two possible outcomes.
Either $B$ is Pfaffian or it contains the single-crossing matching grid of order $k$ as a matching minor.
In the first case we are done, so we may assume the second case.
By choice of $k$ however this implies that $B$ contains $S$ as a matching minor, thereby contradicting our assumption and thus the proof is complete.
\end{proof}

%%%%%%%%%%%%%%%%%%%%%%%%%%%%%%%%%%%%%%%%%%%%%%%%%%%%%%%%%%%%%%%%%%%%%%%%%%%%%%%%%%%%%%%%%%%%%%%%%%%%%%%%%%%%%%%%%%%%%%%%%%%%%%%%%%%%%%%%%%%%%%%%%%%%%%%%%%%%%%%%%%%%%%%%%%%%%%%%%%%%%%%%%%%%%%%%%%%%%%%%%%%%%%%%%%%%%%%%%%%%%%%%%%%%%%%%%%%%%%%%%%%%%%%%%%%%%%%%%%%%%%%%%%%%%%%%%%%%%%%%%%%%%%%%

\section{Counting perfect matchings}\label{sec_countingperfectmatchings}

In this section we deal with the algorithmic complexity of counting perfect matchings on bipartite graphs.
We dedicate \cref{@desfavorecer} to the positive results that are applicable for graph classes where some single-crossing matching minor is excluded.
The negative results are presented in Subsections \ref{@interconnected} and \ref{@unanswerable}.
Here we show that counting perfect matchings is $\#\mathsf{P}$-hard on bipartite graphs excluding $K_{5,5}$ as a matching minor.
Indeed, we identify a matching minor closed class of graphs that excludes $K_{5,5}$ but contains $K_{4,4}$ on which counting perfect matchings is hard.

\subsection{An algorithm for the permanent}
\label{@desfavorecer}

In this subsection we show how use to the structural result in \cref{thm_matchingmainthm} in order to compute the number perfect matchings on bipartite graphs excluding a single-crossing matching minor.
In fact, this result will follow by  dealing with a more general problem of the computation of the labelled generating function of the matchings of  edge weighted bipartite graphs.

\paragraph{Notational conventions.}
Before we start, we give a series of conventions concerning the running time of the algorithms.
Let $f:\mathbb{N}\to\mathbb{N}$ be a function.
Given two non-negative integers $n$ and $k,$ 
we say that $f=\xp(n,k)$ if there is a function $g:\mathbb{N}\to\mathbb{N}$ such that $f=O(n^{g(k)}).$
Also, we say that $f=\fpt(n,k)$ if there is a function $g:\mathbb{N}\to\mathbb{N}$ such that 
$f=g(k)\cdot n^{O(1)}.$ 

\begin{definition}[Arborescence]
An \emph{arborescence} is an orientation $T$ of a tree such that all vertices of $T$ have in-degree at most one.
The unique vertex of $T$ that has indegree 0 is  \emph{the root} of $T$ while all vertices with out-degree 0 are called \emph{leaves} of $T.$
The vertices of $T$ that have out-degree 0 are called the \emph{leaves} of $T$ and their incident edges are called \emph{leaf edges} of $T.$ 
Given a $t\in V(T)$ we define the descendants of $T$ as $\mathsf{desc}_{T}(t)=\{t'\mid \mbox{there is a directed path in $T$ from $t$ to $t'$}\}.$
Also, if $t\in V(T),$ we denote $T_{t}=T[\mathsf{desc}_{T}(t)].$
Notice that for every tree $T$ and $r\in V(T),$ there is a unique arborescence whose underlying graph is $T$ and whose root is $r.$
We call this arborescence the \emph{arborescence of $T$ rooted at $r$}. 
\end{definition}

We also need the following algorithmic result.

\begin{proposition}[\cite{giannopoulou2021excluding}]
\label{@representarme}
There is a function $f:\mathbb{N}\to \mathbb{N}$ and an algorithm that, given a bipartite graph $B,$ outputs a perfect matching decomposition of $B$ of width at most $f(\pmw{B})$ in time $\fpt(n,\pmw{B}).$
\end{proposition}

\paragraph{The generating function of matchings.}
 Let $\mathbb{Z}[x]$ be the set of all polynomials with integer coefficients.
 Given a graph $G$ and a function $\mathbf{p}\colon E(G)\rightarrow{\mathbb{Z}[x]}$ we call the pair $(G,\mathbf{p})$ an \emph{(edge) labelled graph}, and we refer to $\mathbf{p}$ as its \emph{labeling}.
For simplicity, we make the convention that, when we refer to a subgraph $G'$ of $G,$ instead of writing $(G',\mathbf{p}|_{V(G')})$ we just write $(G',\mathbf{p}).$

We use the term \emph{weighted graph} for a pair $(G,\mathbf{w})$ where $G$ is a graph and $\mathbf{w}: E(G)\to \mathbb{Z}$ is a weighting of its edges.
We derive a labelled graph $(G,\mathbf{p}_{\mathbf{w}})$ from $(G,\mathbf{w})$ by setting, for every $e\in E(G),$ $\mathbf{p}_{\mathbf{w}}(e)\coloneqq x^{\mathbf{w}(e)}.$

Now assume we are given a labelled graph $(G,\mathbf{p})$ and a perfect matching $M$ of $G.$
We express the \emph{total weight of $M$ under} $\mathbf{p}$ as the polynomial $$\MatchingMonomial{\mathbf{p}}{M}\coloneqq \prod_{e\in M}\mathbf{p}(e).$$

\begin{definition}[Generating Functions of matchings]\label{def_generatingfunction} 
Let $(G,\mathbf{p})$ be a labelled graph with a perfect matching.
The \emph{labelled generating function of the matchings of $G$}, usually abbreviated as the \emph{$\mathbf{p}$-generating function} or the \emph{generating function}, is defined to be the polynomial $\GenerateMatchings{G,\mathbf{p}}\coloneqq \sum_{M\in\Perf{G}}\MatchingMonomial{\mathbf{p}}{M}.$
Also, given a vertex $v\in V(G),$ we define the function $\mathsf{VertexGen}_{G,\mathbf{p},v}:\partial_{G}(\{v\})\to \mathbb{Z}[x]$ as the function mapping each edge $e=uv$ incident to $v$ to the polynomial $\GenerateMatchings{G-v-u,\mathbf{p}}\cdot \mathbf{p}(e).$ 
\end{definition}

Let $(G,\mathbf{w})$ be a weighted graph and $v\in V(G).$ Then the \emph{generating function of all weighted perfect matchings in $G$} is the polynomial 
\begin{align*} 
\GenerateMatchings{G,\mathbf{w}}\coloneqq \sum_{M\in\Perf{G}}\MatchingMonomial{\mathbf{p}_{\mathbf{w}}}{M}, 
\end{align*}
which, by definition, is equal to $\sum_{e\in \partial_{G}(\{v\})}\mathsf{VertexGen}_{G,\mathbf{p}_{\mathbf{w}},v}(e),$ for every $v\in V(G).$

The purpose of this section is to prove the following algorithmic result.

\begin{theorem}\label{crsazy_gidrst}
	There exists an algorithm that, given a bipartite graph $H\in\mathfrak{S}$ and a weighted bipartite graph $(B,\mathbf{w})$ that excludes $H$ as a matching minor, outputs $\GenerateMatchings{B,\mathbf{w}}$ in time \xp$(|G|,|H|).$
\end{theorem}

\paragraph{Dynamic programming.}
Our first step is to give a dynamic programming algorithm for the computation of $\mathsf{VertexGen}_{B,\mathbf{p},v},$ when $B$ is a bipartite graph of bounded perfect matching width.
This result can be seen as a generalisation of the dynamic programming algorithm proposed in \cite{giannopoulou2021excluding} for counting the number of perfect matchings of a bipartite graph of bounded perfect matching width.

\begin{lemma}
\label{@objectivement}
There exist a function $f:\mathbb{N}\to\mathbb{N}$ and an algorithm that, given a labelled bipartite graph $(B,\mathbf{p})$ and a vertex $v\in V(B),$ outputs the function $\mathsf{VertexGen}_{B,\mathbf{p},v}$
 in time $\xp(|B|,\mathbf{pmw}(B)).$
\end{lemma}

\begin{proof}
Let $n\coloneqq|B|$ and $k\coloneqq f(\pmw{B}),$ where $f$ is the function from \cref{@representarme}.
By \cref{@representarme}, we may assume that we have a perfect matching decomposition $(T,δ)$ of $B$ of width $k.$
We may assume the arborescence $T$ to be rooted at $r=σ^{-1}(v).$
Let $e=(t,t')\in E(T).$ We say that $e$ is a \emph{leaf edge} if $t'$ is a leaf of $T.$
Let also $$\mathcal{M}_e=\{M\subseteq \partial(e)\mid \mbox{$F$ is a matching and~}|F|\leq k\}.$$
Notice that $|\mathcal{M}_e|=\xp(n,k).$
Let $e=(d,t),$ then by $B_e$ we denote the graph $B[\delta(T_t)].$
Given an $F\in\mathcal{M}_{e},$ we define $B_{e}^{F}\coloneqq B_{e}-\cupall F.$
Also, we define the function $\mathsf{Table}_{e}\colon\mathcal{M}_{e}\to{\mathbb{Z}[x]}$ mapping each $F\in\mathcal{M}_{e}$ to the function $\GenerateMatchings{B_{e}^{F},\mathbf{p}}.$
Observe the following:
\begin{itemize}
\item If $e$ is a leaf edge, then for every $F\in \mathcal{M}_{e},$ $\mathsf{Table}_{e}(F)=\mathbf{1}.$

\item If $e$ is the edge incident to the root $r,$ then for every $\{f\}\in \mathcal{M}_{e}$ it holds that $\mathsf{Table}_{e}(\{f\})=\mathsf{VertexGen}_{B,\mathbf{p},v}(f).$
Therefore, if we can compute $\mathsf{Table}_{e}$ we can also compute $\mathsf{VertexGen}_{B,\mathbf{p},v}.$

\item If $e$ is not a leaf edge then it shares the endpoint $t'$ with two other edges, say $e_1$ and $e_2.$
Notice that this, given some $F\in \mathcal{M}_{e},$ implies
$$\mathsf{Table}_{e}(F)=\prod_{\mbox{$(F_{1},F_{2})\in \mathcal{M}_{e_{1}}\times {\mathcal{M}_{e_{1}}:}\atop{F\subseteq F_{1}\cup F_{2}},~F_{1}\setminus F=F_{2}\setminus F$}}\Big(\mathsf{Table}_{e_1}(F_1)\cdot \mathsf{Table}_{e_2}(F_2)\cdot \mathbf{p}(F_{1}\cap F_{2})\Big)$$
and thus, given $\mathsf{Table}_{e_i},$ for both $i\in[2],$ the values of $\mathsf{Table}_{e}$ can be computed in time $\xp(n,k).$
\end{itemize}
These three observations define a dynamic programming procedure that can compute $\mathsf{VertexGen}_{B,\mathbf{p},v}$ in time $\xp(n,k)=\xp(|B|,\mathbf{pmw}(B)),$ as required.
\end{proof}

Notice that the only reason that, in \cref{@objectivement}, we demand $B$ to be bipartite is that this permits us to compute a perfect matching decomposition of bounded width by using \cref{@representarme}.
If such a decomposition is given as part of the input, the bipartiteness condition may be dropped from \cref{@objectivement}.

\paragraph{Reduction to braces} Our next step is to reduce the problem to braces. For this we need first the definition of the splicing operation.

\begin{definition}[Splicing]
 Let $B_{i}, i\in[2]$ be two matching covered bipartite graphs and let $v_{i}\in V(B_{i}), i\in[i]$ such that $v_{i}\in V_{i}, i\in[2].$
 We say that a bipartite graph $B$ is \emph{a splicing of $B_{1}$ and $B_{2}$ at $v_{1}$ and $v_{2}$} if $B$ can be obtained by considering the disjoint union of $B_{1}-v_{1}$ and $B_{2}-v_{2}$ and then adding edges between $N_{B_{1}}(v_{1})$ and $N_{B_{2}}(v_{2})$ such that each vertex in $N_{B_{1}}(v_{1})$ has some neighbour in $N_{B_{2}}(v_{2})$ and vice-versa.
 We also say that $B$ is \emph{a splicing of $B_{1}$ and $B_{2}$} if there are vertices $v_{i}\in V(B_{i}), i\in[2]$
 such that $B$ is a splicing of $B_{1}$ and $B_{2}$ at $v_{1}$ and $v_{2}.$
 \end{definition}
 
Notice that if $B$ is a splicing of $B_{1}$ and $B_{2}$ at $v_{1}$ and $v_{2}$ then $F\coloneqq\partial_{B}(V(B_{1}-v_{1}))=\partial_{B}(V(B_{2}-v_{2}))$ is a non-trivial tight cut of $B.$
Moreover, $B_{1}$ and $B_{2}$ are the two tight cut contractions of $F.$
Consequently, both $B_{1}$ and $B_{2}$ have strictly fewer vertices than $B.$ 
 
\begin{lemma}\label{thm_splicegeneratingfunctions}
Let $\mathcal{B}$ be a class of labelled matching covered bipartite graphs. 
If there is a polynomial-time algorithm computing \textsf{VertexGen} on the braces of $\mathcal{B},$ then there is also a polynomial algorithm for computing \textsf{VertexGen} on all graphs in $\mathcal{B}.$
 \end{lemma}
 
\begin{proof}
Suppose that $(B,\mathbf{p})$ is a labelled matching covered bipartite graph and let $v\in V(B).$ 
It is known, see e.g. \cite{lovasz1987matching,lovasz2009matching}, that if $B$ is not a brace then it can be obtained as the result of a splicing of two matching covered bipartite graphs $B_{1},$ $B_{2},$ where $B_{1}$ is a brace and $v\in V(B_{2}).$
Moreover, by {\cite{lovasz1987matching,lovasz2009matching}}, there is a procedure that can find $B_{1}$ and $B_{2}$ in time polynomial in $|B|.$ 
We define an edge-labelling $\mathbf{p}': E(B_{2})\to\mathbb{Z}[x]$ by, given an edge $e\in E(G_{2}),$ setting 
$$
\mathbf{p}'(e)=\begin{cases}
~~~~~~\mathbf{p}(e) & \mbox{if $e\in E(B_{2}-v_{2})$}\\
{\displaystyle\sum_{y\in N_{B}(x)\cap V(B_{1})}\textsf{VertexGen}_{B_{1},\mathbf{p}_{1}^{},v_{1}}(\{y,v_{1}\})} & \mbox{if $e=v_{2}x\in \partial_{B_{2}}(\{v_{2}\})$}
\end{cases}
$$ 
Notice that the label of each edge $v_{2}x$ sums the weights of all matchings in $B_{1}$ for the edges of $B$ that are incident with $x.$
This implies that $\textsf{VertexGen}_{B,\mathbf{p},v}=\textsf{VertexGen}_{B_{2},\mathbf{p}',v}.$

Assuming that we have a procedure for computing \textsf{VertexGen} on the braces of $\mathcal{B}$ and, given that $B_{1}$ is a brace, we can produce in polynomial time the edge-labeling $\mathbf{p}'$ of $B_{2}$ above.
As $|B_{2}|<|B|,$ this completes the proof, as we may apply the reduction of $(B,\mathbf{p})$ to $(B_{2},\mathbf{p}'),$ as described above, repetitively as long as the resulting $B_{2}$ is not a brace.
 \end{proof}
 
\paragraph{Proofs of \cref{crsazy_gidrst} and \cref{thm_algomainthm}.} 
We need the following result that is an algorithmic consequence of \cref{cor_pfaffianalg}.

\begin{proposition}[\cite{mccuaig2004polya,robertson1999permanents}]
\label{little_pr}
There is a polynomial-time algorithm that, given a labelled Pfaffian brace $(B,\mathbf{p}),$ outputs $\GenerateMatchings{B,\mathbf{p}}.$
\end{proposition} 

\begin{proof}[Proof of \cref{crsazy_gidrst}]
We may readily assume that $B$ is a matching covered bipartite graph, if not, we compute $\GenerateMatchings{B,\mathbf{w}}$ on the elementary components of $B$ and multiply the resulting generating functions (see \cite{lovasz2009matching}).
For every $k\in\mathbb{N},$ we define $\mathcal{B}_{k}$ as the class of bipartite graphs that are either Pfaffian or have perfect matching width at most $k.$
From \cref{thm_matchingmainthm}, we know that $B\in\mathcal{B}_{f(|S|)}$ for some function $f\colon\mathbb{N}\to\mathbb{N}.$ 
Moreover, from \cref{thm_splicegeneratingfunctions} we may also assume that $B$ is a brace.
If $B$ is Pfaffian then we apply \cref{little_pr} and if $\pmw{B}\leq f(|S|),$ then we apply \cref{@objectivement}.
\end{proof} 

\begin{proof}[Proof of \cref{thm_algomainthm}] 
Recall that, if ${\mathbf{1}}$ is the weighing function assigning unit weights on the edges of $B,$ then the permanent of the biadjacency matrid of $B$ is equal to $\GenerateMatchings{B,{\mathbf{1}}}=c\cdot x^{|B|/2}$ where $c=|\Perf{G}|$ is the number of perfect matchings of $G.$
Therefore, the result is a direct corollary of \cref{crsazy_gidrst}, as we may use it in order to compute $\GenerateMatchings{B,{\mathbf{1}}}$ and then return the coefficient of $x^{|B|/2}.$
\end{proof}

We stress that \cref{crsazy_gidrst} provides much more information than the one used in the proof of \cref{thm_algomainthm}.
In fact, the computation of $\GenerateMatchings{G,\mathbf{w}}$ for some weighted graph $(G,\mathbf{w})$ is known as the \emph{dimer problem}.
Moreover, the computation of $\GenerateMatchings{G,{\mathbf{w}}}$ also permits to solve the \textsc{Exact Perfect Matching} problem: given an edge-weighted graph and some non-negative integer $k,$ decide whether there is a perfect matching of total weight exactly $k.$ 
\textsc{Exact Perfect Matching}  was defined by Papadimitriou and Yannakakis in \cite{PapadimitriouY82theco}, has been extensively studied \cite{GurjarKMT17exact,ZhuLM08exact}, with several applications in \cite{BlazewiczFKSW07apoly,liu2021efficient}.
\Cref{crsazy_gidrst} implies that these problems are polynomially solvable on each matching minor closed class of bipartite graphs with perfect matchings that excludes some single-crossing matching minor as a matching minor.

%%%%%%%%%%%%%%%%%%%%%%%%%%%%%%%%%%%%%%%%%%%%%%%%%%%%%%%%%%%%%%%%%%%%%%%%%%%%%%%%%%%%%%%%%%%%%%%%%%%%%%%%%%%%%%%%%%%%%%%%%%%%%%%%%%%%%%%%%%%%%%%%%%%%%%%%%%%%%%%%%%%%%%%%%%%%%%%%%%%%%%%%%%%%%%%%%%%%%%%%%%%%%%%%%%%%%%%%%%%%%%%%%%%%%%%%%%%%%%%%%%%%%%%%%%%%%%%%%%%%%%%%%%%%%%%%%%%%%%%%%%%%%%%%

\subsection{Hardness on shallow vortex matching minors}

\label{@interconnected}

In this subsection our aim is to show that the problem of counting the number of perfect matchings of bipartite graphs is $\classP$-hard
in a graph class called \textit{shallow vortex matching minors.}
Our proof is heavily inspired by and follows along the lines of and analogous result of Curticapean and Xia~\cite{CurticapeanXia22}.

Let us start with the necessary definitions.

Given a weighted graph $(G,w)$ we define 
$$\pmprob(G,w)\coloneqq \sum_{M\in \Perf{G}} \prod_{\substack{e\in M}}w(e).$$
Note that whenever the given graph $G$ is unweighted, the function $\pmprob$ with $w(e)=1$ for every $e\in E(G)$ calculates the 
number of perfect matchings of $G,$ that is, $$\frac{2}{|V(G)|}\cdot \pmprob(G,w)=|\mathcal{M}(G)|.$$
We define as $\pmprob$ the problem where, given as input a weighted graph $(G,w),$ outputs the quantity $\pmprob(G,w).$
In the following we restrict ourselves to this case and write $\pmprob(G)$ instead of $\pmprob(G,w),$ where $w$ is the all one weight function above.

\begin{definition}[Shallow vortex matching grid]\label{def_shallowvortexgrid}
	The \emph{shallow vortex matching grid} $\svmg_{k}$ of order $k$ is defined as follows.
	Let $C_{1},\dots,C_{2k}$ be $2k$ vertex disjoint cycles of length $8k.$
	For every $i\in[2k]$ let $C_{i}=\Brace{v_{1}^{i},v_{2}^{i},\dots,v_{8k}^{i}},$ $V_{1}^{i}\coloneqq\CondSet{v_{j}^{i}}{j\in\Set{1,3,5,\dots,8k-1}},$ $V_{2}^{i}\coloneqq\Fkt{V}{C_{i}}\setminus V_{1}^{i},$ and $M_{i}\coloneqq\CondSet{v_{j}^{i}v_{j+1}^{i}}{v_{j}^{i}\in V_{1}^{i}}.$
	Then $\svmg_{k}$ is the graph obtained from the union of the $C_{i}$ by adding
	\begin{align*}
		\CondSet{v_{j}^{i}v_{j+1}^{i+1}}{i\in[k-1]~\text{and}~j\in\Set{1,5,9,\dots,8k-3}}&\text{, and}\\
		\CondSet{v_{j}^{i}v_{j+1}^{i-1}}{i\in[2,k]~\text{and}~j\in\Set{3,7,11,\dots,8k-1}}&\text{, and}\\
		\CondSet{v_{j}^{1}v_{(j+5)\text{mod}8k}^{1}}{j\in\Set{2,6,\dots,8k-2}}&
	\end{align*}
	to the edge set.
	We call $M\coloneqq\bigcup_{i=1}^{k}M_{i}$ the \emph{canonical matching} of $\svmg_{k}.$
	We also call the edges of $\CondSet{v_{j}^{1}v_{(j+5)\text{mod}8k}^{1}}{j\in\Set{2,6,\dots,8k-2}}$ the \emph{crossing edges} of $\svmg_{k}.$
	See \cref{fig_shallowvortexgrid} for an illustration.
\end{definition}

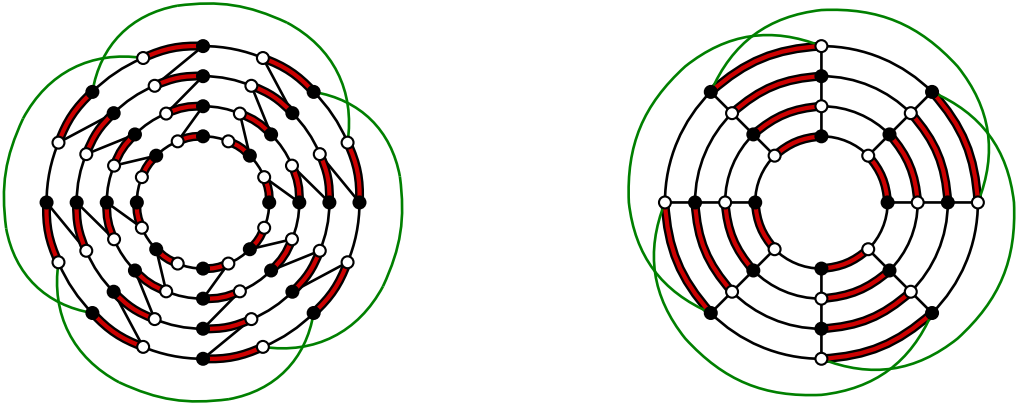
\begin{figure}[!ht]
	\centering
\makefast{% !TEX root = ../single_comb.tex
% !TeX spellcheck = en_UK	
\begin{subfigure}{0.49\textwidth}
\centering
	\begin{tikzpicture}[scale=0.8]

		\pgfdeclarelayer{background}
		\pgfdeclarelayer{foreground}
		\pgfsetlayers{background,main,foreground}
		%
		
		%\draw[e:main,color=green] ()
		
		\draw[e:main] (0,0) circle (11mm);
		\draw[e:main] (0,0) circle (16mm);
		\draw[e:main] (0,0) circle (21mm);
		\draw[e:main] (0,0) circle (26mm);
		
		\foreach \x in {1,...,4}
		{
			\draw[e:main] (\x*90:16mm) -- (\x*90+22.5:11mm);
			\draw[e:main] (\x*90:21mm) -- (\x*90+22.5:16mm);
			\draw[e:main] (\x*90:26mm) -- (\x*90+22.5:21mm);
		}
		
		\foreach \x in {1,...,4}
		{
			\draw[e:main] (\x*90-22.5:16mm) -- (\x*90-45:11mm);
			\draw[e:main] (\x*90-22.5:21mm) -- (\x*90-45:16mm);
			\draw[e:main] (\x*90-22.5:26mm) -- (\x*90-45:21mm);
			
		}
		
		\foreach \x in {1,...,8}
		{
			\draw[e:coloredthin,color=BostonUniversityRed,bend right=13] (\x*45:11mm) to (\x*45+22.5:11mm);
			\draw[e:coloredthin,color=BostonUniversityRed,bend right=13] (\x*45:16mm) to (\x*45+22.5:16mm);
			\draw[e:coloredthin,color=BostonUniversityRed,bend right=13] (\x*45:21mm) to (\x*45+22.5:21mm);
			\draw[e:coloredthin,color=BostonUniversityRed,bend right=13] (\x*45:26mm) to (\x*45+22.5:26mm);
		}
		
		\foreach \x in {1,3,5,7}
		{
			\draw[e:main,color=myGreen,bend left=35] (\x*45:26mm) to (\x*45-37.5:33mm);
			\draw[e:main,color=myGreen,bend left=16] (\x*45-37.5:33mm) to (\x*45-70:33mm);
			\draw[e:main,color=myGreen,bend left=35] (\x*45-70:33mm) to (\x*45-112.5:26mm);
		}

		\foreach \x in {1,...,8}
		{
			\node[v:main] () at (\x*45:11mm){};
			\node[v:main] () at (\x*45:16mm){};
			\node[v:main] () at (\x*45:21mm){};
			\node[v:main] () at (\x*45:26mm){};
			\node[v:mainempty] () at (\x*45+22.5:11mm){};
			\node[v:mainempty] () at (\x*45+22.5:16mm){};
			\node[v:mainempty] () at (\x*45+22.5:21mm){};
			\node[v:mainempty] () at (\x*45+22.5:26mm){};
		}

		\begin{pgfonlayer}{background}
			\foreach \x in {1,...,8}
			{
				\draw[e:coloredborder,bend right=13] (\x*45:11mm) to (\x*45+22.5:11mm);
				\draw[e:coloredborder,bend right=13] (\x*45:16mm) to (\x*45+22.5:16mm);
				\draw[e:coloredborder,bend right=13] (\x*45:21mm) to (\x*45+22.5:21mm);
				\draw[e:coloredborder,bend right=13] (\x*45:26mm) to (\x*45+22.5:26mm);
			}
		\end{pgfonlayer}
	\end{tikzpicture}
\end{subfigure}
\begin{subfigure}{0.49\textwidth}
	\centering
		\begin{tikzpicture}[scale=0.8]

			\pgfdeclarelayer{background}
			\pgfdeclarelayer{foreground}
			\pgfsetlayers{background,main,foreground}

			\draw[e:main] (0,0) circle (11mm);
			\draw[e:main] (0,0) circle (16mm);
			\draw[e:main] (0,0) circle (21mm);
			\draw[e:main] (0,0) circle (26mm);

			\foreach \x in {1,...,8}
			{
				\draw[e:main] (\x*45:11mm) -- (\x*45:16mm);
				\draw[e:main] (\x*45:16mm) -- (\x*45:21mm);
				\draw[e:main] (\x*45:21mm) -- (\x*45:26mm);
			}

			\foreach \x in {2,4,6,8}
			{
			\draw[e:main,color=myGreen] (\x*45:26mm) to [bend right=30] (\x*45+45:32mm) to [bend right=25] (\x*45+90:32mm) to [bend right=30] (\x*45+135:26mm);
			}

			\foreach \x in {2,4,6,8}
			{
				\draw[e:coloredthin,color=BostonUniversityRed,bend right=18] (\x*45:11mm) to (\x*45+45:11mm);
				\draw[e:coloredthin,color=BostonUniversityRed,bend right=18] (\x*45:16mm) to (\x*45+45:16mm);
				\draw[e:coloredthin,color=BostonUniversityRed,bend right=19] (\x*45:21mm) to (\x*45+45:21mm);
				\draw[e:coloredthin,color=BostonUniversityRed,bend right=19] (\x*45:26mm) to (\x*45+45:26mm);
			}

			\foreach \x in {1,3,...,8}
			{
				\node[v:mainempty] () at (\x*45:11mm){};
				\node[v:main] () at (\x*45:16mm){};
				\node[v:mainempty] () at (\x*45:21mm){};
				\node[v:main] () at (\x*45:26mm){};
				\node[v:main] () at (\x*45+45:11mm){};
				\node[v:mainempty] () at (\x*45+45:16mm){};
				\node[v:main] () at (\x*45+45:21mm){};
				\node[v:mainempty] () at (\x*45+45:26mm){};
			}

			\begin{pgfonlayer}{background}
			\foreach \x in {2,4,6,8}
			{
				\draw[e:coloredborder,bend right=18] (\x*45:11mm) to (\x*45+45:11mm);
				\draw[e:coloredborder,bend right=18] (\x*45:16mm) to (\x*45+45:16mm);
				\draw[e:coloredborder,bend right=19] (\x*45:21mm) to (\x*45+45:21mm);
				\draw[e:coloredborder,bend right=19] (\x*45:26mm) to (\x*45+45:26mm);
			}
			\end{pgfonlayer}
			
		\end{tikzpicture}
\end{subfigure}}{\scalebox{.55}{\includegraphics{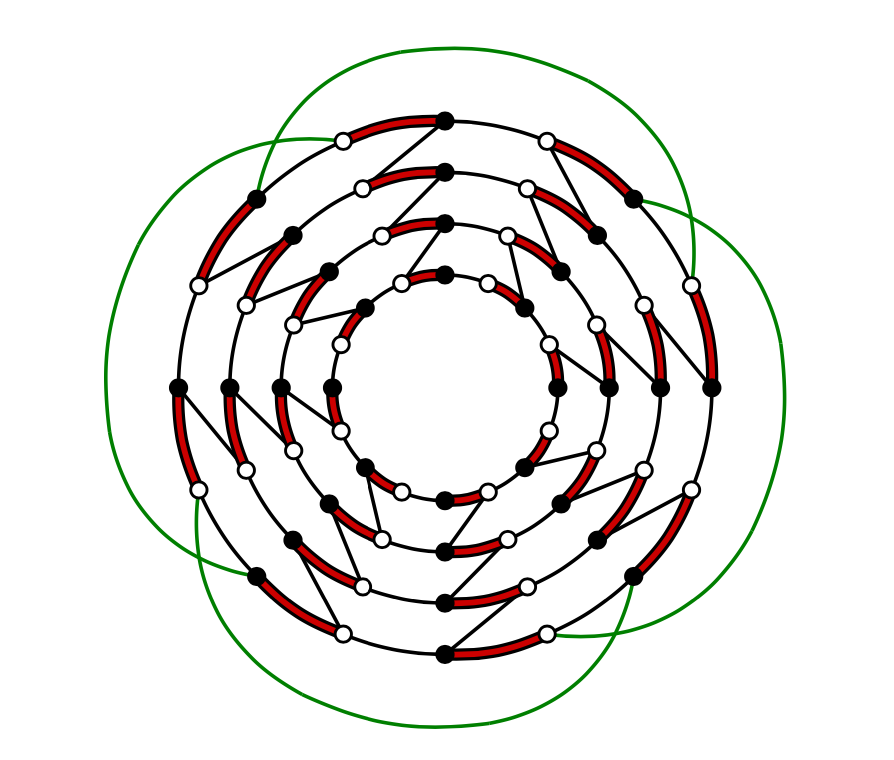}}}
	\caption{The \hyperref[def_shallowvortexgrid]{shallow vortex grid} of order $2$ with its canonical matching.}
	\label{fig_shallowvortexgrid}
\end{figure}

From the definitions of cylindrical matching grids and shallow vortex matching grids it immediately follows that.

\begin{observation}
For every $k\in \mathbb{N},$ $CG_{2k}$ is a conformal subgraph of $\svmg_{k}.$
\end{observation}

\begin{definition}[Shallow vortex matching minor]
We define $\mathcal{V}$ to be the class consisting of all bipartite matching covered graphs $Β$ which are matching minor of a shallow vortex matching grid of order $t$ for some $t.$ We call the graphs in $\mathcal{V}$ \emph{shallow vortex matching minors}. Note that $\mathcal{S}\subseteq \mathcal{V}.$
\end{definition}

To simplify our proofs we also need to define refined vortices as follows.

%\begin{figure}[!ht]
%\centering
%\makefast{ \input{figures/refinedsimplevortex.tex}}{\scalebox{.43}{\includegraphics{figures/refinedsimplevortex}}}
%	\caption{The \hyperref[def_refinedvortex]{refined vortex} of order $2.$}
%	\label{fig_refinedsimplevortex}
%\end{figure}

\begin{definition}[Refined vortex]\label{def_refinedvortex}
	The \emph{refined vortex} $RV_{k}$ of order $k$ is defined as follows.
	Let $C_{1},\dots,C_{2k}$ be $2k$ vertex disjoint cycles of length $4k.$
%	For every $i\in[2k],$ let $C_{i}=\Brace{v_{1}^{i},v_{2}^{i},\dots,v_{2k}^{i}},$ $V_{1}^{i}\coloneqq\CondSet{v_{j}^{i}}{j\in\Set{1,3,5,\dots,2k-1}},$ $V_{2}^{i}\coloneqq\Fkt{V}{C_{i}}\setminus V_{1}^{i},$ and $M_{i}\coloneqq\CondSet{v_{j}^{i}v_{j+1}^{i}}{v_{j}^{i}\in V_{1}^{i}}.$
	$RV_{k}$ is the graph obtained from the union of the $C_{i}$ by adding
	$$\CondSet{v_{j}^{i}v_{j}^{i+1}}{i\in[2k-1]}\text{ and }	\CondSet{v_{j}^{1}v_{(j+3)\text{mod}2k}^{1}}{j\in\Set{2,6,\dots,4k}}$$
	to the edge set. We define the canonical matching of the refined vortices analogously to the canonical matching of the shallow vortex matching grids.
	See \cref{fig_shallowvortexgrid} for an illustration.
\end{definition}

%\begin{definition}[Simple vortex]\label{def_simplevortex}
%	The \emph{simple vortex} $SV_{k}$ of order $k$ is obtained from the refined vortex $RV_{k}$ as follows.
%%	For every $i\in[2k],$ let $C_{i}=\Brace{v_{1}^{i},v_{2}^{i},\dots,v_{2k}^{i}},$ $V_{1}^{i}\coloneqq\CondSet{v_{j}^{i}}{j\in\Set{1,3,5,\dots,2k-1}},$ $V_{2}^{i}\coloneqq\Fkt{V}{C_{i}}\setminus V_{1}^{i},$ and $M_{i}\coloneqq\CondSet{v_{j}^{i}v_{j+1}^{i}}{v_{j}^{i}\in V_{1}^{i}}.$
%	First, we remove all cycles $C_{4},\dots,C_{2k}$ and then we bicontract the internal cycle $C_{3}$ to a cycle of length 4. See, for example~\cref{fig_refinedsimplevortex})
%	\end{definition}

%
%
%
%
%
%
\begin{lemma}\label{lem_rvintosvg}
The refined vortex of order $k$ is a matching minor of the shallow vortex matching grid of order $k$ for every $k\in\mathbb{N}.$
\end{lemma}

\begin{proof}
We skip the tedious proof and instead illustrate a matching minor model of $RV_{6}$ in $\svmg_{6}$ in Figure~\ref{fig_gridvortexinsvg}. It is straightforward to see that the same construction extends to any positive integer $k.$
\end{proof}

\begin{figure}[!ht]
	\centering
	\makefast{ % !TEX root = ../single_comb.tex
% !TeX spellcheck = en_UK	

\begin{tikzpicture}[scale=0.75]

		\pgfdeclarelayer{background}
		\pgfdeclarelayer{foreground}
		\pgfsetlayers{background,main,foreground}

		\draw[e:main] (0,0) circle (31mm);
		\draw[e:main] (0,0) circle (36mm);
		\draw[e:main] (0,0) circle (41mm);
		\draw[e:main] (0,0) circle (46mm);
		\draw[e:main] (0,0) circle (51mm);
		\draw[e:main] (0,0) circle (56mm);
		\draw[e:main] (0,0) circle (61mm);
		\draw[e:main] (0,0) circle (66mm);

		\foreach \x in {1,...,12}
		{
			\draw[e:main] (\x*30:36mm) -- (\x*30-7.5:31mm);
			\draw[e:main] (\x*30:41mm) -- (\x*30-7.5:36mm);
			\draw[e:main] (\x*30:46mm) -- (\x*30-7.5:41mm);
			\draw[e:main] (\x*30:51mm) -- (\x*30-7.5:46mm);
			\draw[e:main] (\x*30:56mm) -- (\x*30-7.5:51mm);
			\draw[e:main] (\x*30:61mm) -- (\x*30-7.5:56mm);
			\draw[e:main] (\x*30:66mm) -- (\x*30-7.5:61mm);
		}
		
		\foreach \x in {1,...,12}
		{
			\draw[e:main] (\x*30+7.5:36mm) -- (\x*30+15:31mm);
			\draw[e:main] (\x*30+7.5:41mm) -- (\x*30+15:36mm);
			\draw[e:main] (\x*30+7.5:46mm) -- (\x*30+15:41mm);
			\draw[e:main] (\x*30+7.5:51mm) -- (\x*30+15:46mm);
			\draw[e:main] (\x*30+7.5:56mm) -- (\x*30+15:51mm);
			\draw[e:main] (\x*30+7.5:61mm) -- (\x*30+15:56mm);
			\draw[e:main] (\x*30+7.5:66mm) -- (\x*30+15:61mm);
		}
		
		\foreach \x in {1,3,...,24}
		{
			\draw[e:main,color=myGreen,bend right=60] (\x*15:66mm) to (\x*15+37.5:66mm);
		}

		\foreach \x in {2,6,...,24}
		{
			\draw[e:marker,color=green,bend right=13] (\x*15:56mm) to (\x*15+30:56mm);
			\draw[e:marker,color=green,bend right=13] (\x*15:46mm) to (\x*15+30:46mm);
			\draw[e:marker,color=green,bend right=13] (\x*15:36mm) to (\x*15+30:36mm);

		}
		
		\foreach \x in {3,11,...,24}
		{
			\draw[e:marker,color=red] (\x*15:56mm) to (\x*15-7.5:61mm) to [bend right=13] (\x*15+7.5:61mm) to (\x*15+15:66mm) to (\x*15+7.5:66mm) to [bend left=60] (\x*15-30:66mm) to (\x*15-22.5:66mm) to [bend left=60] (\x*15-60:66mm) to [bend left=10] (\x*15-75:66mm) to (\x*15-82.55:61mm) to (\x*15-75:61mm) to (\x*15-82.5:56mm); 
		}

		\foreach \x in {7,15,...,24}
		{
			\draw[e:marker,Gray] (\x*15:56mm) to (\x*15-7.5:61mm) to [bend right=13] (\x*15+7.5:61mm) to (\x*15+15:66mm) to (\x*15+7.5:66mm) to [bend left=60] (\x*15-30:66mm) to (\x*15-22.5:66mm) to [bend left=60] (\x*15-60:66mm) to [bend left=10] (\x*15-75:66mm) to (\x*15-82.55:61mm) to (\x*15-75:61mm) to (\x*15-82.5:56mm); 
		}

		\foreach \x in {4,8,...,24}
		{
			\draw[e:marker,color=green,bend right=13] (\x*15:51mm) to (\x*15+30:51mm);
			\draw[e:marker,color=green,bend right=13] (\x*15:41mm) to (\x*15+30:41mm);
			\draw[e:marker,color=green,bend right=13] (\x*15:31mm) to (\x*15+30:31mm);

		}
		
		\foreach \x in {4,8,...,24}
		{
			\draw[e:marker,magenta,bend right=10] (\x*15+7.5:36mm) to (\x*15+22.5:36mm);
			\draw[e:marker,magenta,bend right=10] (\x*15+7.5:46mm) to (\x*15+22.5:46mm);
			\draw[e:marker,magenta,bend right=10] (\x*15+7.5:56mm) to (\x*15+22.5:56mm);
		}
		
		\foreach \x in {2,6,...,24}
		{
			\draw[e:marker,magenta,bend right=10] (\x*15+7.5:31mm) to (\x*15+22.5:31mm);
			\draw[e:marker,magenta,bend right=10] (\x*15+7.5:41mm) to (\x*15+22.5:41mm);
			\draw[e:marker,magenta,bend right=10] (\x*15+7.5:51mm) to (\x*15+22.5:51mm);
		}		
		
		\foreach \x in {1,...,24}
		{
			\draw[e:coloredthin,color=BostonUniversityRed,bend left=6] (\x*15:31mm) to (\x*15-7.5:31mm);
			\draw[e:coloredthin,color=BostonUniversityRed,bend left=6] (\x*15:36mm) to (\x*15-7.5:36mm);
			\draw[e:coloredthin,color=BostonUniversityRed,bend left=5] (\x*15:41mm) to (\x*15-7.5:41mm);
			\draw[e:coloredthin,color=BostonUniversityRed,bend left=5] (\x*15:46mm) to (\x*15-7.5:46mm);
			\draw[e:coloredthin,color=BostonUniversityRed,bend left=5] (\x*15:51mm) to (\x*15-7.5:51mm);
			\draw[e:coloredthin,color=BostonUniversityRed,bend left=5] (\x*15:56mm) to (\x*15-7.5:56mm);
			\draw[e:coloredthin,color=BostonUniversityRed,bend left=5] (\x*15:61mm) to (\x*15-7.5:61mm);
			\draw[e:coloredthin,color=BostonUniversityRed,bend left=5] (\x*15:66mm) to (\x*15-7.5:66mm);
		}
		
		\foreach \x in {1,...,24}
		{
			\node[v:mainempty] () at (\x*15:31mm){};
			\node[v:mainempty] () at (\x*15:36mm){};
			\node[v:mainempty] () at (\x*15:41mm){};
			\node[v:mainempty] () at (\x*15:46mm){};
			\node[v:mainempty] () at (\x*15:51mm){};
			\node[v:mainempty] () at (\x*15:56mm){};
			\node[v:mainempty] () at (\x*15:61mm){};
			\node[v:mainempty] () at (\x*15:66mm){};

			\node[v:main] () at (\x*15+7.5:31mm){};
			\node[v:main] () at (\x*15+7.5:36mm){};
			\node[v:main] () at (\x*15+7.5:41mm){};
			\node[v:main] () at (\x*15+7.5:46mm){};
			\node[v:main] () at (\x*15+7.5:51mm){};
			\node[v:main] () at (\x*15+7.5:56mm){};
			\node[v:main] () at (\x*15+7.5:61mm){};
			\node[v:main] () at (\x*15+7.5:66mm){};
		}

		\begin{pgfonlayer}{background}
			\foreach \x in {1,...,8}
			{
				\foreach \x in {1,...,24}
				{
					\draw[e:coloredborder,bend left=6] (\x*15:31mm) to (\x*15-7.5:31mm);
					\draw[e:coloredborder,bend left=6] (\x*15:36mm) to (\x*15-7.5:36mm);
					\draw[e:coloredborder,bend left=5] (\x*15:41mm) to (\x*15-7.5:41mm);
					\draw[e:coloredborder,bend left=5] (\x*15:46mm) to (\x*15-7.5:46mm);
					\draw[e:coloredborder,bend left=5] (\x*15:51mm) to (\x*15-7.5:51mm);
					\draw[e:coloredborder,bend left=5] (\x*15:56mm) to (\x*15-7.5:56mm);
					\draw[e:coloredborder,bend left=5] (\x*15:61mm) to (\x*15-7.5:61mm);
					\draw[e:coloredborder,bend left=5] (\x*15:66mm) to (\x*15-7.5:66mm);
				}
			}
		\end{pgfonlayer}
	\end{tikzpicture}}{\scalebox{.5}{\includegraphics{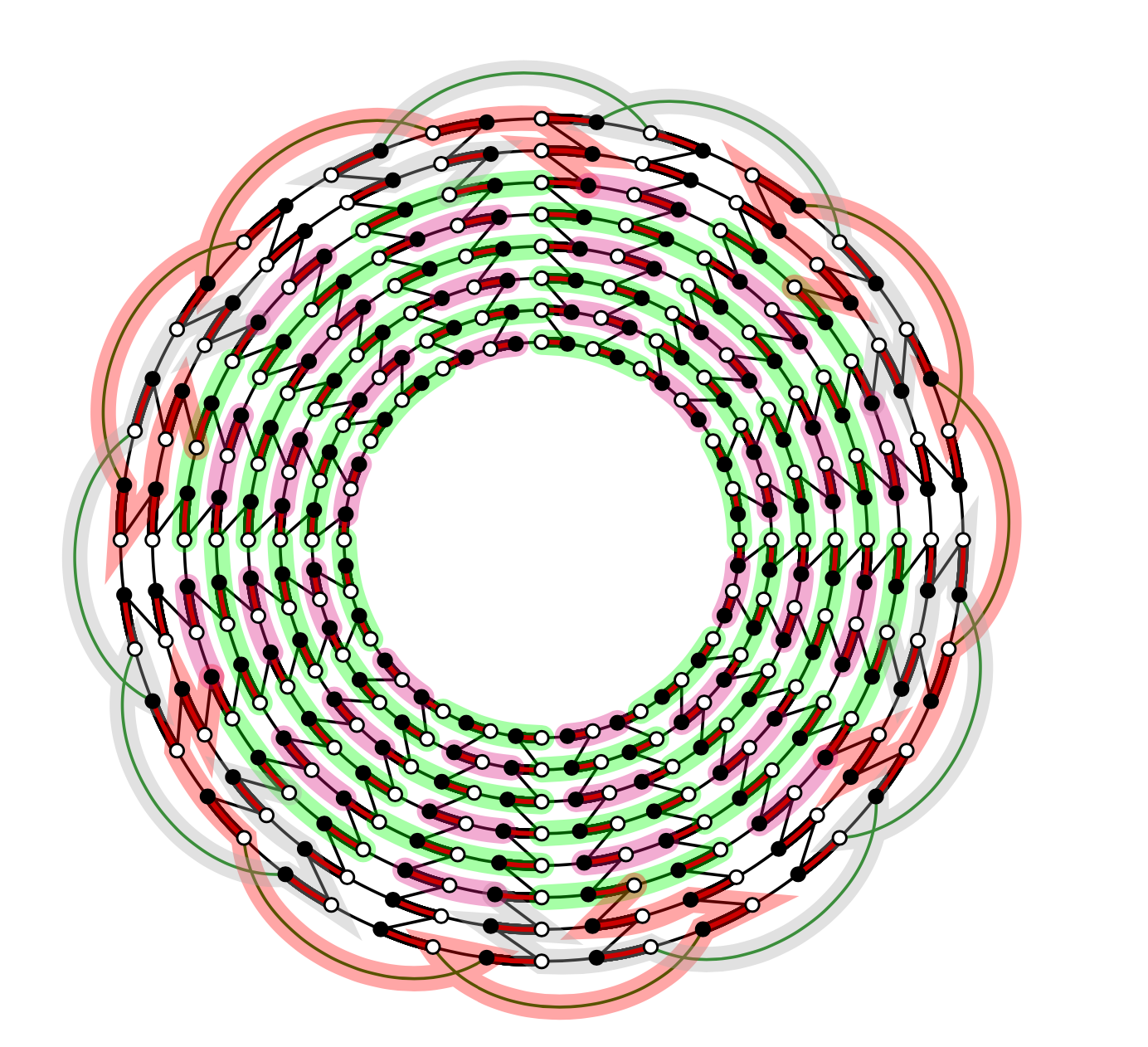}}}
	\caption{The refined vortex of order $6$ as a matching minor of $\svmg_{6}.$ Pink and green denote the bicontractions that result in the vertices of the refined vortex while orange and grey denote the bicontractions that result in its crossing edges.}
	\label{fig_gridvortexinsvg}
\end{figure}
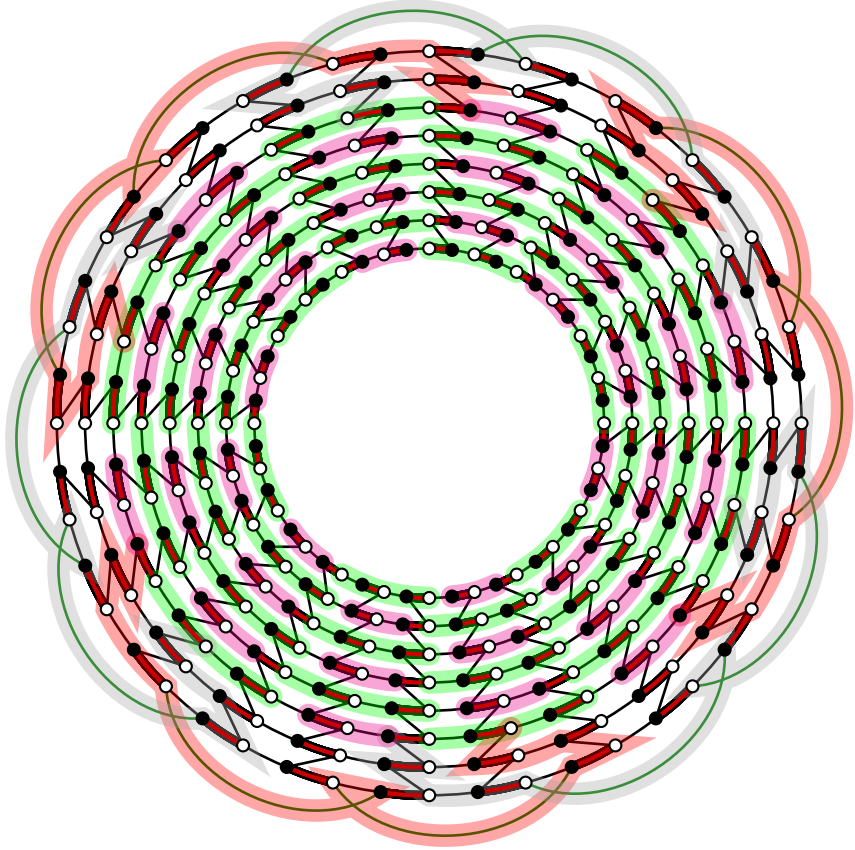

The main result of this subsection is the following.
\begin{theorem}\label{thm_svmmhardness}
If $\mathcal{G}$ is a graph class that is closed under matching minors and contains a shallow vortex matching minor, then $\#\textsc{PerfMatch}$ is \classP-complete on $\mathcal{G}$.  
% The problem $\#\textsc{PerfMatch}$ is $\classP$-hard in the graph class $\mathcal{V}$ of shallow vortex matching minors.
\end{theorem}

\begin{definition}
Let $G$ be a graph drawn on the plane in such a way that two of its edges, $e$ and $f$ cross. 
We denote by $G_{e,f}$ the graph obtained from $G$ by replacing the edges $e$ and
$f$ by the bipartite planar sign-crossing gadget as in Figure~\ref{fig_signcrossgadget}. Note that the planar sign-crossing gadget was defined in~\cite{CaiG14} 
and also used in the proof of~\cite{CurticapeanXia22}. However, since we are working with bipartite graphs we have to be careful in the way we apply the replacement. 
\end{definition}

We define the function $\chi_{e,f}(M)\coloneqq \begin{cases} -1 & \{e,f\}\subseteq M\\ 
1 & \text{otherwise}\end{cases}$

\begin{observation}\label{obs_bipsigncross}
If a bipartite graph $B$ is drawn on the plane in such a way that two of its edges, $f$ and $g,$ cross then
the graph $B_{f,g}$ obtained from the bipartite planar sign-crossing gadget replacement is also bipartite.
\end{observation}

\begin{figure}[!ht]
	\centering
		\makefast{ % !TEX root = ../single_comb.tex
% !TeX spellcheck = en_UK	

	\begin{tikzpicture}[scale=0.8]

		\pgfdeclarelayer{background}
		\pgfdeclarelayer{foreground}
		\pgfsetlayers{background,main,foreground}
	
		\begin{scope}[shift={(-3,0)}]
		\draw[e:main,color=myGreen,ultra thick] (-1.5,0) to (1.5,0);
		\draw[e:main,color=red,ultra thick] (0,-1.5) to (0,1.5);
		\node[v:main] () at (-1.5,0){};
		\node[v:mainempty] () at (1.5,0){};		
		\node[v:mainempty] () at (0,-1.5){};
		\node[v:main] () at (0,1.5){};		
		\end{scope}
		
		\begin{scope}[shift={(3,0)}]
		\draw[e:main,color=myGreen,ultra thick] (-1.5,0) to (-1,0);
		\draw[e:main,color=myGreen,ultra thick] (1.5,0) to (1,0);

		\draw[e:main,color=red,ultra thick] (0,-1.5) to (0,-1);
		\draw[e:main,color=red,ultra thick] (0,1.5) to (0,1);
		\draw[e:main] (1,0) to (0,1) to (-1,0) to (0,-1) to (1,0);
		\draw[e:main] (-0.5,0.5) to (0.5,-0.5);
		\node[v:main] () at (-1.5,0){};
		\node[v:mainempty] () at (1.5,0){};		
		\node[v:mainempty] () at (0,-1.5){};
		\node[v:main] () at (0,1.5){};	
		\node[v:mainempty] () at (-1,0){};
		\node[v:main] () at (1,0){};		
		\node[v:main] () at (0,-1){};
		\node[v:mainempty] () at (0,1){};
		\node[v:main] () at (-0.5,0.5){};
		\node[v:mainempty] () at (0.5,-0.5){};
		\end{scope}

		\begin{pgfonlayer}{background}
		\end{pgfonlayer}
		
	\end{tikzpicture}}{\scalebox{.45}{\includegraphics{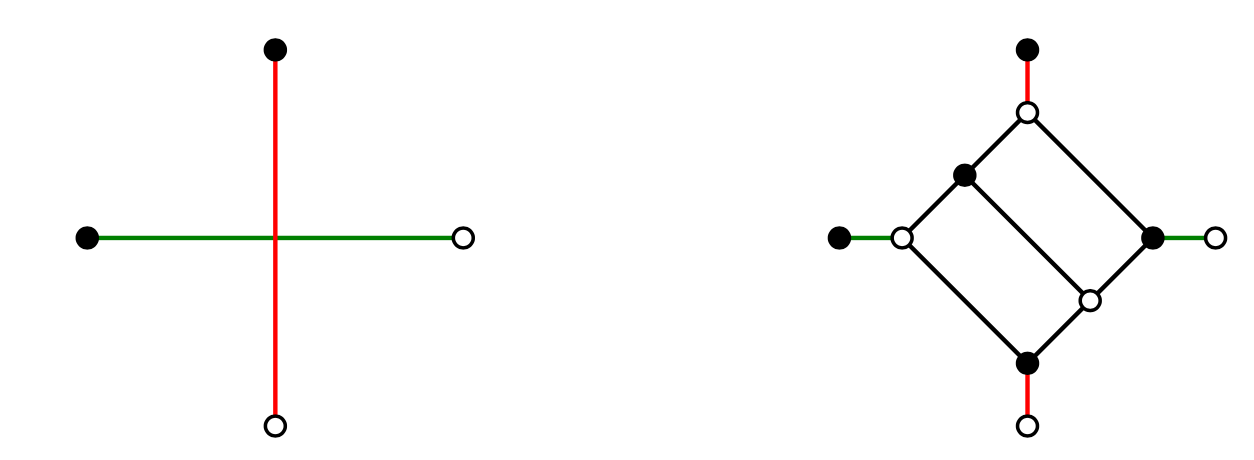}}}
	\caption{The edges $e$ and $f$ (on the left) and the bipartite sign crossing gadget (on the right). Observe that if we would not need the resulting graph to be bipartite we could have replaced it choosing any orientation of the hexagon. However, since we work on bipartite graphs we restrict the replacement to occur as in this figure.}
	\label{fig_signcrossgadget}
\end{figure}
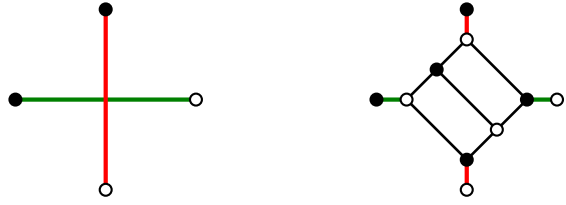

\begin{lemma}[\cite{CurticapeanXia22}]
Let $G$ be a graph drawn on the plane and let $(e_{1},f_{1}), (e_{2},f_{2}),\dots, (e_{t},f_{t})$ be $t$ pairs of crossing edges.
If $G'$ is the graph obtained by repeatedly replacing each one of these crossings with a bipartite planar sign-crossing gadget, then
$$\displaystyle \# \textsc{PerfMatch}(G')=\sum_{M\in \Perf{G}}\prod_{i=1}^{t}\chi_{e_{i},f_{i}}(M).$$
\end{lemma}

Before we proceed we also need the following reduction that was introduced in~\cite{CurticapeanXia22}. Curticapean and Xia originally introduced this reduction for the (ordinary) minor relation. We give the original definition of the construction as it can easily be adapted for bipartite graphs.

\paragraph{Reduction to shallow vortex matching minors} Let $G$ be an unweighted graph with $n$ vertices and $m$ edges. 
We construct the graph $G_{\mathrm{svmm}}$ as follows.

Let $C$ be a circle on the plane dividing it into two regions (we will assume that both regions contain $C$), the \emph{inner} 
and the \emph{outer} region.
\begin{enumerate}
\item We place the vertices $V(G)$ of $G$ on $C$ and draw the edges of $G$ as straight lines of the inner region of $C.$

\item\label{@deliberating} The edges are drawn in such a way that no three edges intersect in the same point and every ray from the centre to the perimeter of $C$ contains at most one crossing point or a vertex of $G.$

\item Let $p_{1},p_{2},\dots,p_{s}$ the locations of all $s\in \mathcal{O}(m^{2})$ crossings of $G$ and $P_{i}$ be the ray of $C$ that contains $p_{i}.$ Let also 
$e_{i},f_{i}$ denote the edges that intersect at $p_{i}.$
We denote by $\ell_{i}$ the segment of $P_{i}$ between $p_{i}$ and the perimeter of $C.$ From \ref{@deliberating} all segments are pairwise disjoint and do not
contain any vertices of $G.$ Observe, however, that each segment $\ell_{i}$ might intersect some edges of the graph, say $m_{i}$ of them. 
Namely, the edges $g_{i,1},g_{i,2},\dots,g_{i,m_{i}}.$

\item Let $N_{i}$ denote a sufficiently narrow neighbourhood of the segment $\ell_{i}.$
We bend the edges $e_{i}$ and $f_{i}$ and push them in such a way that their intersection is not contained in the inner region of $C.$
Notice that we introduce $2m_{i}$ new crossings with $e_{i}$ and $2m_{i}$ new crossing with $f_{i}.$
In particular, two crossings between $e_{i}$ and $g_{i,j}$ and two crossing between $f_{i}$ and $g_{i,j}$ for every $j\in m_{i}.$

\item After processing all crossing points $p_{1},p_{2},\dots,p_{s}$ in the above fashion, we replace the new crossings that were formed 
with the (bipartite) planar sign-crossing gadget. However, we do not replace any crosses on the outer part (that is, the ``original'' crosses that were ``pulled'' outside $C$). 
\end{enumerate}
Given a graph $G$ and a graph $G_{\mathrm{svmm}}$ constructed as above, we denote by $G_{\mathrm{svmm}}'$ the graph obtained from $G_{\mathrm{svmm}}$ by deleting all edges that participate in the remaining crosses of $G_{\mathrm{svmm}}.$
To prove the correctness of the reduction we need the following.

\begin{proposition}[\cite{giannopoulou2021excluding}]\label{thm_excludingplanar}
For every planar bipartite matching covered graph $H$ there exists a number
$\omega_{H} \in \mathbb{N}$ such that $H$ is a matching minor of $CG_{\omega_{H}}.$
Moreover, the models of the vertices of the outer face of $H$ will contain at least one vertex
from the outer face of $CG_{\omega_{H}}.$
\end{proposition}

\begin{lemma}
Given a bipartite graph $B,$ we may construct $B_{\mathrm{svmm}}$ in time $\mathcal{O}(n+m^{3})$ with the following properties:
\begin{enumerate}
\item its edge-weights will be $\pm 1$
\item all vertices of $B_{\mathrm{svmm}}$ are contained in a disk $\Delta$ of the plane
\item all edges of $B_{\mathrm{svmm}}$ intersect $D$ at most at their boundaries, 
\item all edges of $B_{\mathrm{svmm}}$ incident with weight -1 are on the inner region,
\item $\pmprob(B_{\mathrm{svmm}})=\pmprob(B),$ 
\item $B_{\mathrm{svmm}}$ is bipartite, and
\item $B_{\mathrm{svmm}}$ is a shallow vortex matching minor.
\end{enumerate}
\end{lemma}

\begin{proof}
The running time of the algorithm and Properties i) to v) follow directly by Lemma 3.1 from~\cite{CurticapeanXia22}. Property vi) follows from~\cref{obs_bipsigncross}.
Thus in the rest of the proof we will focus solely on the proof of Property vii). For this, we need to prove that there exists a constant
$\omega_{B_{\mathrm{svmm}}}$ such that $B_{\mathrm{svmm}}$ is a matching minor of $\svmg_{\omega_{B_{\mathrm{svmm}}}}.$
Observe that it follows from the reduction that $B_{\mathrm{svmm}}'$ is planar.
By \cref{thm_excludingplanar}, there exists a constant $\omega_{B_{\mathrm{svmm}}'}$ such that
$B_{\mathrm{svmm}}'$ is a matching minor of $CG_{\omega_{B_{\mathrm{svmm}}'}}.$ Moreover, by construction, all crosses 
of $B_{\mathrm{svmm}}$ attach to vertices of $B'_{\mathrm{svmm}}$ that belong to its outer face.
Observe then that by choosing a big enough constant $\omega$ that depends on $\omega_{B_{\mathrm{svmm}}'}$ we may find
that graph $B_{\mathrm{svmm}}'$ as a matching minor in $\svmg_{\omega}$ induced by its innermost cycles and utilise the outer
cycles together with the crossing edges of $\svmg_{\omega}$ to obtain the models of the crossing edges of $B_{\mathrm{svmm}}$
in a fashion similar to \cref{fig_gridvortexinsvg}.
\end{proof}

With this, everything is in place for the proof of \cref{thm_svmmhardness}.

\begin{proof}[Proof of \cref{thm_svmmhardness}]
First, from~\cite{Curticapean15,CurticapeanXia22} the problem \pmprob~on shallow vortex matching minors with weights 
$\pm 1,$ where the edges with weights $-1$ are entirely contained in $B_{\mathrm{svmm}}'$ admits a polynomial Turing 
reduction to unweighted shallow vortex matching minors.
Then the \classP-hardness of the problem follows from the fact that the problem is \classP-hard on unweighted 3-regular 
bipartite graphs~\cite{DagumL92,Valiant79theco},
the above Turing reduction and the construction of the graphs $B_{\mathrm{svmm}}.$
\end{proof}

\subsection{\classP-hardness on $K_{5,5}$-matching-minor-free graphs.}
\label{@unanswerable}

We now  prove that $K_{5,5}$ is not a shallow vortex matching minor. This, combined with~\cref{thm_svmmhardness}, 
implies that computing the number of perfect matchings on $K_{5,5}$-matching-minor-free graphs is \classP-hard.
We conclude the section by showing that $K_{4,4}$ is a shallow vortex matching minor.

Before we proceed we need the following definition.

Let $T'$ be a tree and let $T$ be obtained from $T'$ by subdividing every edge an odd number of times.
Then $\Fkt{V}{T'} \subseteq \Fkt{V}{T}$.
The vertices of $T$ that belong to $T'$ are called \emph{old}, and the vertices in $\Fkt{V}{T} \setminus \Fkt{V}{T'}$ are called \emph{new}.
We say that $T$ is a \emph{barycentric tree}.

\begin{definition}[Matching Minor Model]\label{def_matchingminormodel}
	Let $G$ and $H$ be graphs with perfect matchings.
	An \emph{embedding} or \emph{matching minor model} of $H$ in $G$ is a mapping 
	\begin{equation*}
		\mu \colon \Fkt{V}{H} \cup \Fkt{E}{H} \to \CondSet{F}{F\subseteq G},
	\end{equation*}
	such that the following requirements are met for all $v,v' \in \Fkt{V}{H}$ and $e,e' \in \Fkt{E}{H}$:
	\begin{enumerate}
		\item $\Fkt{\mu}{v}$ is a barycentric subtree in $G$ called the model of  $v$ in $G$,
		
		\item if $v \neq v'$, then $\Fkt{\mu}{v}$ and $\Fkt{\mu}{v'}$ are vertex disjoint,
		
		\item $\Fkt{\mu}{e}$ is an odd path with no internal vertex in any $\Fkt{\mu}{v}$ called the model of $e$ in $G$, and if $e' \neq e$, then $\Fkt{\mu}{e}$ and $\Fkt{\mu}{e'}$ are internally vertex disjoint,
		
		\item if $e=u_1u_2$, then the ends of $\Fkt{\mu}{e}$ can be labelled by $x_1,x_2$ such that $x_i$ is an old vertex of $\Fkt{\mu}{u_i}$,
		
		\item if $v$ has degree one, then $\Fkt{\mu}{v}$ is exactly one vertex, and
		
		\item $G-\Fkt{\mu}{H}$ has a perfect matching, where $\Fkt{\mu}{H'} \coloneqq \bigcup_{x\in \Fkt{V}{H'} \cup \Fkt{E}{H'}}\Fkt{\mu}{x}$ for every subgraph $H'$ of $H$.
	\end{enumerate}	
	If $\mu$ is a matching minor model of $H$ in $G$ we write $\mu\colon H\rightarrow G$.
\end{definition}

\begin{lemma}[\cite{norine2007generating}]\label{lemma_matmodel}
	Let $G$ and $H$ be graphs with perfect matchings.
	There exists a matching minor model $\mu\colon H\rightarrow G$ if and only if $H$ is isomorphic to a matching minor of $G$.
\end{lemma}

\begin{observation}
By enhancing each barycentric tree of \cref{def_matchingminormodel} we may assume that every model of an edge consists
of only one edge.
\end{observation}

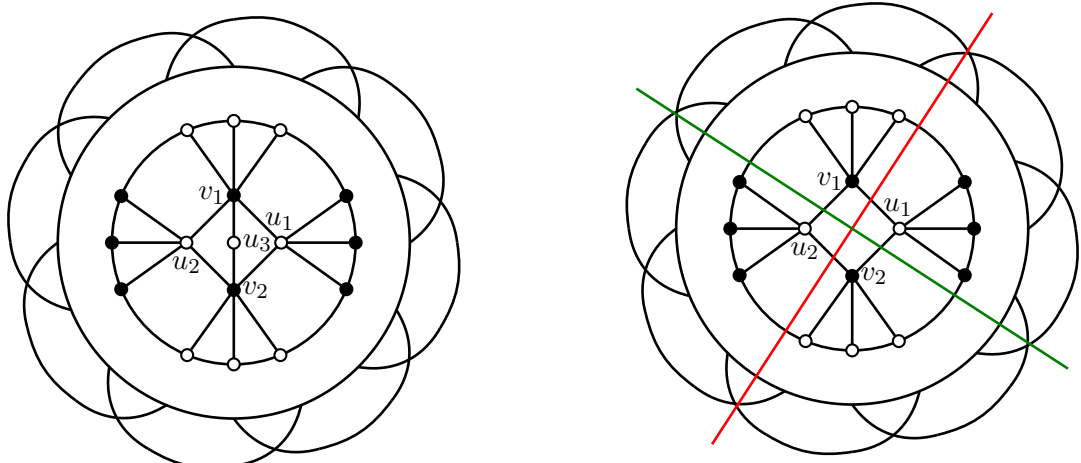
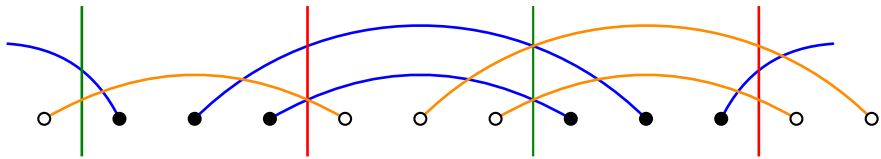
\begin{figure}[!ht]
\centering
% !TEX root = ../single_comb.tex
% !TeX spellcheck = en_US

	\begin{subfigure}{0.49\textwidth}
	\centering
		\begin{tikzpicture}[scale=0.9]

			\pgfdeclarelayer{background}
			\pgfdeclarelayer{foreground}
			\pgfsetlayers{background,main,foreground}

			\draw[e:main] (0,0) circle (18mm);
			\draw[e:main] (0,0) circle (26mm);
			
			\draw[e:main] (0:7mm) to (90:7mm) to (180:7mm) to (270:7mm) to (0:7mm);
			
			\foreach \x in {2,4,6,...,16}
			{
				\draw[e:main,bend right=35] (\x*22.5:26mm) to (\x*22.5+22.5:33mm); 
				\draw[e:main,bend right=20] (\x*22.5+22.5:33mm) to (\x*22.5+45:33mm); 
				\draw[e:main,bend right=35] (\x*22.5+45:33mm)  to (\x*22.5+67.5:26mm);
			}
			\foreach \x in {1,2}
			{	
			\draw[e:main ](\x*180+90:7mm) to (\x*180+90:18mm);
			\draw[e:main] (\x*180+90:7mm) to (\x*180+112.5:18mm);
			\draw[e:main] (\x*180+90:7mm) to (\x*180+67.5:18mm);
			}

			\foreach \x in {1,2}
			{	
			\draw[e:main ](\x*180:7mm) to (\x*180:18mm);
			\draw[e:main] (\x*180:7mm) to (\x*180+22.5:18mm);
			\draw[e:main] (\x*180:7mm) to (\x*180-22.5:18mm);
			\draw[e:main] (0:0mm) to (\x*180+90:7mm);
			}

			\foreach \x in {1,2}
			{
				\node[v:main] () at (\x*180+90:7mm){};				
				\node[v:mainempty] () at (\x*180:7mm){};
				\node[v:mainempty] () at (\x*180+90:18mm){};
				\node[v:mainempty] () at (\x*180+112.5:18mm){};
				\node[v:mainempty] () at (\x*180+67.5:18mm){};
				\node[v:main] () at (\x*180:18mm){};
				\node[v:main] () at (\x*180+22.5:18mm){};
				\node[v:main] () at (\x*180-22.5:18mm){};
			}
						\node () at (25:7.5mm){$u_{1}$};
			\node () at (205:7.7mm){$u_{2}$};
			\node () at (115:7.7mm){$v_{1}$};
			\node () at (295:7.7mm){$v_{2}$};
			\node () at (0.35,0){$u_{3}$};			
			\node[v:mainempty] () at (0:0mm){};		
		\end{tikzpicture}
	\caption{$\widehat{B}_{in}$ is not $K_{3,2}$}
\end{subfigure}
\begin{subfigure}{0.49\textwidth}
	\centering
		\begin{tikzpicture}[scale=0.9]

			\pgfdeclarelayer{background}
			\pgfdeclarelayer{foreground}
			\pgfsetlayers{background,main,foreground}

			\draw[e:main] (0,0) circle (18mm);
			\draw[e:main] (0,0) circle (26mm);
			
			\draw[e:main] (0:7mm) to (90:7mm) to (180:7mm) to (270:7mm) to (0:7mm);
			
			\foreach \x in {2,4,6,...,16}
			{
				\draw[e:main,bend right=35] (\x*22.5:26mm) to (\x*22.5+22.5:33mm); 
				\draw[e:main,bend right=20] (\x*22.5+22.5:33mm) to (\x*22.5+45:33mm); 
				\draw[e:main,bend right=35] (\x*22.5+45:33mm)  to (\x*22.5+67.5:26mm);
			}
			\foreach \x in {1,2}
			{	
			\draw[e:main ](\x*180+90:7mm) to (\x*180+90:18mm);
			\draw[e:main] (\x*180+90:7mm) to (\x*180+112.5:18mm);
			\draw[e:main] (\x*180+90:7mm) to (\x*180+67.5:18mm);
			}

			\foreach \x in {1,2}
			{	
			\draw[e:main ](\x*180:7mm) to (\x*180:18mm);
			\draw[e:main] (\x*180:7mm) to (\x*180+22.5:18mm);
			\draw[e:main] (\x*180:7mm) to (\x*180-22.5:18mm);
			}

			\draw[e:main,color=red] (237:38mm) to (57:38mm);		
			\draw[e:main,color=myGreen]  (327:38mm) to (147:38mm);	
			\foreach \x in {1,2}
			{
				\node[v:main] () at (\x*180+90:7mm){};				
				\node[v:mainempty] () at (\x*180:7mm){};
				\node[v:mainempty] () at (\x*180+90:18mm){};
				\node[v:mainempty] () at (\x*180+112.5:18mm){};
				\node[v:mainempty] () at (\x*180+67.5:18mm){};
				\node[v:main] () at (\x*180:18mm){};
				\node[v:main] () at (\x*180+22.5:18mm){};
				\node[v:main] () at (\x*180-22.5:18mm){};
			}			
			
			\node () at (25:7.5mm){$u_{1}$};
			\node () at (205:7.7mm){$u_{2}$};
			\node () at (115:7.7mm){$v_{1}$};
			\node () at (295:7.7mm){$v_{2}$};
		\end{tikzpicture}
		\caption{The inner part of the $K_{5,5}$ model discussed in the proof. }
\end{subfigure}

\begin{subfigure}{\textwidth}
\centering
\begin{tikzpicture}

\draw[e:main, bend left = 30,color=blue] (-2,0) to (2,0);
\draw[e:main, bend left = 45,color=blue] (-3,0) to (3,0);
\draw[e:main,bend right = 30,color=blue] (-4,0) to (-5.5,1);
\draw[e:main,bend left = 30,color=blue] (4,0) to (5.5,1);

\draw[e:main,color=red] (-1.5,-0.5) to (-1.5,1.5); 
\draw[e:main,color=myGreen,thick] (1.5,-0.5) to (1.5,1.5); 
\draw[e:main,color=red] (4.5,-0.5) to (4.5,1.5); 
\draw[e:main,color=myGreen] (-4.5,-0.5) to (-4.5,1.5);

\draw[e:main, bend left = 30,color=darkorange] (-5,0) to (-1,0);
\draw[e:main, bend left = 30,color=darkorange] (1,0) to (5,0);
\draw[e:main, bend left = 45,color=darkorange] (0,0) to (6,0);

\node[v:mainempty] () at (-5,0){};
\node[v:main] () at (-4,0){};
\node[v:main] () at (-3,0){};
\node[v:main] () at (-2,0){};
\node[v:mainempty] () at (-1,0){};
\node[v:mainempty] () at (0,0){};
\node[v:mainempty] () at (1,0){};
\node[v:main] () at (2,0){};
\node[v:main] () at (3,0){};
\node[v:main] () at (4,0){};
\node[v:mainempty] () at (5,0){};
\node[v:mainempty] () at (6,0){};

\end{tikzpicture}
\caption{We depict the $w_{1}^{i},w_{2}^{i},b_{1}^{i},b_{2}^{i}$ in a linear order together with the paths that join them in order to show that at some point over the red and the green edge they form a flow 4 over an area with a separator of size 3.}
\end{subfigure}
\caption{Showing that $K_{5,5}$ is not a shallow vortex matching minor. The paths $P_{i}$ are depicted as orange and the paths $Q_{i}$ as blue. We avoid naming all paths and vertices explicitly in order to make the figure clearer.}
\label{fig_k55exclusion}
\end{figure}

Given a bipartite graph $B$ we denote by $\mu(B)$ the maximum $t$ such that $K_{t,t}$ is a matching minor of $B.$

\begin{lemma} 
For every shallow vortex matching minor $B,$ $\mu(B)\leq 4.$
%exists a simple vortex $B'$ of order $k,$ for some $k,$ such that $B'$ is a matching minor of $B$ and $\mu(B)\leq \mu(B').$
\end{lemma}

\begin{proof}
Let $B$ be a shallow vortex matching minor and let $j\in \mathbb{N}$ be the smallest integer such that $B$ is a matching minor 
of $\svmg_{j}.$ Then $K_{t,t}$ is a matching minor of $\svmg_{j},$ where $t=\mu(B)$. Towards a contradiction, assume that $t\geq 5.$
Consider a matching minor model of $K_{t,t}$ in $\svmg_{j}.$

We first apply the following two operations. 
We start by working on the models of vertices of $K_{t,t}$ that contain vertices of the two outermost cycles. 
Then we apply bicontractions as much as possible so that their new models are contained in the two outermost 
cycles only, with the extra restriction that we do not remove the crossing edges that attach to them. 
Now, we bicontract all models of vertices that do not contain any vertex on the two outermost cycles and remove any vertices/edges that do not belong to the model, which by definition induces a conformal subgraph of $B.$ By construction the resulting 
graph, $\widehat{B},$ still contains $K_{t,t}$ as a matching minor. 

Observe that bicontractions preserve planarity as well as the bipartiteness of the graph. 
This implies that the graph $\widehat{B}_{in}$ induced  by the vertices not contained in the two outermost cycles is planar, bipartite and an induced subgraph of $K_{t,t}$.

Let us assume that $\widehat{B}_{in}$ contains at least 2 vertices of each part of the bipartition. Observe that only one of
the parts can have more than 3 vertices in $\widehat{B}_{in}$ as otherwise $\widehat{B}_{in}$ would contain $K_{3,3}$,
a contradiction to the face that it is planar.
Then, since it's planar it is either $K_{2,2}$ or $K_{2,\ell}$, $\ell \geq 3$.

%By construction, all vertices of $\widehat{B}_{in}$ are vertices of $K_{t,t}.$ (Recall also here, that the models of the vertices 
%do not contain a vertex from both $\widehat{B}_{in}$ and $\widehat{B}_{out},$ where $\widehat{B}_{out}$ is the graph that remains after the removal 
%of $\widehat{B}_{in}$).
We now show that $\widehat{B}_{in}$ is $K_{2,2}.$ Indeed, towards a contradiction, let us assume that $\widehat{B}_{in}$ is $K_{2,\ell}$
and let $u_{1},u_{2},u_{3}$ and $v_{1},v_{2}$ be vertices of its two parts (where $v_{1}$ and $v_{2}$ belong to the part
that has only two vertices in $\widehat{B}_{in}$), respectively. Then due to planarity, it holds that one of the $u_{i},$ say $u_{3}$ is contained in 
the cycle formed by $u_{1},v_{2},u_{2},$ and $v_{2}.$ This implies that $u_{3}$ can only have $v_{1}$ and $v_{2}$ as neighbours, a contradiction to the assumption that $u_{3}$ is a
vertex of $K_{t,t},$ with $t\geq 5.$ (See also (a) in Figure~\ref{fig_k55exclusion}.) Therefore, $\widehat{B}_{in}$ is the graph $K_{2,2}.$ We abuse
notation and denote by $C_{3}$ the 4-cycle induced by $u_{1},v_{1},u_{2},v_{2}$  (as in, the third cycle of the shallow vortex matching grid).

Since $t\geq 5,$ $K_{5,5}$ is also a matching minor of $\widehat{B}$ and, in particular, there exists a model of it such that $u_{1},u_{2}$ and $v_{1},v_{2}$ are vertices of its
two parts and the models of the rest vertices, say $u_{3},u_{4},u_{5}$ and $v_{3},v_{4},v_{5},$ are entirely contained in $\widehat{B}_{out}.$
Moreover, $u_{3},u_{4},$ and $u_{5}$ are joined by edges to $v_{1}$ and $v_{2}.$
This implies that there exist vertices $w_{1}^{i}$ and $w_{2}^{i}$ in the model 
of $u_{i}$ which belong to the same part of $\widehat{B}$ as $u_{i}$ and such that $\{w_{1}^{i},v_{1}\},\{w_{2}^{i},v_{2}\}\in E(\widehat{B}).$
Similarly, the vertices $v_{3},v_{4},v_{5}$ are joined by edges to $u_{1}$ and $u_{2}.$ Therefore, as above, there exist vertices $b_{1}^{i}$ and $b_{2}^{i}$ in the model 
of $v_{i}$ which belong to the same part of $\widehat{B}$ as $v_{i}$ and such that $\{b_{1}^{i},u_{1}\},\{b_{2}^{i},u_{2}\}\in E(\widehat{B}).$ (See (b) in Figure~\ref{fig_k55exclusion}.)
This implies that there exist paths $P_{i}$ between $w_{1}^{i},w_{2}^{i}$ (that are contained in the model of $u_{i}$), $i\in [3],$ and 
$Q_{i}$ between $b_{1}^{i},b_{2}^{i}$ (that are contained in the model of $v_{i}$), $i\in [3],$ respectively. Moreover, the paths $P_{i}$ and $Q_{i}$ are entirely disjoint, are contained
in $\widehat{B}_{out}.$ Observe that four of these paths (two of the $P_{i}$ and two of the $Q_{i}$) go over the red and the green line in (b) and (c) in Figure~\ref{fig_k55exclusion}, a contradiction as there exists a separator of order at most 3.
Therefore, it is not possible that each part of $K_{t,t}$ has at least two vertices in $\widehat{B}_{in}$.

Observe then that $\widehat{B}_{in}$ is either isomorphic to $K_{1,\ell}$ for some $1\leq \ell \leq t$, or $\widehat{B}_{in}$ is an independent set of size $\ell$ for some $\ell \leq t$. With similar arguments about the existence of many disjoint paths over small separators we may exclude the above cases as well. 
Let us assume that $\ell \geq 3$ and let $u_{1},u_{2},u_{3}$ be the vertices of the largest part in $\widehat{B}_{in}$. Let us distinguish the following two cases.

\textbf{Case 1.} The other part has no vertex in $\widehat{B}_{in}$. Notice that each of the $u_{i}$ has four neighbours in $C_{2}$, namely $b_{1}^{j}$, $j\in [5]$, $b_{2}^{j}$, $j\in [5]$, and $b_{3}^{j}$, $j\in [5]$ respectively. 
Moreover, for every $i\in [3]$, there exists a subpath of $C_{2}$ that contains the vertices $b_{i}^{j}$ and does not contain any vertices $b_{i'}^{j'}$ for $i\neq i'$. Similar to the previous case, we may conclude that in this case there exist four disjoint paths going through a separator of size 3, a contradiction.
Therefore, we may exclude this case.

\textbf{Case 2.} The other part has exactly one vertex in $\widehat{B}_{in}$.
Observe then that, as above each of the $u_{i}$ has four neighbours in $C_{2}$, namely $b_{1}^{j}$, $j\in [4]$, $b_{2}^{j}$, $j\in [4]$, and $b_{3}^{j}$, $j\in [4]$ respectively.  Here, we may use similar arguments while taking into account how many vertices of the 
big part are actually contained in $\widehat{B}_{in}$.

Finally, we may conclude that $\ell\leq 2$. But these cases can be similarly excluded.
\end{proof}

It follows that
\begin{corollary}\label{cor_k55notsvg}
$K_{5,5}$ is not a shallow vortex matching minor.
\end{corollary}

\begin{figure}[!ht]
		\centering
				\makefast{ % !TEX root = ../single_comb.tex
% !TeX spellcheck = en_UK	

		\begin{tikzpicture}[scale=0.8]

			\pgfdeclarelayer{background}
			\pgfdeclarelayer{foreground}
			\pgfsetlayers{background,main,foreground}

			\draw[e:main] (0,0) circle (11mm);
			\draw[e:main] (0,0) circle (16mm);
			\draw[e:main] (0,0) circle (21mm);
			\draw[e:main] (0,0) circle (26mm);
			
			\foreach \x in {1,...,16}
			{
				\draw[e:main] (\x*22.5:11mm) -- (\x*22.5:16mm);
				\draw[e:main] (\x*22.5:16mm) -- (\x*22.5:21mm);
				\draw[e:main] (\x*22.5:21mm) -- (\x*22.5:26mm);
			}
			
			\foreach \x in {2,4,6,...,16}
			{
				\draw[e:main,bend right=35] (\x*22.5:26mm) to (\x*22.5+22.5:33mm); 
				\draw[e:main,bend right=20] (\x*22.5+22.5:33mm) to (\x*22.5+45:33mm); 
				\draw[e:main,bend right=35] (\x*22.5+45:33mm)  to (\x*22.5+67.5:26mm);
			}

			\node[] () at (125:38mm){$v_{1}$};
			\draw[e:marker] (157.5:26mm) to [bend left=35]  (135:33mm) to [bend left=20] (112.5:33mm) to [bend left=35] (90:26mm) to [bend left=13]  (67.5:26mm);
			
			\node[] () at (305:38mm){$v_{2}$};
			\draw[e:marker] (337.5:26mm) to [bend left=35] (315:33mm) to [bend left=20] (292.5:33mm) to [bend left=35] (270:26mm) to [bend left=13] (247.5:26mm);	

			\node[] () at (112.5:29mm){$v_{3}$};
			\draw[e:marker] (112.5:26mm) to (112.5:21mm) to [bend left=30] (45:21mm);

			\node[] () at (-67.5:29mm){$v_{4}$};
			\draw[e:marker] (22.5:16mm) to [bend left=23] (-22.5:16mm) to (-22.5:21mm) to [bend left=23] (-67.5:21mm) to (-67.5:26mm);

			\node[] () at (180:38mm){$u_{1}$};
			\draw[e:marker,color=green] (225:26mm) to [bend left=13] (202.5:26mm) to [bend left=35] (180:33mm) to [bend left=20] (157.5:33mm) to [bend left=35] (135:26mm);
			
			\node[] () at (0:38mm){$u_{2}$}; 						
			\draw[e:marker,color=green] (45:26mm) to [bend left=13] (22.5:26mm) to [bend left=35] (0:33mm) to [bend left=20] (337.5:33mm) to [bend left=35] (315:26mm);

			\node[] at (0:29.5mm){$u_{3}$};			
			\draw[e:marker,color=green,bend left=13] (22.5:21mm) to [bend left=13] (0:21mm) to (0:26mm);
			
			\node[] at (180:29.5mm){$u_{4}$};
			\draw[e:marker,color=green] (180:26mm) to (180:21mm) to [bend left=23] (135:21mm) to (135:16mm) to [bend left=45] (45:16mm);

			%%%%% v1-u1
			\draw[e:marker,color=magenta,bend right=13] (135:26mm) to (157.5:26mm);
			%%%%% v1-u2
			\draw[e:marker,color=magenta,bend right=13] (45:26mm) to (67.5:26mm);
			%%%%% v1-u3
			\draw[e:marker,magenta] (0:26mm) to [bend right=30] (22.5:33mm) to [bend right=20] (45:33mm) to [bend right=30] (67.5:26mm);
			%%%%% v1-u4 
			\draw[e:marker,color=magenta,bend right=13] (157.5:26mm) to (180:26mm);

			%%%%%  v2 - u1
			\draw[e:marker,magenta,bend right=13] (225:26mm) to (247.5:26mm);
			%%%%% v2 - u2
			\draw[e:marker,magenta,bend left=13] (-22.5:26mm) to (-45:26mm);
			%%%% v2 - u3
			\draw[e:marker,magenta,bend left=13] (0:26mm) to (-22.5:26mm);
			%%%% v2 - u4
			\draw[e:marker,magenta] (180:26mm) to [bend right=30] (202.5:33mm) to [bend right=20] (225:33mm) to [bend right=30] (247.5:26mm);

			%%%%% v3-u1
			\draw[e:marker,color=magenta,bend right=13] (112.5:26mm) to (135:26mm);			
			%%%%% v3 - u2
			\draw[e:marker,magenta] (45:26mm) to [bend right=30] (67.5:33mm) to [bend right=20] (90:33mm) to [bend right=30] (112.5:26mm);
			%%%%% v3 - u3
			\draw[e:marker,color=magenta,bend right=13] (22.5:21mm) to (45:21mm);			
			%%%%% v3 - u4
			\draw[e:marker,magenta] (45:21mm) to (45:16mm);

			%%%%% v4 - u1
			\draw[e:marker,magenta] (225:26mm) to [bend right=30] (247.5:33mm) to [bend right=20] (270:33mm) to [bend right=30] (292.5:26mm);
			%%%%% v4 - u2
			\draw[e:marker,color=magenta] (-45:21mm) to (-45:26mm);
			%%%%% v4 - u3
			\draw[e:marker,color=magenta] (22.5:16mm) to (22.5:21mm);			
			%%%%% v4 - u4
			\draw[e:marker,color=magenta] (22.5:16mm) to (45:16mm);

			\foreach \x in {1,...,8}
			{
				\draw[e:coloredthin,color=BostonUniversityRed,bend right=13] (\x*45+22.5:11mm) to (\x*45+45:11mm);
			}
			\foreach \x in {3,4,5,6}
			{
				\draw[e:coloredthin,color=BostonUniversityRed,bend right=13] (\x*45+22.5:16mm) to (\x*45+45:16mm);
			}
			\foreach \x in {4,5}
			{
				\draw[e:coloredthin,color=BostonUniversityRed,bend right=13] (\x*45+22.5:21mm) to (\x*45+45:21mm);
			}
			
			\foreach \x in {1,...,8}
			{
				\node[v:main] () at (\x*45:11mm){};
				\node[v:mainempty] () at (\x*45:16mm){};
				\node[v:main] () at (\x*45:21mm){};
				\node[v:mainempty] () at (\x*45:26mm){};
				\node[v:mainempty] () at (\x*45+22.5:11mm){};
				\node[v:main] () at (\x*45+22.5:16mm){};
				\node[v:mainempty] () at (\x*45+22.5:21mm){};
				\node[v:main] () at (\x*45+22.5:26mm){};
			}

			\begin{pgfonlayer}{background}
				\foreach \x in {1,...,8}
				{
					\draw[e:coloredborder,bend right=13] (\x*45+22.5:11mm) to (\x*45+45:11mm);
				}
				\foreach \x in {3,4,5,6}
				{
					\draw[e:coloredborder,bend right=13] (\x*45+22.5:16mm) to (\x*45+45:16mm);
				}
				\foreach \x in {4,5}
				{
					\draw[e:coloredborder,bend right=13] (\x*45+22.5:21mm) to (\x*45+45:21mm);

				}
			\end{pgfonlayer}
			
		\end{tikzpicture}}{\scalebox{.55}{\includegraphics{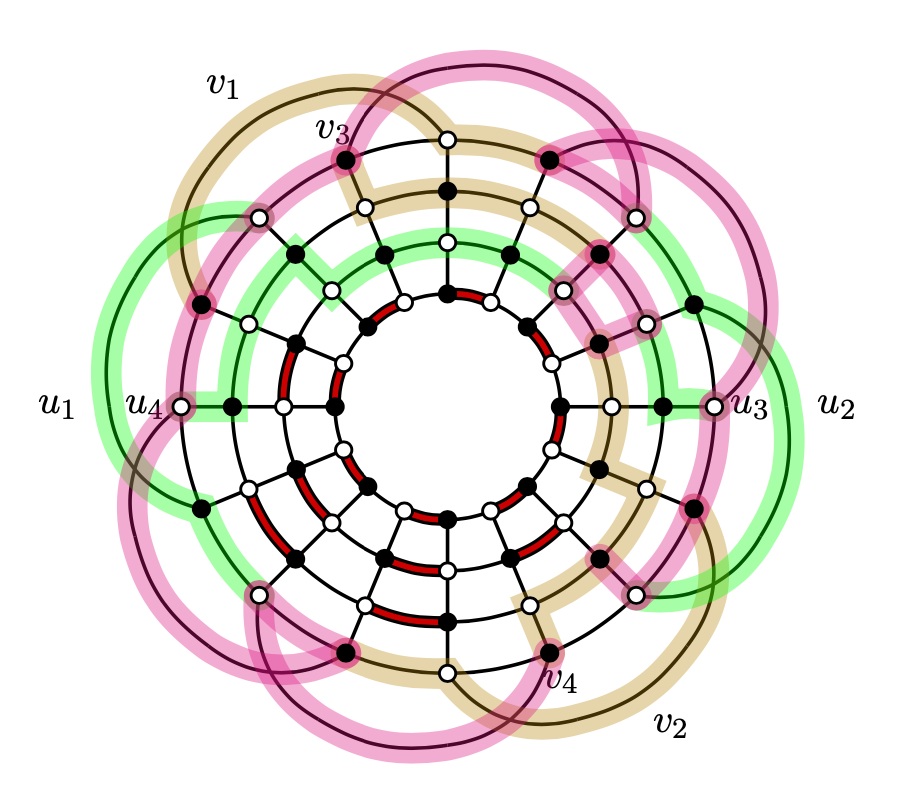}}}
		
	\caption{A model of $K_{4,4}$ in $RV_{4}.$ We depict the vertices of the two parts of the graph with green and yellow markers respectively. Each edge of $K_{4,4}$ is represented by a single edge of the host graph. Thus, we only apply bicontractions to obtain the vertices of $K_{4,4}.$ The edges between vertex models are depicted in pink. The red matching edges certify that the model is a conformal subgraph of the host graph.}
	\label{fig_k44gridmatminor}
\end{figure}
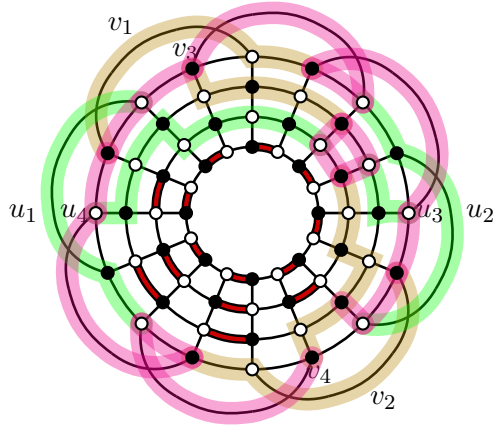

In contrast to \cref{cor_k55notsvg} we show the following. 

\begin{lemma}
$K_{4,4}$ is a shallow vortex matching minor.
\end{lemma}

\begin{proof}
Notice that it is enough to show that $K_{4,4}$ is a matching minor of $RV_{k}$ for some positive integer $k.$
For the proof of this lemma, we depict a matching minor model of $K_{4,4}$ in $RV_{4}$ in~\cref{fig_k44gridmatminor}. For a clearer depiction, we only draw the 4 outer circles of $RV_{4}$ as they are enough for the identification of the $K_{4,4}$ model.
\end{proof}

Combining \cref{cor_k55notsvg} and \cref{thm_svmmhardness} we obtain the main result of this section.

\begin{theorem}\label{thm_k55hardness}
The problem $\#\textsc{PerfMatch}$ is $\classP$-hard in the class of graphs excluding $K_{5,5}$ as a matching minor.
\end{theorem}

%%%%%%%%%%%%%%%%%%%%%%%%%%%%%%%%%%%%%%%%%%%%%%%%%%%%%%%%%%%%%%%%%%%%%%%%%%%%%%%%%%%%%%%%%%%%%%%%%%%%%%%%%%%%%%%%%%%%%%%%%%%%%%%%%%%%%%%%%%%%%%%%%%%%%%%%%%%%%%%%%%%%%%%%%%%%%%%%%%%%%%%%%%%%%%%%%%%%%%%%%%%%%%%%%%%%%%%%%%%%%%%%%%%%%%%%%%%%%%%%%%%%%%%%%%%%%%%%%%%%%%%%%%%%%%%%%%%%%%%%%%%%%%%%

%%%%%%%%%%%%%%%%%%%%%%%%%%%%%%%%%%%%%%%%%%%%%%%%%%%%%%%%%%%%%%%%%%%%%%%%%%%%%%%%%%%%%%%%%%%%%%%%%%%%%%%%%%%%%%%%%%%%%%%%%%%%%%%%%%%%%%%%%%%%%%%%%%%%%%%%%%%%%%%%%%%%%%%%%%%%%%%%%%%%%%%%%%%%%%%%%%%%%%%%%%%%%%%%%%%%%%%%%%%%%%%%%%%%%%%%%%%%%%%%%%%%%%%%%%%%%%%%%%%%%%%%%%%%%%%%%%%%%%%%%%%%%%%%

 \section{Conclusion}
 \label{@strengthening}
 
In this paper we introduced two types of grid-like bipartite graphs: \textsl{single-crossing matching grids} and  \textsl{shallow vortex matching grids}.
The exclusion of these graphs as matching minors determine the conditions of the two main results of this paper, that are \cref{thm_algomainthm} and \cref{thm_svmmhardness}.
Clearly, every single-crossing matching grid is a minor of a shallow vortex matching grid, but the other direction is not true (for instance, it is easy to see that $K_{4,4}$ is not a matching minor of a single-crossing matching grid, while $K_{4,4}$ is a matching minor of the shallow vortex grid -- see \cref{fig_k44gridmatminor}).
\medskip

As stated in \cref{thm_algomainthm}, there exists a function $h_1:\mathbb{N}\to\mathbb{N}$ such that  the problem of computing the permanent of a complete class $\mathcal{A}$ of square $(0,1)$-matrices, can be solved by an algorithm running in time $|A|^{O(f_1(\mathcal{A}))},$ where $f_1(\mathcal{A})$ is the maximum size of a single-crossing matching grid that is a matching minor of $B(A)$ for any $A\in\mathcal{A},$ and $|A|=|V(B(A))|.$
On the other side, our negative result, \cref{thm_svmmhardness}, asserts that such a result may not be expected for a hereditary class $\mathcal{B}$ of square $(0,1)$-matrices for which the function $f_2,$ measuring the maximum size of a shallow vortex matching grid as a matching minor of $B(A)$ among all $A\in\mathcal{B},$ is unbounded.
In the special case where we exclude $K_{t,t},$ $t\geq 0$ from $B(A)$ as a matching minor, we have that, for $t≤3,$  \cref{thm_algomainthm} yields the polynomial algorithm  of \cite{robertson1999permanents,mccuaig2004polya}, while  \cref{thm_svmmhardness} implies $\classP$-hardness for $t≥5.$
Clearly, there is a gap between these two results that needs to be filled (in \cref{fig_classes}, this gap is visualised by the white space above the black line and below the orange line).
A good representative of the general question is the following:
\begin{quotation}
	\noindent \textsl{What is the complexity of computing the permanent of square $(0,1)$-matrices $A,$ where $B(A)$ excludes $K_{4,4}$ as a matching minor?}
\end{quotation}
Recall that $K_{4,4}$ is a matching minor of a shallow vortex matching grid.
We believe that the above question can be resolved positively, that is, there should exist a polynomial-time algorithm for computing the permanent of biadjacency matrices of $K_{4,4}$ matching-minor-free bipartite graphs.
In general, we conjecture that \cref{thm_algomainthm} also holds if we replace $f_{1}$ by $f_{2}.$
Such a result would completely resolve the permanent problem for hereditary classes of square $(0,1)$-matrices as it would indicate that $f_{2}$ (that is the exclusion of shallow vortex matching grids as matching minors) \textsl{precisely} defines the frontier between polynomial time and $\classP$-hardness.
 
A reason for believing the above conjecture is correct is motivated by a recent result on the complexity of counting perfect matchings in $H$-minor-free graphs.
In \cite{ThilikosW22killi1} the class of \textsl{shallow vortex grids} has been introduced, and it was proven that one can count the perfect matchings of graphs from some proper minor closed class $\mathcal{C}$ in polynomial time if $\mathcal{C}$ excludes some shallow vortex grid.
Moreover, in case $\mathcal{C}$ does not, then counting perfect matchings is $\classP$-hard even when restricted to $\mathcal{C}.$

The \textsl{shallow vortex matching grids} introduced in this paper act as matching theoretic counterparts of these shallow vortex grids.
Indeed, with \cref{thm_svmmhardness}, a matching theoretic counterpart to the lower bound of the dichotomy of \cite{ThilikosW22killi1}, based on the result of Curticapean and Xia \cite{CurticapeanXia22}, has been introduced.
Hence, the missing part for the resolution of the above conjecture is to give a polynomial time algorithm for computing the permanent for hereditary classes $\mathcal{A}$ of square $(0,1)$-matrices where the corresponding bipartite graphs $B(\mathcal{A})$ exclude a shallow vortex grid as a matching minor.
A possible approach for this would be to adopt the approach of \cite{ThilikosW22killi1} and try to transfer the ``vortex killing'' technology to the matching minors setting.
This appears to be a highly non-trivial task as it would require the development of a general structure theorem for bipartite graphs excluding a fixed matching minor.
However, towards the proof of such a matching theoretic counterpart of the GMST, the results of the present paper, as well as the results of Giannopoulou, Kreutzer, and Wiederrecht  \cite{giannopoulou2021two, giannopoulou2021excluding} already provide several necessary building blocks.
 \medskip

%%%%%%%%%%%%%%%%%%%%%%%%%%%%%%%%%%%%%%%%%%%%%%%%%%%%%%%%%%%%%%%%%%%%%%%%%%%%%%%%%%%%%%%%%%%%%%%%%%%%%%%%%%%%%%%%%%%%%%%%%%%%%%%%%%%%%%%%%%%%%%%%%%%%%%%%%%%%%%%%%%%%%%%%%%%%%%%%%%%%%%%%%%%%%%%%%%%%%%%%%%%%%%%%%%%%%%%%%%%%%%%%%%%%%%%%%%%%%%%%%%%%%%%%%%%%%%%%%%%%%%%%%%%%%%%%%%%%%%%%%%%%%%%%
%
%
%

\newcommand{\etalchar}[1]{$^{#1}$}

%\bibliographystyle{alphaurl}
%\bibliography{literature_singlecrossing}

\end{document}